\documentclass{amsart}
\usepackage{euscript}           
\usepackage{amsmath,amsthm}     
\usepackage{amssymb}            

\usepackage[frame,matrix,curve, arrow]{xy}

\DeclareMathOperator{\Map}{Map}
\DeclareMathOperator{\Nat}{Nat}
\newcommand{\bfn}{{\mathbf n}}
\newcommand{\bs}{\blacksquare}
\newcommand{\bu}{\bullet}
\renewcommand{\P}{{\mathcal P}}
\newcommand{\bfe}{{\mathbf e}}
\newcommand{\sE}{{\mathcal E}}
\newcommand{\bA}{{\mathbf A}}
\newcommand{\bB}{{\mathbf B}}
\newcommand{\bL}{{{\mathbf L}_{\mathrm{sym}}}}
\newcommand{\bP}{{\mathbf P}}
\newcommand{\bQ}{{\mathbf Q}}
\newcommand{\bR}{{\mathbf R}}

\newcommand{\bS}{{\mathbf S}}
\newcommand{\bM}{{\mathbf M}}
\newcommand{\bMc}{{\mathbf M}^{\mathrm{comm}}}
\newcommand{\bX}{{\mathbf X}}
\newcommand{\bY}{{\mathbf Y}}
\newcommand{\bW}{{\mathbf W}}
\newcommand{\bZ}{{\mathbf Z}}
\renewcommand{\S}{{\mathcal S}}
\newcommand{\fS}{{\mathfrak S}}
\newcommand{\bbO}{{\mathbb O}}
\newcommand{\bbA}{{\mathbb A}}
\newcommand{\bbP}{{\mathbb P}}
\newcommand{\bbC}{{\mathbb C}}
\newcommand{\bfd}{{\mathbf d}}
\newcommand{\A}{{\mathcal A}}
\newcommand{\Ar}{{\A_{\text{rel}}^\Z}}
\newcommand{\Ag}{{\A_{e,*,1}}}
\newcommand{\Cell}{{\mathcal C}ell}
\newcommand{\cl}{\mathrm{cl}}
\newcommand{\id}{{\mathrm{id}}}

\renewcommand{\O}{{\cal O}}
\newcommand{\oS}{{\overline{S}}}
\newcommand{\boS}{{\overline{\mathbf{S}}}}
\newcommand{\owedge}{\mathbin{{\overline{\wedge}}}}
\newcommand{\ssk}{{{ss}} \S_k}
\newcommand{\D}{{\mathcal C}h_{hf}}
\newcommand{\Set}{\S}

\renewcommand{\O}{{\mathcal O}}

\renewcommand{\ss}{\Sigma{{ss}}\S } 
\newcommand{\ad}{{\mathrm{ad}}}
\newcommand{\pre}{{\mathrm{pre}}}
\newcommand{\inj}{{\mathrm{inj}}}
\newcommand{\Sig}{{\mathrm{sig}}}
\newcommand{\Inv}{{\mathrm{Inv}}}
\newcommand{\Sigr}{{\mathrm{sig}_{\mathrm{rel}}}}

\newcommand{\STop}{{\mathrm{STop}}}

\newcommand{\STopFun}{{\mathrm{STopFun}}}

\newcommand{\op}{\bar\omega}
\newcommand{\pp}{\bar\psi}
\usepackage {epsf}
\DeclareMathOperator*{\colim}{colim}

\newcommand{\bx}{\mathbin{\square}}
\newcommand{\sch}{\mathrm{sch}}
\newcommand{\Rel}{\mathrm{Rel,sch}}

\newcommand {\R}{{\mathcal R}}

\newcommand {\T}{{\mathcal T}}
\newcommand{\Msym} {{\bM}_{\mathrm{sym}}}
\newcommand{\Mgeom} {{\bM}_{\mathrm{geom}}}

\newcommand {\Z}{{\mathbb Z}}

\newcommand {\C}{{\cat C}}
\newcommand {\B}{{\cat B}}

\newcommand{\ra}{\rightarrow}                   
\newcommand{\lra}{\longrightarrow}              

\newcommand{\sss}{\S p_{{mss }}}
\newcommand{\cat}{\mathcal}    

\newfont{\german}       {eufm10 at 12pt}
\DeclareMathOperator{\Hom}{Hom} 

\numberwithin{equation}{section}
\newtheorem{thm}{Theorem}[section]
\newcounter{numerierer}
\newcounter{leer}

\newtheorem{defn}[thm]{Definition}

\newtheorem{prop}[thm]{Proposition}

\newtheorem{cor}[thm]{Corollary}
\newtheorem{lemma}[thm]{Lemma}

\theoremstyle{definition}  
\newenvironment{definition}{\begin{defn}\rm}{\end{defn}}

\newtheorem{example}[thm]{Example}

\newtheorem{notation}[thm]{Notation}

\newtheorem{remark}[thm]{Remark}

\begin{document}

\title{Commutativity properties of Quinn spectra}
\author{Gerd Laures}
\address{ Fakult\"at f\"ur Mathematik,  Ruhr-Universit\"at Bochum, NA1/66,
  D-44780 Bochum, Germany}
\author {James E.\ {McClure}}
\address{
Department of Mathematics, Purdue University, 150 N.\ University
Street, West Lafayette, IN 47907-2067}

\begin{abstract}
We give a simple sufficient condition for Quinn's ``bordism-type'' spectra to
be weakly equivalent to commutative symmetric ring spectra.  We also
show that the symmetric signature is (up to weak equivalence) a monoidal
transformation between symmetric monoidal functors, which implies that the
Sullivan-Ranicki orientation of topological bundles is represented by a ring 
map between commutative symmetric ring spectra.  In the course of proving these
statements we give a new description of symmetric L theory which may be of
independent interest.
\end{abstract}
\maketitle
\bigskip

\section{Introduction}
In \cite{MR1388303}, Frank Quinn gave a general machine for constructing
spectra from ``bordism-type theories.''  In our paper \cite{LM12} we gave
axioms for a structure we call an {\it ad theory} and showed that when these
axioms are satisfied (as they are for all of the standard examples) the Quinn
machine can be improved to give a symmetric spectrum $\bM$.  We also showed 
that when
the ad theory is multiplicative (that is, when its ``target category'' is
graded monoidal) the symmetric spectrum $\bM$ is a symmetric ring
spectrum.  Finally, we showed that there are monoidal functors to the category
of symmetric spectra which represent
Poincar\'e bordism over $B\pi$ (considered as a functor of $\pi$) and 
symmetric L-theory (considered as a functor of a ring $R$ with involution).

In this paper we consider commutativity properties.  It relies on the conventions and results of the paper \cite{LM12}; the relevant sections are
3, 6, 7, 9, 10, 17, 18 and 19.

A ``commutative ad theory'' is (essentially) an ad theory whose target 
category is graded symmetric monoidal (the precise definition is given in
Section \ref{comsec}).  Our first main result is

\begin{thm} 
\label{Nov12.3} 
Let $\bM$ be the symmetric ring spectrum associated to a commutative ad theory. 
There is a commutative symmetric ring spectrum $\bMc$ which is 
weakly equivalent in the category of symmetric ring spectra to $\bM$ and depends on it in a natural way.
\end{thm}

More precisely, we construct a symmetric spectrum $\bM$ together with an action of a specific $E_\infty$-operad which naturally depends on the underlying ad theory.  Then we use  the bar construction to obtain a zig-zag of weak equivalences to  the commutative symmetric ring spectrum $\bMc$. We only claim the naturality for strict morphisms. A refined statement of naturally and what happens with 2-morphisms under these functors  is worth being further investigated but it is not part of the paper. 

As a consequence, Theorem \ref{Nov12.3} shows that the L-theory spectrum of a 
commutative ring can be realized as a commutative symmetric ring 
spectrum.\footnote{Lurie \cite{Lurie} has explained another way to prove that 
the L-theory spectrum of a commutative ring can be realized as a commutative 
symmetric ring spectrum.  The method used in \cite{Lurie} does not include 
other examples we consider such as Poincar\'e bordism and (in \cite{BLM12}) 
Witt and IP bordism.}

For the ad theory $\ad_\STop$ of oriented topological bordism (\cite[Section 
6]{LM12}), we showed in \cite[Section 17 and Appendix B]{LM12} that the
underlying spectrum of 
$\bM_\STop$ is weakly equivalent to the usual Thom spectrum $M\STop$. 
It is well-known that $M\STop$ is a commutative symmetric ring spectrum, and 
we have

\begin{thm}
\label{m96}
$(\bM_\STop)^{\mathrm{comm}}$ and $M\STop$ are weakly equivalent in the
category of commutative symmetric ring spectra.
\end{thm}

The proof gives a specific chain of weak equivalences between them.

We also prove a multiplicative property of the symmetric signature.  The 
symmetric signature is a basic tool in surgery theory.  In its simplest form, 
it assigns to an oriented Poincar\'e complex $X$ an element of the symmetric 
L-theory of $\pi_1(X)$; this element determines the surgery obstruction up to 
8-torsion.  Ranicki proved that the symmetric signature of a Cartesian 
product is the product of the symmetric signatures (\cite[Proposition 
8.1(i)]{MR566491}).  The symmetric signature gives a map of spectra from 
Poincar\'e bordism to L-theory (\cite[Proposition 7.10]{KMM}), and we showed 
in \cite{LM12} that it gives a map of symmetric spectra.  In order to 
investigate the multiplicativity of this
map, we give a new (but equivalent) description of the L-spectrum, using
``relaxed'' algebraic Poincar\'e complexes (the relation between these and 
the usual algebraic Poincar\'e complexes is similar to the relation between
$\Gamma$-spaces and $E_\infty$ spaces).  For a ring with involution $R$ there 
is an ad theory $\ad_{\mathrm{rel}}^R$, and the associated spectrum
$\bM_{\mathrm{rel}}^R$
is equivalent to the usual L-spectrum.  The symmetric signature gives a map
$\Sigr$ from the Poincar\'e bordism spectrum (which we denote by $\bM_{e,*,1}$; 
see \cite[Section 7]{LM12}) to $\bM_{\mathrm{rel}}^R$.  We 
prove that this map is weakly 
equivalent to a ring
map between
commutative symmetric ring spectra:

\begin{thm}
\label{m54}
There are
symmetric ring spectra 
$\mathbf A$ and $\mathbf B$, 
commutative symmetric ring spectra
$\mathbf C$ and $\mathbf D$, 
and a commutative diagram 
\[
\xymatrix{
\bM_{e,*,1}
\ar[d]_\Sigr
&
\mathbf A
\ar[l]
\ar[r]
\ar[d]
&
\mathbf C
\ar[d]
\\
\bM_{\mathrm{rel}}^\Z
&
\mathbf B
\ar[l]
\ar[r]
&
\mathbf D
}
\]
in which the horizontal arrows and the right vertical arrow are ring maps and 
the horizontal arrows are weak equivalences.
\end{thm}

In fact $\mathbf C$ is $(\bM_{e,*,1})^{\mathrm{comm}}$, and $\mathbf D$ is 
weakly
equivalent to $(\bM_{\mathrm{rel}}^\Z)^{\mathrm{comm}}$ in the category of
commutative symmetric ring spectra (see Remark
\ref{m80}).

As far as we are aware, there is no previous result in the literature showing
multiplicativity of the symmetric signature at the spectrum level.

In Section \ref{Nov12.1} we prove a stronger statement, that up to weak 
equivalence the symmetric signature is a monoidal transformation between 
symmetric monoidal functors.  In \cite{BLM12} we proved the analogous
statement about the symmetric signature for Witt and IP bordism, using the 
methods of the present paper.

\begin{remark}
The {\it Sullivan-Ranicki orientation} for topological bundles
(\cite{MR2162361}, \cite{MR548575}, \cite[Remark 16.3]{MR1211640}, \cite[Section
13.5]{KMM}) is 
the following composite in the homotopy category of spectra
\[
M\STop \simeq \bQ_{\STop} \xrightarrow{\Sig} L^\Z,
\]
where $\bQ_{\STop}$ denotes the Quinn spectrum of oriented topological bordism
(which was shown to be equivalent to $M\STop$ in \cite[Appendix B]{LM12}).
Combining Theorems \ref{m54} and \ref{m96} shows that the Sullivan-Ranicki
orientation is represented by a ring map of commutative symmetric ring spectra. In  \cite[Section 8]{LM12} a zig-zag between the ad theories $\mbox{ad}_{STop}$ and $\mbox{ad}_{e,*,1}$ is constructed where all maps are multiplicative. 
It follows from the results below that the map from topological to Poincar\'e bordism can be refined to an $E_\infty$-map.
\end{remark}

The results of the present work were already used amongst others in \cite{MR3624098} in connection with  singularities of Baas-Sullivan type.

Here is an outline of the paper.  In Sections \ref{redef} and \ref{comsec} we
give the definition of commutative ad theory.  
The proof of Theorem \ref{Nov12.3} occupies Sections \ref{Nov12.4}--\ref{rect}.
We begin 
in Sections  \ref{Nov12.4} and \ref{gr} by giving a multisemisimplicial 
analogue $\sss$ 
of the category of symmetric spectra.  We observe that an
ad theory gives rise to an object $\bR$ of $\sss$ whose
realization is the symmetric spectrum $\bM$ mentioned above.
Section \ref{m4}
explains the key idea of the proof, which is to interpolate between the
various permutations of the
multiplication map by allowing a different order of
multiplication for each cell.  In Sections \ref{mm} and \ref{monadsss} we use
this idea to create a monad in the category $\sss$
which acts on $\bR$, and in Section \ref{rect} (after a brief technical
interlude in Section \ref{l1}) we use a standard rectification argument (as in 
\cite{MR0420610}) to convert $\bR$ with this action to a strictly commutative
object of $\sss$; passage to geometric realization gives $\bM^{\mathrm{comm}}$.
Next we turn to the proof of Theorem \ref{m54}.
In Sections \ref{relaxed}--\ref{revis}  we introduce the relaxed symmetric
Poincar\'e ad theory and the corresponding version of the symmetric signature.
In Sections \ref{colored}--\ref{msss} we create a monad in the category
$\sss\times\sss$ which acts on the pair $(\bR_{e,*,1},\bR_{\mathrm{rel}}^\Z)$,
and in Section \ref{rect2} we adapt the argument of Section \ref{rect} to
prove Theorem \ref{m54}.  Section \ref{Nov12.1} gives the statement of the
stronger version of Theorem \ref{m54} mentioned above, and Section \ref{l2}
gives the proof.  Section \ref{Thom} gives the proof of Theorem \ref{m96}. 
There is a property of the smash product in symmetric spectra which is needed to prove the main results of the paper. It is taken care of in Appendix \ref{free}.
Appendix \ref{ff14} investigates the functorial properties of the ad theories of relaxed quasi-symmetric complexes. 
\subsection*{Notations}
We adopt the notations of the articles \cite{LM12} \cite{BLM12}\cite{MR3624098}: the ambient categories, ad theories, symmetric spectra and Quinn spectra are called $\A_x$, $\ad_x$, $\bM_x$ and $\bQ_x$ respectively, with $x$ some parameter depending on the objects we work with. Note that the ambient category does not determine the ad theory. The latter does determine the associated spectra in a functorial way but we write $\bM_x$ instead of $\bM(ad_x)$ to keep the notation simpler. 
\subsection*{Acknowledgments} The authors benefited from a workshop on
forms of homotopy theory held at the Fields Institute. They would like to
thank Matthias Kreck for suggesting the problem to the first author and also
Carl-Friedrich B\"odigheimer,
Jim Davis,
Steve Ferry,
John Harper,
Mike Mandell,
Frank Quinn,
Andrew Ranicki,
John Rognes,
Stefan Schwede,
Michael Weiss
and
Bruce Williams
for useful hints and helpful discussions. The first
author is  grateful to the Max Planck Institute in Bonn for its hospitality. 
The authors also like to thank the first referee for pointing out  the need for Appendix \ref{free}.  They are also grateful to the second referee for many useful suggestions and improvements.
\section{Some redefinitions}
\label{redef}

One of the ingredients in the definition of ad theory in \cite{LM12} is the
target ``$\Z$-graded category'' $\A$ (see \cite[Definitions 3.3 and 
3.10]{LM12}).  For the purposes of that paper, there was no reason to allow 
morphisms in $\A$ between objects of the same dimension (except for identity 
maps).  For the present paper, we do need such morphisms
(see Definition \ref{perm} below and the proof of 
Theorem \ref{Nov12.3}).  We 
therefore begin by giving modified versions of some of the definitions of 
\cite{LM12}.

\begin{definition} 
\label{Z}
(cf. Definition 3.3 of \cite{LM12})\
Let $\Z$ be the poset of integers regarded as a category with one morphism for each relation. By an involution we mean an endofunctor $i$ which strictly satisfies $i^2=id$. Give $\Z$ the trivial involution.
A {\em $\Z$-graded category} is a small
category $\A$ with involution, together with 
involution-preserving functors
$ d: \A \lra \Z$ (called the {\it dimension function}) and $\emptyset : \Z
\lra \A$
such that 
$d \, \emptyset$ is equal to the identity functor. We will use the notation $|a|$ instead of $da$. Note that  if $|a|>|b|$ then there are no morphisms from $a$ to $b$.

%
\end{definition}

The definition of a strict monoidal structure on a $\Z$-graded category
(\cite[Definition 18.1]{LM12}) needs no change, provided that one uses the
new definition of $\Z$-graded category. 

Next we explain how to modify the specific examples of target categories in
\cite{LM12} by adding morphisms which preserve dimension.

For the category $\A_{\mathrm{STop}}$ (\cite[Example 3.5]{LM12}) the morphisms 
between objects of the same dimension are the orientation-preserving 
homeomorphisms.

For the category $\A_{\pi,Z,w}$ (\cite[Definition 7.3]{LM12}) the morphisms
between objects $(X,f,\xi)$ and $(X',f',\xi')$ of the same dimension are the
maps $g:X\to X'$ such that $f'\circ g=f$ and $g_*(\xi)=\xi'$.

We do not need the analogous modification for the category $\A^R$ 
(\cite[Definition 9.5]{LM12}) because we will be using the version in Section
\ref{relaxed}.

The definition of an ad theory  \cite[Definition
3.10]{LM12} stays the same. In section \ref{revis}, however,  we will restrict our attention to ad theories which are only defined on {\em strict} ball complexes. A ball complex is called strict if  each component of the intersection of two cells  is a single cell. This assumption does not affect the constructions or results of \cite{LM12}.

\section{Commutative ad theories}
\label{comsec}

\begin{definition}
\label{perm}
Let $\A$ be a $\Z$-graded category. 
A {\it permutative structure} on $\A$ is a strict monoidal 
structure 
$(\boxtimes, \varepsilon)$ (\cite[Definition 18.1]{LM12}) together with 
a natural isomorphism
\[
\gamma_{x,y}:x\boxtimes y\to i^{|x||y|} y\boxtimes x
\]
such that 

(a) $i\gamma_{x,y}=\gamma_{ix,y}=\gamma_{x,iy}$, 

(b) each of the maps $\gamma_{\emptyset,y}$, $\gamma_{x,\emptyset}$
is the identity map of $\emptyset$,

(c) the composite
\[
\xymatrix{
x\boxtimes y
\ar[r]^-{\gamma_{x,y}}
&
i^{|x||y|}y\boxtimes x
\ar[rr]^-{i^{|x||y|}(\gamma_{y,x})}
&
&
x\boxtimes y
}
\]
is the identity.

(d) $\gamma_{x,\varepsilon}$ is the identity, and

(e) the diagram
\[
\xymatrix{
&
x\boxtimes y\boxtimes z 
\ar[ld]_{1\boxtimes\gamma}
\ar[rd]^\gamma
&
\\
i^{|y||z|}x\boxtimes z\boxtimes y 
\ar[rr]^-{i^{|y||z|}\gamma\boxtimes 1}
&&
i^{|z|(|x|+|y|)}z\boxtimes x\boxtimes y 
}
\]
commutes. A {\it strict map} of $\Z$-graded categories with permutative structures is a map $f:\A \ra \mathcal{B}$ of $\Z$-graded categories for which
$f(x\boxtimes y )= f(x)\boxtimes f(y)$ holds for objects and morphisms. Moreover, such a functor takes $\epsilon$ to $\epsilon$ and the diagram
\[ 
\xymatrix{ 
f(x \boxtimes y )\ar[r]^{f(\gamma)} \ar[d]^= & f( i^{|x||y|} y\boxtimes x)\ar[d]^= \\
f(x)\boxtimes f( y) \ar[r]^{\gamma} &i^{|fx||fy|} f(y)\boxtimes f(x )
}
\]
commutes.
\end{definition}

\begin{remark} 
\label{Nov12.2}
The analogue of the coherence theorem for symmetric monoidal categories 
\cite{MR0170925}
holds in this context with essentially the same proof.
\end{remark}

\begin{definition}
\label{comad}
A {\it commutative ad theory} is a multiplicative ad theory \cite[Definitions
3.10 and 18.4]{LM12}, 
with the extra property that every pre $K$-ad which is isomorphic to a $K$-ad 
is a $K$-ad,
together with a permutative structure on the target 
category $\A$. A {\it strict map} of commutative ad theories is a strict map of the ambient categories which takes ads to ads.
\end{definition}

Examples are $\ad_C$ when $C$ is a commutative DGA (see \cite[Example 
3.12]{LM12}), $\ad_\STop$ (see \cite[Section 6]{LM12}), $\ad_{e,*,1}$ 
(see \cite[Section 7]{LM12}),
$\ad_\STopFun$ (see
\cite[end of Section 8]{LM12}), $\ad_{\text{rel}}^R$ when $R$ is commutative
(see Section \ref{relaxed} below), $\ad_{\mathrm{IP}}$ (\cite[Section
4]{BLM12}) and 
$\ad_{\mathrm{IPFun}}$ (\cite[Section 6.1]{BLM12}).

\begin{remark}
The extra property in Definition \ref{comad} is used in the proof of Lemma
\ref{m3} below.
\end{remark}

For later use we record some notation for iterated products.  

\begin{definition}
\label{Nov13}
(i) For a permutation $\eta\in\Sigma_j$, let $\epsilon(\eta)$ denote $0$ if
$\eta$ is even and $1$ if $\eta$ is odd.

(ii)
Let $\A$ be a $\mathbb Z$-graded category with a permutative 
structure.
Let $\eta\in \Sigma_j$.  Define a functor 
\[
\eta_{\bigstar}:\A^{\times j}\to \A
\]
(where $\A^{\times j}$ is the $j$-fold Cartesian product)
by 
\[
\eta_{\bigstar}(x_1,\ldots,x_j)
=
i^{\epsilon(\bar{\eta})} (x_{\eta^{-1}(1)}\boxtimes\cdots\boxtimes 
x_{\eta^{-1}(j)})
\]
where $\bar{\eta}$ is the block permutation that takes blocks 
${\mathbf b}_1$, \dots, ${\mathbf b}_j$
of size
$|x_1|$, \dots, $|x_j|$ into the order ${\mathbf b}_{\eta^{-1}(1)}$, \dots,
${\mathbf b}_{\eta^{-1}(j)}$.
\end{definition}

\begin{remark}
\label{Nov13.1}
Note that, by Remark \ref{Nov12.2}, $\eta_{\bigstar}(x_1,\ldots,x_j)$ is 
canonically isomorphic to $x_1\boxtimes\cdots\boxtimes x_j$.
\end{remark}

\section{Multisemisimplicial symmetric spectra}
\label{Nov12.4}

In this section we define a category
$\sss$ (the ss stands for ``semisimplicial'')
which is a multisemisimplicial version of the category $\S p$
of symmetric spectra.  The motivation for the definition is that the sequence 
$R_k$ in \cite[Definition 17.2]{LM12} should give an object of $\sss$.

Recall that we write $\Delta_\inj$ for the category whose objects are the sets
$\{0,\ldots,n\}$ and whose morphisms are the monotonically increasing
injections. 

A based $k$-fold {\it multisemisimplicial set} is a 
contravariant functor from the Cartesian product $(\Delta_\inj)^{\times k}$ 
to the category $\Set_*$ of based 
sets.  In particular, a based $0$-fold multisemisimplicial set is just a 
based set.

Next note that given a category $\C$ with a left action of a group $G$ one can
define a category $G\ltimes \C$ whose objects are those of $\C$ and whose
morphisms are pairs $(\alpha,f)$ with $\alpha\in G$ and $f$ a morphism 
of $\C$; the domain of $(\alpha,f)$ is the domain of $f$ and the target is
$\alpha^{-1}$ applied to the target of $f$.
Composition is defined by 
\[
(\alpha,f)\circ(\beta,g)=(\alpha\beta,\beta^{-1}(f)\circ g)
\]

\begin{remark}
\label{Aug27.1}
(i) $\C$ is imbedded in $G\ltimes \C$ by taking the morphism $f$ of $\C$ to
the morphism $(e,f)$ of $G\ltimes \C$, where $e$ is the identity element of $G$.

(ii)
The morphism $(\alpha,f)$ is the composite $(\alpha,\id)\circ(e,f)$.
\end{remark}

\begin{definition}  
Let $\Sigma_k$ act on $(\Delta_\inj^{\text{op}})^{\times k}$ by permuting the 
factors 
(when $k=0$, $\Sigma_0$ is the trivial group).  For each subgroup $H$ of 
$\Sigma_k$ let $H\ssk$ be the category of functors from 
$ H \ltimes (\Delta_\inj^{\text{op}})^{\times k}$ to $\Set_*$.
\end{definition}

By Remark \ref{Aug27.1}, an object of $H\ssk$ can be thought of as a based 
$k$-fold multisemisimplicial set with a 
left ``action'' of $H$ in which $H$ also acts on the multidegrees.

\begin{definition}
(i) A multisemisimplicial symmetric sequence $\bX$ is a sequence 
$X_k$, $k\geq 0$, such that $X_k$ is an object of $\Sigma_k\ssk$. 

(ii)
A {\it morphism} of multisemisimplicial symmetric sequences from $\bX$ to $\bY$
is a sequence of morphisms $f_k:X_k\to Y_k$ in $\Sigma_k\ssk$.
\end{definition}

The category of multisemisimplicial symmetric sequences will be denoted by 
$\ss$.

For our next definition, recall \cite[Definitions 17.2 and 17.3]{LM12}.

\begin{definition}
\label{Aug29}
(i)
For each $k\geq 0$ extend the object $R_k=\ad^k(\Delta^\bullet)$ of $\ssk$ to an object 
of 
$\Sigma_k\ssk$ by letting
\[
(\alpha,\mathrm{id})_*(F)=i^{\epsilon(\alpha)}\circ F\circ \alpha_\#
\]
(where $\alpha\in \Sigma_k$ and $F\in\ad^k(\Delta^{\mathbf n})$).

(ii)
Let $\bR$ denote the object of $\ss$ whose $k$-th term is $R_k$. 
\end{definition}

Next we assemble the ingredients needed to define a symmetric monoidal
structure on $\ss$.

\begin{definition} 
\label{Aug29.4}
Given $A\in \Sigma_k\ssk$ and $B\in 
\Sigma_l{ss}\S_l$, 
define the object $A\wedge B\in (\Sigma_k\times\Sigma_l){{ss}}\S_{k+l}$ by
\[
(A\wedge B)_{{\mathbf m},\bfn}=
A_{\mathbf m}\wedge B_{\bfn}
\]
(where $\mathbf m$ is a $k$-fold multi-index and $\bfn$ is an $l$-fold
multi-index).
\end{definition}

\begin{definition}
Given $H\subset G\subset \Sigma_k$, define a functor
\[
I_H^G: H\ssk\to G\ssk 
\]
by letting
$I_H^G A$ 
be the left Kan extension of $A$ along 
$H\ltimes (\Delta_\inj^{\text{op}})^{\times k}
\to G\ltimes (\Delta_\inj^{\text{op}})^{\times k}$.
\end{definition}

\begin{remark}
\label{Aug29.2}
For later use we give an explicit description of $I_H^G A$.
For each multi-index $\bfn$, we have
\[
(I_H^G A)_\bfn=
\Bigl(
\bigvee_{\alpha\in G} A_{\alpha^{-1}(\bfn)}
\Bigr)
/H
\]
where the action of $H$ is defined as follows: if $\beta\in H$ and 
$x$ is an
element in the $\alpha$-summand then $\beta$ takes $x$ to 
the element $(\beta,\mathrm{id})_*(x)$ in the $\alpha\beta^{-1}$-summand.
\end{remark}

\begin{notation}
\label{m6}
We denote the equivalence class of an element $x$ in 
the $\alpha$-summand of $I_H^G A$ by $[\alpha,x]$; note that $[\alpha,x]=
(\alpha,\mathrm{id})_*[e,x]$.
\end{notation}

\begin{definition} 
\label{m7}
Given $\bX,\bY\in \ss$, define $\bX\otimes \bY\in \ss$ by
\[
(\bX\otimes \bY)_k=\bigvee_{j_1+j_2=k}\,
I^{\Sigma_k}_{\Sigma_{j_1}\times \Sigma_{j_2}}\, 
\bigl(\bX_{j_1}\wedge \bY_{j_2}\bigr).
\]
\end{definition}

The proof that $\otimes$ is a symmetric monoidal product is essentially the 
same as the corresponding proof in \cite[Section 2.1]{MR1695653}. The 
symmetry map
\[
\tau:\bX\otimes\bY\to \bY\otimes\bX
\]
is given (as in \cite{{MR1695653}}) by
\begin{equation}
\label{m8}
\tau([\alpha,x\wedge y])
=[\alpha\rho_{l,k},y\wedge x]
\end{equation}
where $x\in (X_k)_{\mathbf{m}}$, $y\in (Y_l)_\bfn$, and $\rho_{l,k}$ is the 
permutation of $\{1,\ldots,k+l\}$ which moves the first $l$ elements to the
end and the last $k$ elements to the front.

Next we will give the definition of the category $\sss$ and its
symmetric monoidal product.  First we need a sphere object.

\begin{definition}
\label{z1}
(i) Let $S^1$ be the based semisimplicial set that consists of the base 
point together with a 1-simplex $s$.  

(ii) Let $S^k$ be the object of $\Sigma_k\ssk$ obtained from $(S^1)^{\wedge
k}$ by letting 
\[
(\alpha,\mathrm{id})_*(s\wedge\cdots\wedge s)=
s\wedge\cdots\wedge s
\]

(iii) Let 
$\bS$ be the object of $\ss$ whose $k$-th term is $S^k$.  
\end{definition}

It is easy to check
that $\bS$ is a commutative monoid in  $\ss$.

\begin{definition} 
\label{m23}
$\sss$ is the category of modules over $\bS$.
\end{definition}

\begin{remark}  
One can give a more explicit version of this definition:
an object of $\sss$ consists of an object $\bX$ of $\ss$ together with 
suspension maps
\[
\omega:S^1\wedge X_k\to X_{k+1}
\]
for each $k$, such that the iterates of the $\omega$'s satisfy appropriate
equivariance conditions.
\end{remark}

\begin{example} 
\label{Aug29.1}
The object $\bR$ of Definition \ref{Aug29} can be given suspension maps as
follows:
with the notation of 
\cite[Definition 17.4(i)]{LM12}, define 
\[
\omega:S^1\wedge R_k \to R_{k+1}
\]
by 
\[
\omega(s\wedge F)=\lambda^*(F)
\]
(where $s$ is the 1-simplex of $S^1$ and $F\in\ad^k(\Delta^{\mathbf n})$).
The resulting object of $\sss$ will also be denoted $\bR$.
\end{example}

\begin{definition}\label{smash} (cf.\ \cite[Definition 2.2.3]{MR1695653})
\label{m19}
For $\bX,\bY\in \sss$, define the {\it smash product} $\bX\wedge \bY$ to be the 
coequalizer of the diagram
\[
\bX\otimes \bS\otimes \bY \rightrightarrows
\bX\otimes \bY 
\]
where the right action of $\bS$ on $\bX$ is the composite
\[
\bX\otimes \bS\to \bS\otimes \bX\to \bX.
\]
\end{definition}

The proof that $\wedge$ is a symmetric monoidal product is essentially the
same as the corresponding proof in \cite[Section 2.2]{MR1695653}.


%

\section{Geometric realization}
\label{gr}

Let $G$ be a subgroup of $\Sigma_k$.
By Remark \ref{Aug27.1}(i), an object of $G\ssk$ has an underlying $k$-fold 
multisemisimplicial set.

\begin{definition}  The {\it geometric realization} $|A|$ 
of an object $A\in G\ssk$ is the geometric realization of its 
underlying $k$-fold multisemisimplicial set where, additionally, the realization of the base points are collapsed to a single point.
\end{definition}

\begin{definition} 
(i)
A map in $G\ssk$ is a {\it weak equivalence} if it 
induces a weak equivalence of realizations.

(ii) A map $\bX\to \bY$ in $\ss$ or in $\sss$ is a {\it weak equivalence} if 
each 
map $X_k\to Y_k$ is a weak equivalence.
\end{definition}

\begin{prop}  
\label{m98}
For $A\in \Sigma_k\ssk$, the following formula
gives a natural left $\Sigma_k$ action on $|A|$:
\[
\alpha([u_1,\ldots,u_k,a])
=
[u_{\alpha^{-1}(1)}, \ldots, u_{\alpha^{-1}(k)}, (\alpha,\mathrm{id})_*(a)];
\]
here $(u_1,\ldots,u_k)\in 
\Delta^\bfn$, $a\in A_\bfn$, and $[u_1,\ldots,u_k,a]$ denotes the class of 
$(u_1,\ldots,u_k,a)$ in $|A|$.
\qed
\end{prop}

\begin{prop} 
\label{m38}
For $H\subset G\subset \Sigma_k$ and $A\in H\ssk$ 
there is a natural isomorphism of based $G$-spaces
\[
|I_H^G A|\cong G_+\wedge_H\, |A|.
\]
\end{prop}

\begin{proof}
The proof is easy, using Remark \ref{Aug29.2}.
\end{proof}

\begin{cor}  Geometric realization induces a symmetric monoidal
functor from $\sss$ to the category of symmetric spectra $\S p$; in 
particular, the realization of a (commutative) monoid in $\sss$ is a 
(commutative) monoid in $\S p$.
\qed
\end{cor}

\section{A family of multiplication maps}
\label{m4}

In this section we begin the proof of Theorem \ref{Nov12.3}.
We shall construct a monad $\bbP$ together with maps $$\bbA=\bbA ssoc\ra \bbP \ra \bbC omm=\bbP'$$ such that a) the functor $\bR$ factors over $\bbP$-algebras, b) the transformation $\bbP \ra \bbP' $ is a weak equivalence in a suitable sense so $\bR$ is weakly equivalent as a functor to $\bbP$-algebras to another functor that factors over $\bbP'$-algebras.

From now until the end of Section \ref{rect} we fix a $\Z$-graded permutative 
category $\A$ and a commutative ad theory with values in $\A$.  Let $\bR$ be 
the object of $\sss$ constructed from this ad theory as in Example
\ref{Aug29.1}.

Let $\bM$ be the symmetric ring spectrum associated to the ad theory
(\cite[Proposition 17.5 and Theorem 18.5]{LM12}).  By definition, $M_k=|R_k|$.
The multiplication of $\bM$
is induced by the collection of maps
\[
\mu:(R_k)_{\mathbf{m}}\wedge(R_l)_\bfn\to (R_{k+l})_{\mathbf{m},\bfn}
\]
defined by
\begin{equation}
\label{m15}
\mu(F\wedge G)(\sigma_1\times \sigma_2,o_1\times o_2)
=
i^{l \dim \sigma_1}F(\sigma_1,o_1)\boxtimes G(\sigma_2,o_2)
\end{equation}
(this is well-defined because, 
by \cite[Definition 18.1(b)]{LM12},
reversing the orientations $o_1$ and $o_2$ 
does not change the right-hand side). 
These maps give $\bR$ the structure of a monoid in $\sss$ (the proof is 
is essentially the same as for \cite[Theorem 18.5]{LM12}).  

In general, even though the ad theory is commutative, 
$\bR$ is not a commutative monoid (this would require the 
product in the target category $\A$ to be {\it strictly} graded commutative).
Instead we have the following.
Recall Definition \ref{Nov13} and Notation \ref{m6}.

\begin{lemma}
\label{m10}
Let $m:\bR\wedge\bR\to\bR$ be the product and
let $\eta\in\Sigma_j$.  Then the composite
\[
m_\eta:\bR^{\wedge j}
\xrightarrow{\eta}
\bR^{\wedge j}
\xrightarrow{m}
\bR
\]
is determined by the formula
\begin{multline*}
m_\eta([e,F_1\wedge\cdots\wedge F_j])
(\sigma_1\times\cdots\times\sigma_j,o_1\times\cdots\times 
o_j)
\\
=
i^{\epsilon(\zeta)}
\eta_\bigstar
(F_1(\sigma_1,o_1), \ldots, F_j(\sigma_j,o_j)),
\end{multline*}
where $e$ is the identity element of the relevant symmetric group
and $\zeta$ is the block permutation that takes blocks
${\mathbf b}_1$, \dots, ${\mathbf b}_j$, ${\mathbf c}_1$, \dots, ${\mathbf
c}_j$
of size
$\deg F_1$, \dots, $\deg F_j$, $\dim \sigma_1$, \dots, $\dim \sigma_j$ into 
the order
${\mathbf b}_1$, ${\mathbf c}_1$, \dots, ${\mathbf b}_j$, ${\mathbf c}_j$.
\end{lemma}

\begin{proof}
It suffices to prove this when $\eta$ is a transposition, and in this case the
proof is an easy calculation using Equations \eqref{m8} and
\eqref{m15} and \cite[Definition 17.3]{LM12}.
\end{proof}

The key idea in the proof of Theorem \ref{Nov12.3} is that 
there is a family of operations which can be used 
to interpolate between the various 
$m_\eta$.  To construct this family,
we allow a different permutation of the factors for each cell 
of $\Delta^{\bfn_1}\times\cdots\times \Delta^{\bfn_j}$, as explained in our 
next definition.  We begin by defining the operations for pre-ads 
(see \cite[Definitions 3.8(i) and 3.10(ii)]{LM12}).

\begin{definition}  
\label{m22}
(i) Given a ball complex $K$, let $U(K)$ denote the set of all cells of $K$.

(ii) 
Let $k_1,\ldots,k_j$ be non-negative integers and let $\bfn_i$ be a
$k_i$-fold multi-index for $1\leq i\leq j$.  For any map
\[
a: U(\Delta^{\bfn_1}\times\cdots\times \Delta^{\bfn_j})\to \Sigma_j
\]
define a map on preads (i.e.\ functors $F$ from the oriented cells $(\sigma, o)$ to the ambient category)
\[
a_*:
\pre^{k_1}(\Delta^{\bfn_1})
\times
\cdots
\times
\pre^{k_j}(\Delta^{\bfn_j})
\to
\pre^{k_1+\cdots+k_j}(\Delta^{(\bfn_1,\ldots,\bfn_j)})
\]
by 
\begin{multline*}
a_*(F_1,\ldots, F_j)(\sigma_1\times\cdots\times\sigma_j,o_1\times\cdots\times 
o_j))
\\
=
i^{\epsilon(\zeta)}
(a(\sigma_1\times\cdots\times\sigma_j))_\bigstar
(F_1(\sigma_1,o_1), \ldots, F_j(\sigma_j,o_j)),
\end{multline*}
where $\zeta$ is the block permutation that takes blocks 
${\mathbf b}_1$, \dots, ${\mathbf b}_j$, ${\mathbf c}_1$, \dots, ${\mathbf c}_j$
of size
$k_1$, \dots, $k_j$, $\dim \sigma_1$, \dots, $\dim \sigma_j$ into the order 
${\mathbf b}_1$, ${\mathbf c}_1$, \dots, ${\mathbf b}_j$, ${\mathbf c}_j$.
\end{definition}

\begin{lemma}
\label{m3}
If $F_i\in \ad^{k_i}(\Delta^{\bfn_i})$ for $1\leq i\leq j$ then 
$a_*(F_1,\ldots, F_j)
\in \ad^{k_1+\cdots+k_j}(\Delta^{(\bfn_1,\ldots,\bfn_j)})$.
\end{lemma}

\begin{proof}
This is immediate from 
the extra property 
in Definition \ref{comad},
Remark \ref{Nov13.1}, and \cite[Definition 18.4(b)]{LM12}.
\end{proof}

Recalling that $(R_k)_{\bfn}=\ad^k(\Delta^\bfn)$ with basepoint at the trivial
ad (see \cite[Definitions 3.8(ii), 3.10(b) and 18.1(c)]{LM12}), we have now 
constructed an operation
\[
a_*:
(R_{k_1})_{\bfn_1}
\wedge\cdots\wedge
(R_{k_j})_{\bfn_j}
\to
(R_{k_1+\cdots+k_j})_{\bfn_1,\ldots,\bfn_j}
\]
for each 
\[
a: U(\Delta^{\bfn_1}\times\cdots\times \Delta^{\bfn_j})\to \Sigma_j.
\]

For later use, we give the relation between $a_*$ and the suspension map
$\omega:S^1\wedge R_k\to R_{k+1}$.

\begin{definition}
\label{m24}
For a ball complex $K$, let
\[
\Pi:U(\Delta^1\times K)\to U(K)
\]
be the map which takes $\sigma\times\tau$ to $\tau$, where $\sigma$ is a
simplex of $\Delta^1$ and $\tau$ is a simplex of $K$.
\end{definition}

\begin{lemma}
\label{m5}
Let $s$ be the 1-simplex of $S^1$, let $F_i\in (R_{k_i})_{\bfn_i}$ for $1\leq
i\leq j$, and let $a: U(\Delta^{\bfn_1}\times\cdots\times \Delta^{\bfn_j})\to
\Sigma_j$.  Then
\[
\omega(s\wedge a_*(F_1\wedge\cdots\wedge F_j))
=
(a\circ \Pi)_*(\omega(s\wedge F_1)\wedge\cdots\wedge F_j).
\]
\end{lemma}

\begin{proof}
This follows from Lemma \ref{m10} (because the permuted
multiplication commutes with suspension).
It can also be proved by a straightforward calculation using Definitions
\ref{Nov13} and \ref{Aug29.1} and \cite[Definitions 17.4 and 3.7(ii)]{LM12}.
\end{proof}

In the remainder of this section, we show that the action of the 
operations $a_*$ can be described in a way that begins to resemble the action 
of an operad; this resemblance will be developed further in the next two 
sections.

\begin{definition}
\label{m16}
(i) For $j,k\geq 0$ define
an object $\O(j)_k$ of $\Sigma_k\ssk$ by
\[
(\O(j)_k)_\bfn=\Map(U(\Delta^\bfn),\Sigma_j)_+
\]
(where the $+$ denotes a disjoint basepoint);  the morphisms in
$(\Delta_\inj^{\text{op}})^{\times k}$ act in the evident way, and the
morphisms of the form $(\alpha,\mathrm{id})$ with $\alpha\in \Sigma_k$ act by 
permuting the factors in $\Delta^\bfn$.

(ii) For $j\geq 0$ define $\O(j)$ to be the object of $\ss$ with $k$-th term
$\O(j)_k$.
\end{definition}

\begin{definition} 
(i)
For $A,B\in \Sigma_k\ssk$, define the {\it degreewise smash product}
\[
A\owedge B \in \Sigma_k\ssk
\]
by
\[
(A\owedge B)_\bfn=A_\bfn\wedge B_\bfn,
\]
with the diagonal action of $\Sigma_k$.

(ii)
For $\bX,\bY\in\ss$, define
$\bX\owedge \bY\in\ss$ by 
\[
(\bX\owedge \bY)_k=X_k\owedge Y_k.
\]
\end{definition}

\begin{remark}
The difference between the degreewise smash product $A\owedge B$ and the
previously defined smash product $A\wedge B$ is that the former is only 
defined when $A$ and $B$ are $k$-fold multisemisimplicial sets for the same 
$k$, and the result is again a $k$-fold multisemisimplicial set, whereas 
$A\wedge B$ is defined when $A$ is $k$-fold and $B$ is $l$-fold, and the result
is $(k+l)$-fold.  
\end{remark}

Our next definition assembles the operations $a_*$ for a given $j$ into a 
single map.  

\begin{definition}
\label{m17}
Let $j\geq 0$.
Define a map 
\[
\phi_j: \O(j)
\owedge 
\bR^{\otimes j}
\to
\bR
\]
in $\ss$ by the formulas
\[
\phi_j(a\wedge [e, F_1\wedge\cdots\wedge F_j])
=
a_*(F_1\wedge\cdots\wedge F_j)
\]
(where $e$ denotes the identity element of the relevant symmetric group)
and
\[
\phi_j(a\wedge [\alpha, F_1\wedge\cdots\wedge F_j])
=
(\alpha,\mathrm{id})_*\phi_j((\alpha^{-1},\mathrm{id})_*a\wedge [e,
F_1\wedge\cdots\wedge F_j]).
\]
\end{definition}

\begin{lemma}
\label{m18}
The map $\phi_j$ induces a map
\[
\psi_j: \O(j)
\owedge 
\bR^{\wedge j}
\to
\bR
\]
in $\ss$.
\end{lemma}

\begin{proof}
This is a straightforward calculation using Example \ref{Aug29.1}, 
Definition \ref{m19}, and 
\cite[Definition 17.3 and Lemma 18.7]{LM12}.
\end{proof}

\section{A monad in $\ss$}
\label{mm}

In the next section we will show 
that there is a monad $\bbP$ in $\sss$
with the property that the maps $\psi_j$ constructed in Lemma \ref{m18}
give an action of $\bbP$ on $\bR$.   As preparation, in this section we
prove the analogous result in $\ss$; that is, we
show that there is a monad $\bbO$ in $\ss$ for which the maps $\phi_j$ of
Definition \ref{m17} give an action of $\bbO$ on $\bR$.

First we observe that the collection of objects $\O(j)$ has a composition map
analogous to that of an operad.   Recall that May defines an operad $\mathcal 
M$ in the category of sets with ${\mathcal M}(j)=\Sigma_j$ (\cite[Definition 
3.1(i)]{MR0420610}).  Let $\gamma_{\mathcal M}$ denote the composition 
operation in $\mathcal M$.  Also recall Definition \ref{m7} and Notation 
\ref{m6}. 

\begin{definition}
Given $j_1,\ldots,j_i\geq 0$ define a map 
\[
\gamma:
\O(i)
\owedge
\bigl(\O(j_1)\otimes\cdots\otimes O(j_i)\bigr)
\to
\O({j_1+\cdots+j_i})
\]
in $\ss$ by the formulas
\[
\gamma(a\wedge [e, b_1\wedge \ldots\wedge b_i])
(\sigma_1\times\cdots\times \sigma_i)
=
\gamma_{\mathcal M}(a(\sigma_1\times\cdots\times
\sigma_i),b_1(\sigma_1),\ldots, b_i(\sigma_i))
\]
(where $e$ is the identity element of the relevant symmetric group)
and
\[
\gamma(a\wedge [\alpha, b_1\wedge \ldots\wedge b_i])
=
(\alpha,\mathrm{id})_*\gamma((\alpha^{-1},\mathrm{id})_*a\wedge [e, b_1\wedge
\ldots\wedge b_i]).
\]
\end{definition}

In order to formulate the associativity property of $\gamma$, we 
note  that
for $\bX_1,\ldots,\bX_i, \bY_1,\ldots,\bY_i \in \ss$ there is a natural map
\[
\chi:(\bX_1\owedge \bY_1)\otimes\cdots\otimes(\bX_i\owedge \bY_i)
\to 
(\bX_1\otimes\cdots\otimes \bX_i)\owedge
(\bY_1\otimes\cdots\otimes \bY_i)
\]
given by 
\[
\chi([\alpha,x_1\wedge y_1\wedge \cdots\wedge x_i\wedge y_i])
=
[\alpha,x_1\wedge \cdots\wedge x_i]\wedge[\alpha,y_1\wedge\cdots\wedge y_i].
\]

\begin{lemma}
\label{m14}
The operation $\gamma$ has the following associativity property:
the composite
\begin{multline*}
\O(i)
\owedge
\Bigl(
\bigl(
\O(j_1)\owedge
\bigl(
\O(l_{11})\otimes\cdots\otimes\O(l_{1j_1})
\bigr)
\bigr)
\otimes\cdots\otimes
\bigl(\O(j_i)\owedge
\bigl(
\O(l_{i1})\otimes\cdots\otimes\O(l_{ij_i})
\bigr)
\bigr)
\Bigr)
\\
\xrightarrow{1\owedge(\gamma\otimes\cdots\gamma)}
\O(i)\owedge
\bigl(
\O(l_{11}+\cdots+l_{1j_1})
\otimes
\cdots\otimes
\O(l_{i1}+\cdots+l_{ij_i})
\bigr)
\\
\xrightarrow{\gamma}
\O(l_{11}+\cdots+l_{ij_i})
\end{multline*}
is the same as the composite
\begin{multline*}
\O(i)
\owedge
\Bigl(
\bigl(
\O(j_1)\owedge
\bigl(
\O(l_{11})\otimes\cdots\otimes\O(l_{1j_1})
\bigr)
\bigr)
\otimes\cdots\otimes
\bigl(\O(j_i)\owedge
\bigl(
\O(l_{i1})\otimes\cdots\otimes\O(l_{ij_i})
\bigr)
\bigr)
\Bigr)
\\
\xrightarrow{1\owedge \chi}
\O(i)\owedge
\bigl(
\O(j_1)\otimes \cdots\otimes  \O(j_i)
\bigr)
\owedge
\bigl(
\O(l_{11})\otimes\cdots\otimes\O(l_{i,j_i})
\bigr)
\\
\xrightarrow{\gamma\owedge 1}
\O(j_1+\cdots+j_i)\owedge
\bigl(
\O(l_{11})\otimes\cdots\otimes\O(l_{i,j_i})
\bigr)
\\
\xrightarrow{\gamma}
\O(l_{11}+\cdots+l_{ij_i}).
\end{multline*}
\qed
\end{lemma}

To formulate the unital property of $\gamma$ we first need to consider the unit
object for the operation $\owedge$.

\begin{definition}
\label{m1}
Let $\oS_k$ be the object of $\Sigma_k\ssk$ 
which has a copy of $S^0$ in each multidegree (with each morphism of
$\Sigma_k \ltimes (\Delta_\inj^{\text{op}})^{\times k}$  acting as the identity of $S^0$), and let $\boS$ be the 
object of $\ss$ with $k$-th term $\oS_k$. 
\end{definition}

\begin{remark}
\label{m13}
(i) 
$\boS\owedge \bX\cong \bX$ for any $\bX\in\ss$.

(ii)
$\O(0)$ and $\O(1)$ are both equal to $\boS$. 

(iii)
$\boS$ is a commutative monoid in $\ss$ with multiplication 
\[
m:\boS\otimes\boS\to \boS
\] 
given by
\[
m([\alpha,s_1\wedge s_2])=t,
\]
where $s_1$ and $s_2$ are any nontrivial simplices and $t$ is the nontrivial 
simplex in the relevant multidegree.
\end{remark}

\begin{lemma}
\label{m20}
The operation $\gamma$ has the following unital property:
the diagrams
\[
\xymatrix{
\O(j)\owedge \boS^{\otimes j}
\ar[r]^-=
\ar[d]_{1\owedge m}
&
\O(j)\owedge \O(1)^{\otimes j}
\ar[d]^\gamma
\\
\O(j)\owedge \boS
\ar[r]^-\cong
&
\O(j)
}
\]
and
\[
\xymatrix{
\boS\owedge\O(j)
\ar[r]^-\cong
\ar[d]_=
&
\O(j)
\\
\O(1)\owedge\O(j)
\ar[ru]_\gamma
}
\]
commute.
\qed
\end{lemma}

To complete the analogy between $\gamma$ and the composition map of an operad,
we need an equivariance property.

\begin{definition}  
Define a right action of $\Sigma_j$ on $\O(j)$ by
\[
(a\alpha)(\sigma)=a(\sigma)\cdot\alpha,
\]
where $a\in \Map(U(\Delta^\bfn),\Sigma_j)_+$, $\sigma\in U(\Delta^\bfn)$, and
$\cdot$ is multiplication in $\Sigma_j$.
\end{definition}

\begin{lemma}
\label{m72}
(i) The following diagram commutes for all $\alpha\in \Sigma_i$.
\[
\xymatrix{
\O(i)\owedge \bigl(\O(j_1)\otimes\cdots\otimes\O(j_i)\bigr)
\ar[r]
\ar[dd]_{\alpha\owedge(\beta_1\otimes\cdots\otimes\beta_i)}
&
\O(i)\owedge
\bigl(\O(j_{\alpha^{-1}(1)})\otimes\cdots\otimes\O(j_{\alpha^{-1}(i)})\bigr)
\ar[d]_\gamma
\\
&
\O(j_1+\cdots+j_i)
\ar[d]_{\gamma_M(\alpha,\beta_1,\ldots,\beta_i)}
\\
\O(i)\owedge \bigl(\O(j_1)\otimes\cdots\otimes\O(j_i)\bigr)
\ar[r]^-\gamma
&
\O(j_1+\cdots+j_i)
}
\]
\qed
\end{lemma}

Now we use the data defined so far to construct a monad in the category $\ss$.

\begin{definition}
\label{m25}
(i) For $\bX\in\ss$, give $\O(j)\owedge \bX^{\otimes j}$ the diagonal right
$\Sigma_j$ action.

(ii) Define a functor $\bbO:\ss\to\ss$ by 
\[
\bbO(\bX)=
\bigvee_{j\geq 0}
\bigl(
\O(j)\owedge \bX^{\otimes j}
\bigr)
/ \Sigma_j.
\]

(iii)
Define a natural transformation
\[
\iota:\bX\to \bbO \bX
\]
to be the composite
\[
\bX\xrightarrow{\cong}
\boS\owedge \bX
=
\O(1)\owedge \bX
\hookrightarrow
\bbO(\bX).
\]

(iv)
Define 
\[
\mu:\bbO\bbO \bX\to \bbO \bX
\]
to be the natural transformation induced by the maps
\begin{multline*}
\O(i)
\owedge
\Bigl(
\bigl(
\O(j_1)\owedge
\bX^{\otimes j_1}
\bigr)
\otimes\cdots\otimes
\bigl(\O(j_i)\owedge
\bX^{\otimes j_i}
\bigr)
\Bigr)
\\
\xrightarrow{1\owedge\chi}
\Bigl(\O(i)\owedge
\bigl(
\O(j_1)\otimes \cdots\otimes  \O(j_i)
\bigr)
\Bigr)
\owedge
\bX^{\otimes (j_1+\cdots+j_i)}
\\
\xrightarrow{\gamma\owedge 1}
\O(j_1+\cdots+j_i)\owedge
\bX^{\otimes (j_1+\cdots+j_i)}.
\end{multline*}
\end{definition}

\begin{prop}
\label{m27}
The transformations $\mu$ and $\iota$ define a monad
structure on $\bbO$.
\end{prop}

\begin{proof}  
This is immediate from Lemmas \ref{m14} and \ref{m20}.
\end{proof}

We conclude this section by giving the action of $\bbO$ on $\bR$.  
Observe that
the map 
\[
\phi_j: \O(j)
\owedge 
\bR^{\otimes j}
\to
\bR
\]
of Definition \ref{m17} induces a map
\begin{equation}
\label{m21}
\bigl(
\O(j)\owedge \bR^{\otimes j}
\bigr)
/ \Sigma_j
\to
\bR.
\end{equation}

\begin{definition} Define 
\label{m75}
\[
\nu: \bbO \bR\to \bR
\]
to be the map whose restriction to
$\bigl( \O(j)\owedge \bX^{\otimes j} \bigr) / \Sigma_j $ 
is the map \eqref{m21}.
\end{definition}

\begin{prop}
\label{m30}
$\nu$ is an action of $\bbO$ on $\bR$.
\end{prop}

\begin{proof}
We need to show that the diagrams
\[
\xymatrix{
\bR
\ar[r]^\iota
\ar[rd]_=
&
\bbO \bR
\ar[d]^\nu
\\
&
\bR
}
\]
and
\[
\xymatrix{
\bbO\bbO \bR
\ar[r]^\mu
\ar[d]_{\bbO \nu}
&
\bbO \bR
\ar[d]_\nu
\\
\bbO \bR
\ar[r]^\nu
&
\bR
}
\]
commute.   The first is obvious and for the second it suffices to check that
the composite
\begin{multline*}
\O(i)
\owedge
\Bigl(
\bigl(
\O(j_1)\owedge
\bR^{\otimes j_1}
\bigr)
\otimes\cdots\otimes
\bigl(\O(j_i)\owedge
\bR^{\otimes j_i}
\bigr)
\Bigr)
\\
\xrightarrow{1\owedge\chi}
\O(i)\owedge
\bigl(
\O(j_1)\otimes \cdots\otimes  \O(j_i)
\bigr)
\owedge
\bR^{\otimes (j_1+\cdots+j_i)}
\\
\xrightarrow{\gamma\owedge 1}
\O(j_1+\cdots+j_i)\owedge
\bR^{\otimes (j_1+\cdots+j_i)}
\xrightarrow{\phi}
\bR
\end{multline*}
is the same as the composite
\begin{multline*}
\O(i)
\owedge
\Bigl(
\bigl(
\O(j_1)\owedge
\bR^{\otimes j_1}
\bigr)
\otimes\cdots\otimes
\bigl(\O(j_i)\owedge
\bR^{\otimes j_i}
\bigr)
\Bigr)
\xrightarrow{1\owedge(\phi\otimes\cdots\phi)}
\O(i)\owedge
\bR^{\otimes i}
\xrightarrow{\phi}
\bR.
\end{multline*}
\end{proof}

\section{A monad in $\sss$}
\label{monadsss}

We begin by giving $\O(j)\owedge \bX^{\wedge j}$ the structure of a
multisemisimplicial symmetric spectrum when $\bX\in \sss$.  
The definition is
motivated by Lemma \ref{m5}.  Recall Definition \ref{m24}.

\begin{definition}
\label{m29}
Let $j,k\geq 0$. 
Let $s$ be the 1-simplex of $S^1$.   
Define
\[
\omega:S^1\wedge (\O(j)\owedge \bX^{\wedge j})_k
\to
(\O(j)\owedge \bX^{\wedge j})_{k+1}
\]
as follows: for
$a\in (\O(j)_k)_\bfn$ and $x\in
((X^{\wedge j})_k)_\bfn$, let
\[
\omega(s\wedge (a\wedge x))
=
(a\circ \Pi)\wedge \omega(s\wedge x).
\]
The identification given in Definition \ref{smash} was omitted in this notation.
\end{definition}

\begin{definition}
(i) For $\bX\in\sss$, give $\O(j)\owedge \bX^{\wedge j}$ the diagonal right
$\Sigma_j$ action.

(ii) Define a functor $\bbP:\sss\to\sss$ by
\[
\bbP(\bX)=
\bigvee_{j\geq 0}
\bigl(
\O(j)\owedge \bX^{\wedge j}
\bigr)
/ \Sigma_j.
\]
\end{definition}

To give $\bbP$ a monad structure we need

\begin{lemma}
\label{m26}
The composite in Definition \ref{m25}(iv) induces a map
\begin{multline*}
\O(i)
\owedge
\Bigl(
\bigl(
\O(j_1)\owedge
\bX^{\wedge j_1}
\bigr)
\wedge\cdots\wedge
\bigl(\O(j_i)\owedge
\bX^{\wedge j_i}
\bigr)
\Bigr)
\\
\to
\O(j_1+\cdots+j_i)\owedge
\bX^{\wedge (j_1+\cdots+j_i)}.
\end{multline*}
in $\sss$.
\qed
\end{lemma}

\begin{definition}
(i)
Define a natural transformation
\[
\iota:\bX\to \bbP \bX
\]
to be the composite
\[
\bX\xrightarrow{\cong}
\boS\owedge \bX
=
\O(1)\owedge \bX
\hookrightarrow
\bbP(\bX).
\]

(ii)
Define
\[
\mu:\bbP\bbP \bX\to \bbP \bX
\]
to be the natural transformation induced by the maps constructed in Lemma
\ref{m26}.
\end{definition}

\begin{prop}
The transformations $\mu$ and $\iota$ define a monad
structure on $\bbP$.
\end{prop}

\begin{proof}
This follows from Proposition \ref{m27} by passage to quotients.
\end{proof}

Next we give the action of $\bbP$ on $\bR$.
By Definition \ref{m29} and Lemma \ref{m5},
the map
\[
\psi_j: \O(j)
\owedge 
\bR^{\wedge j}
\to
\bR
\]
of Lemma \ref{m18} is a map in $\sss$.  It
induces a map
\begin{equation}
\label{m28}
\bigl(
\O(j)\owedge \bR^{\wedge j}
\bigr)
/ \Sigma_j
\to
\bR.
\end{equation}
in $\sss$.

\begin{definition} Define
\[
\nu: \bbP \bR\to \bR
\]
to be the map whose restriction to
$\bigl( \O(j)\owedge \bX^{\wedge j} \bigr) / \Sigma_j $
is the map \eqref{m28}.
\end{definition}

\begin{prop}
\label{m62}
$\nu$ is an action of $\bbP$ on $\bR$.
\end{prop}

\begin{proof}
This follows from Proposition \ref{m30} by passage to quotients.
\end{proof}

For use in Section \ref{rect} we record a lemma.

\begin{lemma}
\label{m35}
(i) There is a functor $\Upsilon$ from $\bbP$ algebras to monoids in $\sss$ 
(with respect to $\wedge$) which is the identity on objects.

(ii) The geometric realization of $\Upsilon(\bR)$ is the symmetric ring
spectrum $\bM$ of \cite[Theorem 18.5]{LM12}.
\end{lemma}

\begin{proof}  
Part (i).
Let ${\mathbb A}$ be the monad
\[
{\mathbb A}(\bX)=\bigvee_{j\geq 0} \bX^{\wedge j}.
\]
Then a monoid in $\sss$ is the same thing as an
$\mathbb A$-algebra, so it suffices to give a map of monads from $\mathbb A$ 
to $\bbP$.  

For each $j,k\geq 0$ and each $k$-fold multi-index $\bfn$, define an
element
\[
e_{j,k,\bfn}\in (\O(j)_k)_\bfn
\]
to be the constant function $U(\Delta^\bfn)\to\Sigma_j$ whose value is the
identity element of $\Sigma_j$.  Next define a map
\[
\boS\to \O(j)
\]
by taking the nontrivial simplex of $(\boS_k)_\bfn$ to $e_{j,k,\bfn}$.  

Now the composite
\[
{\mathbb A}(\bX)=\bigvee_{j\geq 0} \bX^{\wedge j}
\cong
\bigvee_{j\geq 0} \boS\owedge \bX^{\wedge j}
\to
\bigvee_{j\geq 0} \bigl(\O(j)\owedge \bX^{\wedge j}\bigr)/\Sigma_j
=\bbP(\bX)
\]
is a map of monads.

Part (ii) is an easy consequence of the definitions.
\end{proof}
It is worth mentioning that the map from $\bbA$ to $\bbP$ does not factor over the commutative monad because $e_{j,k,n}$ is not a fixed point for the $\Sigma_n$-action.

\section{Degreewise smash product and geometric realization}
\label{l1}

For the proof of Theorem \ref{Nov12.3} 
we need to know the relation between $\owedge$ and geometric realization.  

There is a natural map
\[
\kappa:|A\owedge B|\to |A|\wedge |B|
\]
defined by
\[
\kappa([u,x\wedge y])=[u,x]\wedge[u,y].
\]
The analogous map for multisimplicial sets is a homeomorphism, but the
situation for multisemisimplicial sets is more delicate.

\begin{definition} 
\label{m31}
A multisemisimplicial set {\it has compatible degeneracies} if it
is in the image of the
forgetful functor from multisimplicial sets to multisemisimplicial sets.
\end{definition}

\begin{example}
\label{m40}
(i) One can 
define compatible degeneracies on $\O(j)_k$ for each $j,k\geq 0$ by using the 
codegeneracy maps between the $\Delta^\bfn$.

(ii) If $\bX\in \sss$ and $X_k$ has compatible degeneracies for all $k$ then 
each $\bigl(\O(j)\owedge \bX^{\wedge j}\bigr)_k$ has compatible
degeneracies.
\end{example}

\begin{prop} 
\label{Aug29.3}
If the underlying multisemisimplicial sets of $A$ and
$B$ have compatible degeneracies then $\kappa$ is a weak equivalence.
\end{prop}

\begin{proof} 
Let $\tilde{A}$ and $\tilde{B}$ be multisimplicial sets whose underlying 
multisemisimplicial sets are $A$ and $B$.  Then the underlying 
multisemisimplicial set of the degreewise smash product 
$\tilde{A}\owedge\tilde{B}$ is $A\owedge B$.  Consider the following 
commutative diagram, where $|\ |$ in the bottom row denotes realization of 
multisimplicial sets, $\tilde{\kappa}$ is defined analogously to $\kappa$, 
and the vertical maps collapse the degeneracies:
\[
\xymatrix{
|A\owedge B|
\ar[r]^{\kappa}
\ar[d]
&
|A|\wedge |B|
\ar[d]
\\
|\tilde{A}\owedge\tilde{B}|
\ar[r]^{\tilde{\kappa}}
&
|\tilde{A}|\wedge|\tilde{B}|
}
\]
The map $\tilde{\kappa}$ is a homeomorphism, and the vertical arrows are weak 
equivalences by the multisimplicial analogue of \cite[Lemma A.5]{MR0353298}, so
$\kappa$ is a weak equivalence.
\end{proof}

Next we give a sufficient condition for a multisemisimplicial set to 
have compatible degeneracies.  Let $D^n$ denote the semisimplicial set 
consisting of the nondegenerate simplices of the standard simplicial 
$n$-simplex.  For a multi-index $\bfn$, let $D^\bfn$ denote the $k$-fold 
multisemisimplicial set 
\[
D^{n_1}\times \cdots
\times D^{n_k}.
\]

\begin{definition} 
(i)
A {\it horn} in $D^\bfn$ is a subcomplex $E$ which contains all 
elements of $D^\bfn$ except for the top-dimensional element and one of its 
faces.

(ii)
A $k$-fold multisemisimplicial set $A$ satisfies the 
{\it multi-Kan condition} if every map from a horn in $D^\bfn$ to $A$ 
extends to a map $D^\bfn\to A$.
\end{definition}

The following result is proved in  \cite{MC08}.

\begin{prop} 
\label{m32}
If $A$ satisfies the multi-Kan condition then it
has compatible degeneracies.
\end{prop}

Our next result is proved in the same way as \cite[Lemma 15.12]{LM12} and
does not require the ad theory to be commutative.
 
\begin{prop}  
\label{m33}
For each $k$, $R_k$ satisfies the multi-Kan condition.
\end{prop}

\section{Rectification}
\label{rect}

In this section we complete the proof of Theorem \ref{Nov12.3}.

First we consider a monad in $\sss$ which is simpler than $\bbP$.

\begin{definition}
(i) Define $\bbP'(\bX)$ to be $\bigvee_{j\geq 0} \,
\bX^{\wedge j}/\Sigma_j$.

(ii)  For each $j\geq 0$, let
\[
\xi_j:\O(j)\to \boS
\]
be the map which takes each nontrivial simplex of $\O(j)_k$ to the nontrivial
simplex of $\oS_k$ in the same multidegree.
Define a natural transformation 
\[
\Xi: \bbP\to\bbP'
\]
to be the wedge of the composites
\[
\bigl(\O(j)\owedge \bX^{\wedge j}\bigr)/\Sigma_j
\xrightarrow{\xi_j\owedge 1}
\bigl(\boS\owedge \bX^{\wedge j}\bigr)/\Sigma_j
\xrightarrow{\cong}
\bX^{\wedge j}/\Sigma_j.
\]
\end{definition}

\begin{prop}
\label{m34}
{\rm (i)} An algebra over $\bbP'$ is the same thing as a 
commutative monoid in $\sss$.

{\rm (ii)} $\Xi$ is a map of monads.

{\rm (iii)} Suppose that each $X_k$ has compatible degeneracies (see 
Definition \ref{m31}).  Let $\bbP^q$ denote the $q$-th iterate of $\bbP$. 
Then each map 
\[
\Xi:\bbP^q(\bX)\to \bbP'\bbP^{q-1}(\bX)
\]
is a weak equivalence.
\end{prop}

Parts (i) and (ii) are immediate from the definitions.  Part (iii) 
will be proved at the end of this section.  

\begin{proof}[Proof of Theorem \ref{Nov12.3}.]
We apply the monadic bar 
construction (\cite[Construction 9.6]{MR0420610}) to obtain simplicial objects 
$B_\bu(\bbP,\bbP,\bR)$
and $B_\bu(\bbP',\bbP,\bR)$ in $\sss$.  We write $\bR_\bu$ for the constant simplicial
object which is $\bR$ in each simplicial degree.  There are maps of simplicial
$\bbP$-algebras
\begin{equation}
\label{m36}
\bR_\bu \xleftarrow{\varepsilon}
B_\bu(\bbP,\bbP,\bR)
\xrightarrow{\Xi_\bu}
B_\bu(\bbP',\bbP,\bR),
\end{equation}
where $\varepsilon$ is induced by the action of $\bbP$ on $\bR$ (see 
\cite[Lemma 9.2(ii)]{MR0420610}).  The map $\varepsilon$ is a homotopy 
equivalence of simplicial
objects (\cite[Proposition 9.8]{MR0420610}) and the map $\Xi_\bu$ is a weak 
equivalence in each simplicial degree by Propositions \ref{m32}, \ref{m33} and
\ref{m34}(iii).   
$B_\bu(\bbP',\bbP,\bR)$ is a simplicial algebra over 
$\bbP'$, 
which by Proposition \ref{m34}(i) is the same thing as a simplicial commutative 
monoid in $\sss$.  Moreover, by Lemma \ref{m35}(i), $\bR_\bu$ and 
$B_\bu(\bbP,\bbP,\bR)$ are simplicial monoids, and $\varepsilon$ 
and $\Xi_\bu$ are maps of simplicial monoids.

The objects of the diagram \eqref{m36} are simplicial objects in $\sss$.  We
obtain a diagram 
\begin{equation}
\label{m37}
|\bR_\bu| \xleftarrow{|\varepsilon|}
|B_\bu(\bbP,\bbP,\bR)|
\xrightarrow{|\Xi_\bu|}
|B_\bu(\bbP',\bbP,\bR)|
\end{equation}
of simplicial objects in $\S p$ (the category of symmetric spectra)
by applying the geometric realization functor $\sss\to \S p$
to the diagram \eqref{m36} in each simplicial degree.
The map $|\varepsilon|$ is a
homotopy equivalence of simplicial objects and the map $|\Xi_\bu|$ is a weak 
equivalence in each simplicial degree. The object $|B_\bu(\bbP',\bbP,\bR)|$ 
is a simplicial commutative symmetric ring spectrum, the objects $|\bR_\bu|$ 
and $|B_\bu(\bbP,\bbP,\bR)|$ are simplicial symmetric ring 
spectra, and the maps $|\varepsilon|$ and $|\Xi_\bu|$ are maps of simplicial 
symmetric ring spectra.

Finally, we apply geometric realization to the diagram \eqref{m37}.
We define $\bMc$ to be $||B_\bu(\bbP',\bbP,\bR)||$.
Now we have a diagram
\begin{equation}
\label{m81}
\bM=|\bR|
\xleftarrow{||\varepsilon||}
||B_\bu(\bbP,\bbP,\bR)||
\xrightarrow{||\Xi_\bu||}
||B_\bu(\bbP',\bbP,\bR)||=\bMc
\end{equation}
in $\S p$.  The map $||\varepsilon||$ is a
homotopy equivalence (cf.\ \cite[Corollary 11.9]{MR0420610}) and
the map $||\Xi_\bu||$ is a weak equivalence by \cite[Theorem E]{Reedy}. 
$\bMc$ is a commutative symmetric ring spectrum, $\bM$ is the symmetric ring
spectrum of \cite[Theorem 18.5]{LM12} (by Lemma \ref{m35}(ii)),
$||B_\bu(\bbP,\bbP,\bR)||$
is a  symmetric ring spectrum, and $||\varepsilon||$ and 
$||\Xi_\bu||$ are maps of symmetric ring spectra.  
\end{proof}

We conclude this section with the proof of part (iii) of Proposition \ref{m34}.
First we need a lemma (which for later use  we state in more generality than we
immediately need). Recall that a preorder is a set with a reflexive and
transitive relation $\leq$.  Examples are $\Sigma_j$, with every element $\leq$
every other,
and $U(K)$ (see Definition \ref{m22}(i)), with $\leq$ induced by 
inclusions of cells.

\begin{lemma}
\label{m39}
Let $P$ be a preorder with an element which is $\geq$ all other elements, and 
let $k\geq 0$. Define a $k$-fold multisemisimplicial set $A$ by 
\[
A_\bfn=\Map_{\mathrm{preorder}}(U(\Delta^\bfn),P).
\]
Then

\rm{(i)} $A$ has compatible degeneracies, and

\rm{(ii)} $A$ is weakly equivalent to a point.
\end{lemma}

\begin{proof}
For (i), we can give $A$ compatible degeneracies 
by using the codegeneracy maps between the $\Delta^\bfn$.

Part (ii).
Let $\tilde{A}$ be a multisimplicial set whose
underlying multisemisimplicial set is $A$.
Let 
$d\tilde{A}$ be its diagonal.
The multisimplicial analogue of \cite[Lemma A.5]{MR0353298} implies that 
$|A|$ is weakly equivalent to $|\tilde{A}|$, and it is well-known 
that the latter is homeomorphic to $|d\tilde{A}|$.  It therefore 
suffices to show that the simplicial set $d\tilde{A}$ is weakly 
equivalent to a point.  

Let $\Delta^n_{\mathrm{simp}}$ denote the standard 
simplicial $n$ simplex and let $\partial \Delta^n_{\mathrm{simp}}$ denote its 
boundary.  Then it suffices by \cite[Theorem I.11.2]{MR1711612} to show that 
every map 
from $\partial \Delta^n_{\mathrm{simp}}$ to $d\tilde{A}$ extends to
$\Delta^n_{\mathrm{simp}}$.  

Let $D^n$ (respectively $\partial D^n$) be the semisimplicial set consisting of
the nondegenerate simplices of $\Delta^n_{\mathrm{simp}}$ (resp., $\partial
\Delta^n_{\mathrm{simp}}$).
Since $\Delta^n_{\mathrm{simp}}$ (resp., $\partial
\Delta^n_{\mathrm{simp}}$) is the free simplicial set generated by 
$D^n$ (respectively $\partial D^n$), 
it suffices to show that every semisimplicial map from
$\partial D^n$ to $d\tilde{A}$ extends to $D^n$, and this is obvious from
the definition of $A$.
\end{proof}

Note that if $P$ is $\Sigma_j$ with the preorder described above then $A_+$ 
is $\tilde{\O}(j)_k$.

\begin{proof}[Proof of \ref{m34}(iii).]  
We begin with the case $q=1$, so we want to show that the 
map $\Xi:\bbP(\bX)\to\bbP'(\bX)$ is a weak equivalence.
It suffices to show that the map
\[
(\O(j)\owedge \bX^{\wedge j})/\Sigma_j
\to
\bX^{\wedge j}/\Sigma_j
\]
is a weak equivalence for each $j$. Proposition \ref{PropA2} and Remark \ref{RemA3}
show that the $\Sigma_j$ actions are 
free away from the basepoint. We will now employ the following fact: if $X$ is a cellular $G$-spectrum with a free action away from the base point then 
the canonical map from $EG_+\wedge_{G}X$ to $X/G$ is a weak equivalence. Moreover, $(EG)_+\wedge_{G}X$ sits in a  fibration with base $BG$ and fibre $X$.
Hence,  it suffices to show that each map
\[
\O(j)\owedge \bX^{\wedge j}
\to
\bX^{\wedge j}
\]
is a weak equivalence.  
Now the object $\bigl(\O(j)\owedge \bX^{\wedge j}\bigr)_k$ 
comes from
\[
\bigvee_{k_1+\cdots+k_j=k}\, \O(j)_k\owedge 
I_{\Sigma_{k_1}\times\cdots\times\Sigma_{k_j}}^{\Sigma_k}
\bigl(
\bX_{k_1}\wedge\cdots\wedge \bX_{k_j}
\bigr)
\]
and we have 
\[
\O(j)_k\owedge
I_{\Sigma_{k_1}\times\cdots\times\Sigma_{k_j}}^{\Sigma_k}
\bigl(
\bX_{k_1}\wedge\cdots\wedge \bX_{k_j}
\bigr)
\cong
I_{\Sigma_{k_1}\times\cdots\times\Sigma_{k_j}}^{\Sigma_k}
\Bigl(\O(j)_k\owedge
\bigl(
\bX_{k_1}\wedge\cdots\wedge \bX_{k_j}
\bigr)
\Bigr),
\]
so it suffices by Proposition \ref{m38} to show that each map
\[
\O(j)_k\owedge (\bX_{k_1}\wedge\cdots\wedge \bX_{k_j})
\to
\bX_{k_1}\wedge\cdots\wedge \bX_{k_j}
\]
is a weak equivalence, and this follows from Example \ref{m40}, Proposition 
\ref{Aug29.3}, and Lemma \ref{m39}. 

The general case follows from the case $q=1$ and Example \ref{m40}(ii).
\end{proof}

\section{Proof of Theorem \ref{m96}}

\label{Thom}

It is well known that Thom spectra are commutative symmetric
ring spectra (see for example \cite{MR2578805}; we recall this below).
In this section we show that the Thom spectrum $M\STop$ obtained from the bar construction  is weakly
equivalent, in the category of commutative symmetric ring spectra, to the
commutative symmetric ring spectrum $(\bM_\STop)^{\mathrm{comm}}$ given by 
Theorem \ref{Nov12.3}.  

Our first task is to construct the following chain of weak equivalences 
in the category of symmetric spectra.
\begin{equation}
\label{m94}
\bM_{\STop}
\xleftarrow{f_1}
\bY
\xrightarrow{f_2}
\bX
\xrightarrow{f_3}
M\STop
\end{equation}

First recall that $M\STop$ has as $k$-th space the Thom space $T(\STop(k))$.  The
$\Sigma_k$ action on $T(\STop(k))$ is induced by the conjugation  action on
$\STop(k)$.

For the construction of $\bX$ we need some facts about multisimplicial 
sets.  Given a space $Z$ and $k\geq 1$, let $S_\bu^{k\text{-multi}}(Z)$ be 
the $k$-fold multisimplicial set whose simplices in multidegree $\bfn$ are 
the maps $\Delta^\bfn\to Z$.  There is a natural map 
\[
|S_\bu^{k-\mathrm{multi}}(Z)|\to Z
\]
(where $|\ |$ denotes realization of the underlying multisimplicial set)
which is a weak equivalence by \cite{MO124837}
and the multisimplicial analogue of \cite[Lemma A.5]{MR0353298}.  If $Z$ is a
based space, there are natural maps 
\[
\lambda:
\Sigma|S_\bu^{k-\mathrm{multi}}(Z)|
\to
|S^1\wedge S_\bu^{k-\mathrm{multi}}(Z)|
\]
and 
\[
\kappa:
S^1\wedge S_\bu^{k-\mathrm{multi}}(Z)
\to
S_\bu^{(k+1)-\mathrm{multi}}(\Sigma Z)
\]
defined as follows.  Given $t\in[0,1]$, $u\in\Delta^{\bfn}$, and
$g:\Delta^{\bfn}\to Z$, let $\bar{t}$ denote the image of $t$ under the
oriented affine homeomorphism $[0,1]\to \Delta^1$, and define 
\[
\lambda(t\wedge [u,g])
=
[(\bar{t},u),s\wedge g],
\]
where $s$ is the nontrivial simplex of $S^1$.  Define 
\[
\kappa(s\wedge g)(\bar{t},u)=t\wedge g(u).
\]
Then the diagram
\begin{equation}
\label{m97}
\xymatrix
{
\Sigma|S_\bu^{k-\mathrm{multi}}(Z)|
\ar[d]
\ar[r]^-\lambda
&
|S^1\wedge S_\bu^{k-\mathrm{multi}}(Z)|
\ar[d]_{|\kappa|}
\\
\Sigma Z
&
|S_\bu^{(k+1)-\mathrm{multi}}(\Sigma Z)|
\ar[l]
}
\end{equation}
commutes.

Now let $X_k=|S_\bu^{k-\mathrm{multi}}(T(\STop(k)))|$. 
We define the $\Sigma_k$ action on $X_k$ as follows.  For $\alpha\in\Sigma_k$
and $g:\Delta^\bfn\to T(\STop(k)))$, let
$\alpha(\bfn)=(n_{\alpha^{-1}(1)},\ldots,n_{\alpha^{-1}(k)})$
and let $\alpha(g)$ be the
composite
\[
\Delta^{\alpha(\bfn)}
\xrightarrow{\alpha^{-1}}
\Delta^{\bfn}
\xrightarrow{g}
T(\STop(k))
\xrightarrow{\alpha}
T(\STop(k)).
\]
This makes $S_\bu^{k-\mathrm{multi}}(T(\STop(k)))$ an object of 
$\Sigma_k\ssk$, and now Proposition \ref{m98} gives the $\Sigma_k$ action on
$X_k$.
Next define the structure map
\[
\Sigma X_k\to X_{k+1}
\]
to be the composite
\begin{multline*}
\Sigma|S_\bu^{k-\mathrm{multi}}(T(\STop(k)))|
\xrightarrow{\lambda}
|S^1\wedge S_\bu^{k-\mathrm{multi}}(T(\STop(k)))|
\\
\xrightarrow{|\kappa|}
|S_\bu^{(k+1)-\mathrm{multi}}(\Sigma T(\STop(k)))|
\to
|S_\bu^{(k+1)-\mathrm{multi}}(T(\STop(k+1)))|,
\end{multline*}
where the last map is induced by the structure map of $M\STop$.
Let $\bX$ be the symmetric spectrum consisting of the spaces $X_k$ with these
structure maps.
Define $f_3:\bX\to M\STop$ to be the sequence of weak equivalences
\[
|S_\bu^{k-\mathrm{multi}}(T(\STop(k)))|
\to
T(\STop(k)).
\]
The commutativity of diagram \eqref{m97} shows that $f_3$ is a map of symmetric
spectra.

Next let 
$S_\bu^{k-\mathrm{multi,}{\mathord{\pitchfork}}}(T(\STop(k)))$ be the
sub-multisemisimplicial set of $S_\bu^{k-\mathrm{multi}}(T(\STop(k)))$
consisting of maps whose restrictions to each face of $\Delta^{\bfn}$ are
transverse to the zero section (see \cite{MR1201584} for topological transversality).  Let $\bY$ be the subspectrum of $\bX$ with
$k$-th space $|S_\bu^{k-\mathrm{multi,}{\mathord{\pitchfork}}}(T(\STop(k)))|$,
and let $f_2:\bY\to\bX$ be the inclusion.

\begin{lemma}
\label{m104}
$f_2$ is a weak equivalence.
\end{lemma}

\begin{proof}
Since $S_\bu^{k-\mathrm{multi,}{\mathord{\pitchfork}}}(T(\STop(k)))$ and
$S_\bu^{k-\mathrm{multi}}(T(\STop(k)))$ satisfy the multi-Kan condition, they
have compatible degeneracies by Proposition \ref{m32}. It therefore suffices to
show that the inclusion 
$
S_\bu^{k-\mathrm{multi,}{\mathord{\pitchfork}}}(T(\STop(k)))
\subset
S_\bu^{k-\mathrm{multi}}(T(\STop(k)))
$
induces a weak equivalence on the diagonal semi-simplicial sets, and this
follows from 
\cite[Section
9.6]{MR1201584} and 
the definition of homotopy groups (\cite[Definition 
3.6]{MR1206474}). 
\end{proof}

It remains to construct $f_1$.  Let $S\subset T(\STop(k))$ be the zero 
section. First we observe that, if $g:\Delta^\bfn\to T(\STop(k))$ is a map 
whose restriction to each face is transverse to $S$, 
we obtain 
an element $F\in\ad_\STop(\Delta^\bfn)$ by letting 
$F(\sigma,o) $ be $g^{-1}(S)\cap \sigma$ with the orientation determined by $o$.
This construction gives a map 
\[
S_\bu^{k-\mathrm{multi,}{\mathord{\pitchfork}}}(T(\STop(k)))
\to
(\bR_\STop)_k,
\]
in $\Sigma_k\ssk$, and applying geometric realization gives 
a $\Sigma_k$ equivariant map
$Y_k\to (\bM_\STop)_k$; we let $f_1$ be the sequence of these maps.


\begin{lemma}
$f_1$ is a weak equivalence.
\end{lemma}

\begin{proof}
For a $k$-fold multisemisimplicial set $A$, let $A'$ be the semisimplicial set
whose $n$-th set is $A_{0,\ldots,0,n}$.  There is an evident map 
\[
\phi:|A'|\to|A|.
\]
If $A$ is $S_\bu^{k-\mathrm{multi}}(Z)$, then $A'$ is $S_\bu(Z)$, and if $A$
is $R_k$ then $A'$ is the semisimplicial set $P_k$ of \cite[Definition
15.4(i)]{LM12}, with realization $(\bQ_\STop)_k$ (\cite[Definitions 
15.4(ii) and 15.8]{LM12}).  Now we have a commutative diagram 
\[
\xymatrix{
(\bM_\STop)_k
&
Y_k
\ar[l]_-{f_1}
\ar[r]^-{f_2}
&
X_k
\ar[r]^-{f_3}
&
T(\STop(k))
\\
(\bQ_\STop)_k
\ar[u]_{\phi_1}
&
|S_\bu^{{\mathord{\pitchfork}}}(T(\STop(k)))|
\ar[l]_-{g_1}
\ar[r]^-{g_2}
\ar[u]_{\phi_2}
&
|S_\bu(T(\STop(k)))|
\ar[r]^-{g_3}
&
T(\STop(k))
\ar[u]_=
}
\]
Here $g_3$ is the usual weak equivalence, and
$g_2$ is a weak equivalence by \cite[Section
9.6]{MR1201584} and
the definition of homotopy groups (\cite[Definition
3.6]{MR1206474}), so $\phi_2$ is a weak equivalence. $g_1$ was shown to be a
weak equivalence in \cite[Appendix B]{LM12}, and $\phi_1$ was shown to be a
weak equivalence in \cite[Section 15]{LM12}, so $f_1$ is a weak equivalence as
required.
\end{proof}

This completes the construction of diagram \eqref{m94}.

Next we recall that $M\STop$ is a commutative symmetric ring spectrum with
product
\[
T(\STop(k))\wedge T(\STop(l))
\to T(\STop(k+l)).
\]
$\bX$ is also a commutative symmetric ring spectrum, with the
product
\begin{multline*}
|S_\bu^{k-\mathrm{multi}}(T(\STop(k)))|
\wedge
|S_\bu^{l-\mathrm{multi}}(T(\STop(l)))|
\\
\to
|S_\bu^{(k+l)-\mathrm{multi}}(T(\STop(k))\wedge T(\STop(l)))|
\to
|S_\bu^{(k+l)-\mathrm{multi}}(T(\STop(k+l)))|,
\end{multline*}
and $\bY$ is a commutative symmetric ring spectrum with the product it inherits
from $\bX$.  The maps $f_2$ and $f_3$ are maps of symmetric ring spectra, so 
to complete the proof of Theorem \ref{m96}, it suffices to show

\begin{lemma}
\label{m99}
$(\bM_\STop)^{\mathrm{comm}}$ and $\bY$ are isomorphic in the homotopy
category of commutative symmetric ring spectra.
\end{lemma}

The proof of Lemma \ref{m99} is outsourced to Appendix \ref{Lemma m99}. The results of Appendix \ref{Lemma m99} use material from sections 17 and 18.

\section{Relaxed symmetric Poincar\'e complexes}
\label{relaxed}

For a commutative ring $R$ with the trivial involution,  we would like to
apply Theorem \ref{Nov12.3} to obtain a commutative model for the symmetric 
L-spectrum of $R$.  However, the ad theory $\ad^R$ defined in Section 9 of 
\cite{LM12}, with the product defined in \cite[Definition 9.12]{LM12}, is not 
commutative. The difficulty is that this product is defined using a
noncommutative coproduct
\[
\Delta:W\to W\otimes W
\]
for the standard resolution $W$ of $\Z$ by $\Z [\Z /2]$-modules.
In this section and the next we give an equivalent ad theory which is
commutative.

Fix a ring $R$ with involution. For a complex $C$ of left $R$-modules, let $C^t$ be the complex of right $R$-modules obtained from $C$ by applying the involution of $R$. As usual, give $C^t\otimes_RC$ the $\Z/2$-action which switches the factors. Write $\D$  for the category of homotopy finite chain complexes
as in \cite[Definition 9.2(v)]{LM12}.

\begin{definition}  A {\it relaxed quasi-symmetric complex of dimension $n$} 
is a quadruple $(C,D,\beta,\varphi)$, where $C$ is an object of $\D$, $D$ is an
object of $\D$ with a $\Z/2$ action, $\beta$ is a quasi-isomorphism
$C^t\otimes_R C\to D$ which is also a $\Z/2$ equivariant map, and $\varphi$ 
is an element of $D_n^{\Z/2}$.
\end{definition}

\begin{example}
\label{m44}
(i)
If $(C,\varphi)$ is a quasi-symmetric complex as defined in \cite[Definition
9.3]{LM12}, then the quadruple $(C,(C^t\otimes_R C)^W,\beta,\varphi)$ is a 
relaxed quasi-symmetric complex, where $\beta:C^t\otimes_R C\to 
(C^t\otimes_R C)^W$ 
is induced by the augmentation $W\to \Z$.

(ii) Relaxed quasi-symmetric complexes arise naturally from the construction of
the symmetric signature of a Witt space given in \cite{GBF31}; see
\cite{BLM12}.
\end{example}

\begin{definition}
We define a category $\A_{\text{rel}}^R$ (the rel stands for relaxed) as 
follows.  The objects of $\A_{\text{rel}}^R$ are the relaxed quasi-symmetric
complexes.  A morphism $(C,D,\beta,\varphi)\to (C',D',\beta',\varphi')$ is a
pair $(f:C\to C',g:D\to D')$, where $f$ and $g$ are $R$-linear chain 
maps, $g$ is $\Z/2$ equivariant, 
$g\beta=\beta'(f\otimes f)$, 
and (if $\dim \varphi=\dim\varphi'$)
$g_*(\varphi)=\varphi'$.
\end{definition}

\begin{definition}
\label{quasi}
A morphism $(f,g):(C,D,\beta,\varphi)\to (C',D',\beta',\varphi')$ between 
objects of the same dimension is a {\it quasi-isomorphism} if $f$ (and hence
also $g$) is a quasi-isomorphism.
\end{definition}

$\A_{\text{rel}}^R$ is a  $\Z$-graded 
category, where $d$ was defined above, $i$ takes $(C,D,\beta, \varphi)$ 
to $(C,D,\beta, -\varphi)$
and $\emptyset_n$ is the $n$-dimensional object 
for which $C$ and $D$ are zero in all degrees. Since the set of morphisms between objects of different dimensions is independent of the chains $\varphi$ the category $\A_{\text{rel}}^R$ is a balanced in the obvious way (\cite[Definition 5.1]{LM12}).

\begin{remark}
\label{m45}
The construction of Example \ref{m44}(i) gives a morphism 
\[
\A^R\to \A_{\text{rel}}^R
\]
of $\Z$-graded categories. 
\end{remark}

Next we must say what the $K$-ads with values in $\A_{\text{rel}}^R$ are.  We
need some preliminary definitions and a lemma.  For a balanced pre $K$-ad $F$ 
we will use the notation
\[
F(\sigma,o)=(C_\sigma,D_\sigma,\beta_\sigma,\varphi_{\sigma,o}).
\]

Recall \cite[Definition 9.7]{LM12}.

\begin{definition}
A balanced pre $K$-ad $F$ is 
{\it well-behaved} if $C$ and $D$ are well-behaved.
\end{definition}

Next recall \cite[Definition 12.2]{LM12}.  

\begin{lemma}
\label{m41}
Let $F$ be a well-behaved pre $K$-ad. Then

\rm{(i)} $C^t\otimes_R C$ is well-behaved, and

\rm{(ii)} the map
\[
(\beta_\sigma)_*:
H_*((C^t\otimes_R C)_\sigma/(C^t\otimes_R C)_{\partial\sigma})
\to
H_*(D_\sigma/D_{\partial\sigma})
\]
is an isomorphism for each $\sigma$.
\end{lemma}

\begin{proof}
For part (i), first recall (by \cite[Definitions 9.7(b) and 9.6(ii)]{LM12}) that $C$ takes morphism to cofibrations, that is, split monomorphisms in each degree. Moreover, the canonical map from
$C_{\partial \sigma}$ to $C_\sigma$ is a cofibration for each $\sigma$. 
Let us fix a cell $\sigma$ of  $K$ for the rest of the proof. For each degree $n$, we will construct a set $S_\tau$ and a basis $(b_s)_{s\in S_\tau}$ of $C_\tau$ by induction for all $\tau \subset \sigma$ in a functorial way. We will omit the degree from the notation for this part. For points $\tau$, we simply choose a basis. These add up to a basis of $C_{\partial \rho}$ for 1-cells $\rho$. Given a cell $\tau$  of dimension $k$  we may suppose that we have already constructed a basis of $C_{\partial \tau}$. Since $C_{\partial \tau}\ra C_\tau $ is split injective and all modules, including the quotient module,  are free we may extend this basis to a basis of $C_\tau$. This way, we constructed a basis of $C_{\partial \rho}$ for all $\rho$ of dimension $k+1$ as well: the set of all $b_s$ with $s\in \bigcup_{\tau \subsetneq \rho} S_\tau$ gives a basis because the free functor commutes with colimits. 
\par
We have to show that the map
$$ (C^t\otimes_RC)_{\partial \sigma}=\colim_{\tau \subsetneq \sigma}  C^t_\tau \otimes_R C_\tau \lra C_\sigma^t \otimes C_\sigma $$
is a cofibration. A basis of the target in degree $n$ is indexed by the union $\bigcup_{p+q=n}S_{\sigma,p }\times S_{\sigma,q}$ where we now have to take care of the different degrees in the notation. A basis for the source is the subset given by the union of all pairs coming from $S_{\tau,p} \times S_{\tau,q}$ for $\tau \subsetneq \sigma$. Since the map is induced by the free functor it is split injective. 

For part (ii), first observe that the fact that $C^t\otimes_R C$ and $D$ are
well-behaved implies that they are
Reedy cofibrant (\cite[Definition 15.3.3(2)]{MR1944041}).
The colim that defines $(C^t\otimes_R C)_{\partial\sigma}$ is a hocolim 
by \cite[Theorem 19.9.1(1) and Proposition 15.10.2(2)]{MR1944041},
and similarly for $D_{\partial\sigma}$, so the map 
\[
(\beta_\sigma)_*:H_*((C^t\otimes_R C)_{\partial\sigma})\to
H_*(D_{\partial\sigma})
\]
is an isomorphism by \cite[Theorem 19.4.2(1)]{MR1944041}, and this implies the 
lemma.
\end{proof}

Recall \cite[Example 3.12 ]{LM12}.

\begin{definition}
A balanced pre $K$-ad
$F$ is {\it closed} if, for each $\sigma$, the map from the cellular chains 
$ \cl(\sigma)$ to $D_\sigma$
which takes $\langle \tau,o\rangle$ to $\varphi_{\tau,o}$
is a chain map.
\end{definition}

Note that if $F$ is closed then
$\varphi_{\sigma,o}$ represents an element $[\varphi_{\sigma,o}]\in 
H_*(D_\sigma/D_{\partial\sigma})$. 

\begin{notation}
\label{ff11}
For a balanced pre $K$-ad $F$ and a cell $\sigma$ of $K$, let
\[
j_\sigma:
(C^t\otimes_R C)_\sigma/(C^t\otimes_R C)_{\partial\sigma}
\to
(C_\sigma/C_{\partial\sigma})^t\otimes_R C_\sigma
\]
be the map that takes $[c \otimes c']$ to $[c]\otimes c'$.
\end{notation}

The next definition is a little more complicated than the corresponding
definition in \cite[Section 9]{LM12} in order to satisfy the extra
condition in Definition \ref{comad}.

\begin{definition}
\label{m47}
(i) 
A {\it balanced $K$-ad} is a 
pre $K$-ad $F$ with the following properties:

(a) it is balanced, well-behaved and closed, and

(b) for each $\sigma$ the 
slant product with 
$(j_\sigma)_*(\beta_\sigma)_*^{-1}([\varphi_{\sigma,o}])$
is an isomorphism
\[
H^*(\Hom_R(C_\sigma,R))
\to
H_{\dim \sigma-\deg F-*}(C_\sigma/C_{\partial\sigma}).
\]

(ii)A $K$-ad is a pre $K$-ad which is naturally quasi-isomorphic to a balanced
$K$-ad.

We mention that, by the definition of an ad theory,  a $(K,L)$-ad is just a $(K,L)$-pread which is a $K$-ad.
\end{definition}
Write $\ad_{\text{rel}}^R$ for the set of $K$ ads with values in
$\A_{\text{rel}}^R$.

\begin{remark}
\label{m46}
The morphism of Remark \ref{m45} takes ads to ads.
\end{remark}

\begin{thm}
\label{m42}
$\ad_{\mathrm{rel}}^R$ is an ad theory.
Moreover, 
when $R$ is commutative with the trivial involution,
$\ad_{\text{rel}}^R$ is a commutative ad theory.
\end{thm}

For the proof of Theorem \ref{m42} we need a product operation.

\begin{definition}
(i) For $i=1,2$, let $R_i$ be a ring with involution and let 
$(C^i,D^i,\beta^i,\varphi^i)$
be an object of
$\A_{\text{rel}}^R$.  Define
\[
(C^1,D^1,\beta^1,\varphi^1)\otimes 
(C^2,D^2,\beta^2,\varphi^2)
\]
to be the following object of $\A_{\text{rel}}^{R_1\otimes R_2}$:
\[
(C^1\otimes C^2,D^1\otimes D^2, \gamma, \varphi^1\otimes\varphi^2),
\]
where $\gamma$ is the composite
\begin{multline*}
(C^1\otimes C^2)^t\otimes_{R_1\otimes R_2}\,(C^1\otimes C^2)
\cong 
((C^1)^t\otimes_{R_1} C^1) \otimes ((C^2)^t\otimes_{R_2} C^2)
\\
\xrightarrow{\beta^1\otimes \beta^2}
D^1\otimes D^2.
\end{multline*}

(ii) For $i=1,2$, suppose given
a ball complex $K_i$ and a
pre $K_i$-ad $F_i$ of degree $k_i$ with values in $\A_{\text{rel}}^{R_i}$.
Define a pre $(K_1\times K_2)$-ad $F_1\otimes F_2$ with values in
$\A_{\text{rel}}^{R_1\otimes R_2}$ by
\[
(F_1\otimes F_2)(\sigma\times\tau,o_1\times o_2)
=
i^{k_2\dim \sigma} F_1(\sigma,o_1)\otimes F_2(\tau,o_2).
\]
\end{definition}

\begin{lemma}
\label{m43}
For $i=1,2$, suppose given a ball complex $K_i$ and a
$K_i$-ad $F_i$ with values in $\A_{\mathrm{rel}}^{R_i}$.
Then $F_1\otimes F_2$ is a $(K_1\times K_2)$-ad. \qed
\end{lemma}

\begin{proof}[Proof of \ref{m42}.]
We only need to verify parts (d),(f) and (g) of \cite[Definition 3.10]{LM12}.

For part (d),  we have to show that a pre $K$-ad is a $K$-ad if it restricts to a $\sigma$-ad for each closed cell $\sigma$ of $K$.  It suffices to consider the case of a $K$-pread $F$ with $K=L\sqcup_{\partial \sigma } \sigma$  whose restriction to $L$ is naturally quasi-isomorphic to a balanced $L$-ad $G$ and to a balanced $\sigma$-ad $H$ on $\sigma$. We have to give a quasi-isomorphism of $F$  to a balanced $K$-ad $I$. 
In the following, we will concentrate on the first complex in the datum of a pread and we will use the same letter for this complex. The second complex and all other entries will the be clear then.
Define the restriction of $I$ to $L$ be $G$. It remains to define $I_\sigma$, a map to $F_\sigma$ and the maps from its lower dimensional cells. 
For each $\tau\subset \partial \sigma$ we are given a quasi isomorphism $g_\tau $ from $G_\tau$ to  $F_\tau$ and an $h_\tau$ from $H_\tau$ to $F_\tau$. Since $G_\tau$ and $H_\tau$ are cofibrant we find a map $f_\tau:G_\tau\ra H_\tau$ such that $h_\tau f_\tau$ is homotopic to $g_\tau$. This only uses the fact that isomorphisms in the homotopy category between cofibrant objects can be represented by chain maps up to homotopy. Similarly, using the fact that 
 the restriction of $G$ to $\partial \sigma$ is balanced we find a map $f_{\partial \sigma}:G_{\partial \sigma} \ra H_{\partial \sigma}$ whose restriction to $\tau\subset \partial \sigma$ is homotopic to $H_{\tau \subset \partial \sigma}f_\tau$. We also find a system of compatible homotopies, that is,   a homotopy between  $h_{\partial \sigma}  f_{\partial \sigma}$ and $g_{\partial \sigma}$. Let $I_\sigma$ be the mapping cylinder of 
$$H_{\partial \sigma \subset \sigma}f_{\partial \sigma} : G_{\partial\sigma } \lra H_\sigma.$$ Then the constructed homotopy and the map $h_\sigma$  complete the required quasi-isomorphism $I_\sigma\ra F_\sigma$. The other maps are obvious.

\par

For part (f), 
let $F$ be a $K'$-ad.  
We may assumed it is balanced.  
Let
\[
F(\sigma,o)=(C_\sigma,D_\sigma,\beta_\sigma,\varphi_{\sigma,o}).
\]
We need to define a $K$-ad $E$ which agrees with $F$ on each
residual subcomplex of $K$. 
As in the proof of \cite[Theorem 6.5]{LM12},
we may assume by induction that $K$ is a ball complex structure for the $n$
disk with one $n$ cell $\tau$, and
that $K'$ is a subdivision of $K$ which agrees with $K$ on the boundary.
We only need to define $E$ on the top cell $\tau$ of $K$.  We define
$E(\tau,o)$ to be $(C_\tau, D_\tau, \beta_\tau, \varphi_{\tau,o})$, where
\begin{itemize}
\item
$C_\tau=\colim_{\sigma\in K'} C_\sigma$, 
\item
$D_\tau=\colim_{\sigma\in K'} D_\sigma$, 
\item
$\beta=\colim_{\sigma\in K'} 
\beta_\sigma$,
and
\item
$
\varphi_{\tau,o}=\sum \varphi_{\sigma,o'}
$,
where $(\sigma,o')$ runs through the $n$-dimensional cells of $K'$ with
orientations induced by $o$.
\end{itemize}

The fact that $E$ satisfies part (a) of Definition \ref{m47} is a consequence 
of \cite[Proposition A.1(ii)]{LM12}.  We will deduce the isomorphism in 
part (b) of Definition \ref{m47} from \cite[Proposition 12.4]{LM12}, and for
this we need some facts from \cite[Section 12]{LM12}.

First recall that
for a well-behaved functor
\[
B:\Cell^\flat(K')\to\D,
\]
we write 
\[
\Nat(\cl,B)
\]
for the chain complex
of natural transformations of graded abelian groups; the differential is given
by
\[
\partial(\nu)=\partial\circ \nu-(-1)^{|\nu|}\nu\circ \partial.
\]

Recall \cite[Definition 12.3]{LM12} and also the map $\Phi$ defined just
before the statement of \cite[Lemma 12.6]{LM12}.
Consider the diagram
\[
\xymatrix{
H_*(\Nat(\cl,D))
\ar[r]^-\Phi
&
H_{*+n}(D_\tau,D_{\partial\tau})
\\
H_*(\Nat(\cl,C^t\otimes_R C))
\ar[u]^\beta
\ar[r]^-\Phi
&
H_{*+n}((C^t\otimes_R C)_\tau,(C^t\otimes_R C)_{\partial\tau}).
\ar[u]^\beta
}
\]
The horizontal maps are isomorphisms by
\cite[Lemma 12.6]{LM12}, and the right-hand vertical map is an isomorphism by 
the proof of Lemma \ref{m41}(ii).  Hence the left-hand vertical map is an
isomorphism. 

The collection $\{\varphi_{\sigma,o}\}$ gives a cycle $\nu$ in
$\Nat(\cl,D)$.  Let $\mu\in \Nat(\cl,C^t\otimes_R C)$ be a representative for
$\beta^{-1}([\nu])$.  Now fix an orientation $o$ for $\tau$.
Let $\psi\in C^t_\tau\otimes_R C_\tau$ be
$\sum \mu(\langle\sigma,o' \rangle)$,
where $(\sigma,o')$ runs through the $n$-dimensional cells of $K'$ with
orientations induced by $o$.  
Then $\psi$ is a representative of 
$(j_\sigma)_*\beta^{-1}([\varphi_{\tau,o}])$, so
it suffices to show that the cap product with
$\psi$ is an isomorphism $H^*(\Hom_R(C_\tau,R)) \to H_{n-\deg 
F-*}(C_\tau/C_{\partial\tau})$, and this follows from 
\cite[Proposition 12.4]{LM12}.

It remains to verify part (g) of \cite[Definition 3.10]{LM12}.  
Let $0,1,\iota$ denote the three cells of the unit
interval $I$, with their standard orientations.  
As in the proof of \cite[Theorem 9.11]{LM12}, it suffices to construct 
a relaxed symmetric Poincar\'e $I$-ad $H$ over $\Z$ which takes both 0 and 1 
to the object $(\Z,\Z,\gamma,1)$, where $\gamma$ is the isomorphism 
$\Z\otimes \Z\to \Z$.  The proof of \cite[Theorem 9.11]{LM12} gives
a symmetric Poincar\'e $I$-ad $G$ with 
$G(0)=G(1)=(\Z,\epsilon)$, where $\epsilon:W\to \Z\otimes \Z$ is the 
composite of the augmentation with $\gamma^{-1}$.
Let us denote the object $G(\iota)$ by $(C,\varphi)$.
Applying Remark \ref{m46} to $G$ gives a relaxed symmetric Poincar\'e $I$-ad 
$G'$ with $G'(\iota)=(C,(C\otimes C)^W,\beta,\varphi)$, where $\beta$ is
induced by the augmentation. 
Let $e_0$ (resp., $e_1$) be the inclusion 
$0\hookrightarrow\iota$ (resp., $1\hookrightarrow \iota$) and for $i=0,1$ let 
$g_i=G(e_i):\Z\to C$.  Then
\[
\partial\varphi=
(g_1\otimes g_1)\circ \epsilon-(g_0\otimes g_0)\circ\epsilon
\]
because $G$ is closed.
We can therefore construct the required $I$-ad $H$ from $G'$ by replacing 
$G'(0)$ and $G'(1)$ by $(\Z,\Z,\gamma,1)$.
\end{proof}

\section{Equivalence of the spectra associated to $\ad^R$ and 
$\ad_{\text{\rm{rel}}}^R$}
\label{eq}

By Remark \ref{m46}, 
the morphism
\[
\A^R\to \A_{\text{rel}}^R
\]
of Remark \ref{m45} induces a map of spectra
\[
\bQ^R\to \bQ_{\text{rel}}^R
\]
(see \cite[Section 15]{LM12})
and a map of symmetric spectra
\[
\bM^R\to \bM_{\text{rel}}^R
\]
(see \cite[Section 17]{LM12}).

\begin{thm}
\label{m55}
The maps 
\[
\bQ^R\to \bQ_{\mathrm{rel}}^R
\]
and 
\[
\bM^R\to \bM_{\mathrm{rel}}^R.
\]
are weak equivalences.
\end{thm}

\begin{remark}
The method that will be used to prove Theorem \ref{m54} can be used to show 
that 
$\bM^R\to \bM_{\text{rel}}^R$ is weakly equivalent to a map of symmetric ring 
spectra.
\end{remark}

Recall \cite[Definitions 4.1 and 4.2]{LM12}.
By \cite[Theorem 16.1, Remark 14.2(i) and Corollary 17.9(iii)]{LM12},
Theorem \ref{m55} follows from

\begin{prop}
\label{m48}
The map of bordism groups 
\[
\Omega^R_*\to(\Omega_{\mathrm{rel}}^R)_*
\]
is an isomorphism.
\end{prop}

The proof of Proposition \ref{m48} will occupy the rest of this section.
The following lemma proves surjectivity.
As usual, for a chain complex $A$ with a $\Z/2$ action, we write $A^{h\Z/2}$
for $(A^W)^{\Z/2}$.  The augmentation induces a map $A^{\Z/2}\to A^{h\Z/2}$.

\begin{lemma}
\label{m53}
Let
\[
(C,D,\beta,\varphi)
\]
be a relaxed symmetric Poincar\'e $*$-ad and let 
\[
\psi\in (C^t\otimes_R C)^{h\Z/2}
\]
represent the image of $\varphi$ under the map
\[
H_*(D^{\Z/2})\to 
H_*(D^{h\Z/2})
\xleftarrow{\cong}
H_*((C^t\otimes_R C)^{h\Z/2})
\]
(where the isomorphism is induced by $\beta$).
Then $(C,\psi)$ is a symmetric Poincar\'e $*$-ad, and $(C,D,\beta,\varphi)$ is
bordant to
\[
(C,(C^t\otimes_R C)^W,\gamma,\psi),
\]
where $\gamma$ is induced by the augmentation. 
\end{lemma}

For the proof we need another lemma.

\begin{lemma}
\label{m52}
Let $(C,D,\beta,\varphi)$ be a relaxed symmetric Poincar\'e $*$-ad.

{\rm (i)} If $\psi\in D^{\Z/2}$ is any representative for the homology class
$[\varphi]\in H_*(D^{\Z/2})$ then
$(C,D,\beta,\psi)$ is bordant to $(C,D,\beta,\varphi)$.

{\rm (ii)}
If 
\[
(f,g):(C,D,\beta,\varphi)\to (C',D',\beta',\varphi')
\]
is a map of $*$-ads of the same dimension for which $f$ (and hence also $g$) 
is a quasi-isomorphism then $(C,D,\beta,\varphi)$ and 
$(C',D',\beta',\varphi')$ are bordant.
\end{lemma}

The proof of Lemma \ref{m52} is deferred to the end of the section.

\begin{proof}[Proof of Lemma \ref{m53}]
The fact that $(C,\psi)$ satisfies \cite[Definition 9.9]{LM12} (only part (b)
is relevant) is immediate from Definition \ref{m47}(b).

To see that 
$(C,D,\beta,\varphi)$ and $(C,(C^t\otimes_R C)^W,\gamma,\psi)$ are bordant, let 
$\delta$ denote the composite 
\[
C^t\otimes_R C\xrightarrow{\beta} D\to D^W,
\]
let $\omega\in D^{h\Z/2}$ be the image of $\varphi$,
and let $\omega'\in D^{h\Z/2}$ be the image of $\psi$ under the map
$(C\otimes C)^{h\Z/2}\to D^{h\Z/2}$ induced by $\beta$.
Part (ii) of Lemma \ref{m52} shows that
$ (C,D,\beta,\varphi)$ and $(C,D^W,\delta,\omega)$ are bordant, and also that
$(C,(C^t\otimes_R C)^W,\gamma,\psi)$ and $(C,D^W,\delta,\omega')$
are bordant.  But $[\omega]=[\omega']$ in $H_*(D^{h\Z/2})$, so the result 
follows from part (i) of Lemma \ref{m52}.
\end{proof}

Next we show that the map in Proposition \ref{m48} is 1-1.  
So let $(C_0,\varphi_0)$ and
$(C_1,\varphi_1)$ be symmetric Poincar\'e $*$-ads and let $F$ be a relaxed
symmetric Poincar\'e bordism between them.  
Let $0,1,\iota$ denote the three cells of the unit
interval $I$, with their standard orientations.
Denote the object $F(\iota)$ by $(C,D,\beta,\varphi)$. 
It suffices to show that there is a symmetric Poincar\'e  $I$-ad $G$ with
\begin{equation}
\label{m49}
G(0)=(C_0,\varphi_0),\quad 
G(1)=(C_1,\varphi_1),\quad 
G(\iota)=(C,\chi)
\end{equation}
for an element $\chi$ which we will now construct.

$\varphi$ represents an element 
\[
[\varphi]\in H_*(D^{\Z/2},(C_0^t \otimes_R
C_0)^{h\Z/2}\oplus (C_1^t \otimes_R C_1)^{h\Z/2}).
\]
The map $D\to D^W$ induced by the augmentation gives a map 
\begin{multline*}
H_*(D^{\Z/2},(C_0^t \otimes_R
C_0)^{h\Z/2}\oplus (C_1^t \otimes_R C_1)^{h\Z/2}).
\\
\to
H_*(D^{h\Z/2},
((C_0^t \otimes_R C_0)^W)^{h\Z/2}
\oplus 
((C_1^t \otimes_R C_1)^W)^{h\Z/2});
\end{multline*}
let $x$ be the image of $\varphi$ under this map.
The map $\beta:C^t \otimes_R C\to D$ gives an isomorphism
\begin{multline*}
(\beta^{h\Z/2})_*:H_*((C^t \otimes_R C)^{h\Z/2},(C_0^t \otimes_R
C_0)^{h\Z/2}\oplus (C_1^t \otimes_R C_1)^{h\Z/2}).
\\
\to
H_*(D^{h\Z/2},
((C_0^t \otimes_R C_0)^W)^{h\Z/2}
\oplus 
((C_1^t \otimes_R C_1)^W)^{h\Z/2});
\end{multline*}
let $y=(\beta^{h\Z/2})_*^{-1}(x)$.

\begin{lemma}
\label{m50}
The image of $y$ under the boundary map 
\begin{multline*}
H_*((C^t \otimes_R C)^{h\Z/2},(C_0^t \otimes_R
C_0)^{h\Z/2}\oplus (C_1^t \otimes_R C_1)^{h\Z/2})
\\
\xrightarrow{\partial}
H_{*-1}((C_0^t \otimes_R
C_0)^{h\Z/2}\oplus (C_1^t \otimes_R C_1)^{h\Z/2})
\end{multline*}
is $-[\varphi_0]+[\varphi_1]$.
\end{lemma}

Before proving this we conclude the proof of Proposition
\ref{m48}.  The lemma implies that there is a representative $\chi$ of $y$ with

\begin{equation}
\label{m51}
\partial\chi=-\varphi_0+\varphi_1.  
\end{equation}
It suffices to show that, with this choice of
$\chi$, the symmetric Poincar\'e pre $I$-ad $G$ given by equation \eqref{m49} 
is an ad. Equation \eqref{m51} says that $G$ is closed, and part (b) of
\cite[Definition 9.9]{LM12} follows from Definition \ref{m47}(b) and the fact
that the image of $[\chi]$ under the map
\begin{multline*}
H_*((C^t \otimes_R C)^{h\Z/2},(C_0^t \otimes_R
C_0)^{h\Z/2}\oplus (C_1^t \otimes_R C_1)^{h\Z/2})
\\
\to
H_*((C^t \otimes_R C)^W,(C_0^t \otimes_R
C_0)^W\oplus (C_1^t \otimes_R C_1)^W)
\\
\xleftarrow{\cong}
H_*(C^t \otimes_R C,(C_0^t \otimes_R
C_0)\oplus (C_1^t \otimes_R C_1))
\end{multline*}
is the same as the image of $[\varphi]$ under the map 
\begin{multline*}
H_*(D^{\Z/2},(C_0^t \otimes_R
C_0)^{h\Z/2}\oplus (C_1^t \otimes_R C_1)^{h\Z/2})
\\
\to
H_*(D,(C_0^t \otimes_R
C_0)^W\oplus (C_1^t \otimes_R C_1)^W)
\\
\xrightarrow{\beta_*^{-1}}
H_*(C^t \otimes_R C,(C_0^t \otimes_R
C_0)\oplus (C_1^t \otimes_R C_1)).
\end{multline*}
\qed

\begin{proof}[Proof of Lemma \ref{m50}]
We know that the image of $[\varphi]$ under the boundary map
\begin{multline*}
H_*(D^{\Z/2},(C_0^t \otimes_R
C_0)^{h\Z/2}\oplus (C_1^t \otimes_R C_1)^{h\Z/2})
\\
\xrightarrow{\partial}
H_{*-1}((C_0^t \otimes_R
C_0)^{h\Z/2}\oplus (C_1^t \otimes_R C_1)^{h\Z/2})
\end{multline*}
is 
$-[\varphi_0]+[\varphi_1]$,
so it suffices to show that for $i=0,1$ the maps 
\[
(C_i^t \otimes_R C_i)^{h\Z/2}
\to
((C_i^t \otimes_R C_i)^W)^{h\Z/2}
\]
induced by $D\to D^W$ and by $\beta$ give the same map in homology.  
If we think of these as maps 
\[
a_i,b_i:((C_i^t \otimes_R C_i)^W)^{\Z/2}
\to
((C_i^t \otimes_R C_i)^{W\otimes W})^{\Z/2}
\]
(with diagonal $\Z/2$ action on $W\otimes W)$ then $a_i$ and $b_i$ are 
induced by the maps
\[
e_1,e_2:W\otimes W\to W
\]
given by the augmentations on the two factors.
Now the $\Z/2$
equivariant map
\[
\Delta:W\to W\otimes W
\]
of \cite[page 175]{MR560997} has the property that $e_1\circ \Delta$ and 
$e_2\circ \Delta$ are both the identity map, so if 
\[
d:((C_i^t \otimes_R C_i)^{W\otimes W})^{\Z/2}
\to
((C_i^t \otimes_R C_i)^W)^{\Z/2}
\]
is the map induced by $\Delta$ then $d\circ a_i$ and $d \circ b_i$ are both the
identity map.  But $\Delta$ is a $\Z/2$ chain homotopy equivalence, so $d$ is a
homology isomorphism and it follows that $a_i$ and $b_i$ induce the same map in
homology as required.
\end{proof}

It remains to prove Lemma \ref{m52}.
Let $F$ be the cylinder of $(C,D,\beta,\varphi)$
(which was constructed in the last paragraph of the proof of Theorem 
\ref{m42}).  Then $F(0)$ and $F(1)$ are both
$(C,D,\beta,\varphi)$.  Write
\[
F(\iota)=(C_\iota,D_\iota,\beta_\iota,\varphi_\iota)
\]
and let
\[
(h,k):(C,D,\beta,\varphi)
\to
(C_\iota,D_\iota,\beta_\iota,\varphi_\iota)
\]
be the map $F(1)\to F(\iota)$.

For part (i), the hypothesis gives an element $\rho\in D^{\Z/2}$ with
$\partial\rho=\psi-\varphi$.  Let $\rho'\in D_\iota^{\Z/2}$ be the image of
$\rho$ under $k:D\to D_\iota$.  Define an $I$-ad 
$G$ by 
\[
G(0)=(C,D,\beta,\varphi),\quad
G(1)=(C,D,\beta,\psi),\quad
G(\iota)=(C_\iota,D_\iota,\beta_\iota,\varphi_\iota+\rho').
\]
Then $G$ is the desired bordism.

For part (ii), we first show that $(C,D,\beta,\varphi)$ is bordant to 
$(C,D',\beta_1,\varphi')$, where $\beta_1$ is the composite
\[
C^t \otimes_R C\xrightarrow{\beta} D\xrightarrow{g} D'
\]
The idea is to construct a suitable mapping cylinder.
Let $D_1$ be the pushout of the diagram
\[
\xymatrix{
D
\ar[r]^k
\ar[d]_g
&
D_\iota
\\
D'
&
}
\]
Let $\varphi_1$ be the image of $\varphi_\iota$ in $D_1$ and let 
$\beta_2$ be the composite
\[
C_\iota^t\otimes_R C_\iota\to D_\iota\to D_1.
\]
Define an $I$-ad $H$ by
\[
H(0)=(C,D,\beta,\varphi),\quad
H(1)=(C,D',\beta_1,\varphi'),\quad
H(\iota)=(C_\iota,D_1,\beta_2,\varphi_1).
\]
Then $H$ is the desired bordism.

To conclude the proof we show that $(C,D',\beta_1,\varphi')$ is bordant to 
$(C',D',\beta',\varphi')$.  Let $C_1$ be the pushout of the diagram
\[
\xymatrix{
C
\ar[r]^h
\ar[d]_f
&
C_\iota
\\
C'
&
}
\]
Let $(C'_\iota,D'_\iota,\beta'_\iota,\varphi'_\iota)$ be the cylinder of
$(C',D',\beta',\varphi')$.  Let $\beta_3$ be the map
\[
C_1^t\otimes_R C_1
\to
D'_\iota.
\]
Define an I-ad $H_1$ by 
\[
H_1(0)=(C,D',\beta_1,\varphi'),\quad
H_1(1)=(C',D',\beta',\varphi'),\quad
H_1(\iota)=(C_1,D'_\iota,\beta_3,\varphi'_\iota).
\]
Then $H_1$ is the desired bordism.
\qed

\section{The symmetric signature revisited}
\label{revis}


Fix a group $\pi$, a simply-connected free $\pi$-space $Z$,
and a homomorphism $w:\pi\to \{\pm 1\}$, and recall the symmetric spectrum
$\bM_{\pi, Z,w}$ (\cite[Sections 7 and 17]{LM12}) which represents
$w$-twisted Poincar\'e bordism over $Z/\pi$. 

Let $R$ denote the group ring $\Z[\pi]$ with the $w$-twisted involution
(\cite[page 196]{MR566491}).

We now restrict our attention to strict ball complexes. Recall from section \ref{redef} that a ball complex is {\it strict} if each
component of the intersection of two cells is a single cell.
In \cite[Section 10]{LM12} we gave a functor
\[
\Sig:\A_{\pi,Z,w}\to \A^R
\]
which induces a natural transformation
\[
\Sig: \ad_{\pi,Z,w}(K) \to \ad^R(K)
\]
for strict ball complexes $K$.

\begin{remark}
(i) The restriction to strict ball
complexes was not mentioned in \cite{LM12} but is necessary, because if $K$ 
has a cell $\tau$ whose boundary is not strict and if $F$ is a $K$-ad
$(X_\sigma, f_\sigma,\xi_{\sigma,o})$ then 
the map 
\begin{equation}
\label{aa11}
\colim_{\sigma\subset\partial\tau} S_*(X_\sigma)
\to
S_*(X_{\partial\tau})
\end{equation}
is not a monomorphism (because simplices with support in $\sigma\cap\sigma'$
but not in a cell of $\sigma\cap\sigma'$ will have two representatives in the 
colimit), and hence $\Sig\circ F$ is not well-behaved.

(ii) If $K$ is strict then the map \eqref{aa11} has a left inverse for all
$\tau$
(because its image is the subcomplex of $S_*(X_\tau)$ generated by the
simplices that land in some $X_\sigma$ with $\sigma\subset \partial\tau$, and
the left inverse takes each such simplex to a representative for it in the
colimit system; all such representatives are identified because $K$ is strict).
Hence the map \eqref{aa11} is the inclusion of a direct summand, as required
for $\Sig\circ F$ to be well-behaved.

(iii) The restriction to strict ball complexes does not affect the results
about the symmetric signature in
\cite{LM12} because the only ball complexes that occur in \cite[Sections 15,
17--19]{LM12} are products of simplices, and these are strict.
\end{remark}

In this section we give a functor
\[
\Sigr:\A_{\pi,Z,w}\to\A_{\text{rel}}^R
\]
which induces a natural transformation
\[
\Sigr:\ad_{\pi,Z,w}(K) \to 
\ad_{\text{rel}}^R(K)
\]
for strict ball complexes $K$.

Let $(X,f,\xi)$ be an object of $\A_{\pi,Z,w}$ (\cite[Definition 7.3]{LM12}).

In the special case where $\pi$ is the trivial group and $Z$ is a point, the
definition is easy:  
\[
\Sigr(X,f,\xi)=(S_*X,S_*(X\times X),\beta,\varphi),
\]
where $\beta$ is the cross product $S_*X\otimes S_*X\xrightarrow{\times} 
S_*(X\times X)$ and
$\varphi$ is the image of $\xi$ under the diagonal map.  

The definition in the general case is similar.
Recall that we write
$\widetilde{X}$ for the pullback of $Z$ along $f$ and $\Z^w$ for $\Z$ with the
right $R$ action determined by $w$.
Also recall
\cite[Definition 7.1]{LM12}.

\begin{definition}
\label{symsigdef}
{\rm (i)} Give $S_*(\widetilde{X})$ the left $R$ module structure determined by
the action of $\pi$ on $\widetilde{X}$ and give $S_*(\widetilde{X}\times
\widetilde{X})$ 
and
$S_*(\widetilde{X})\otimes
S_*(\widetilde{X})$
the left $R$ module structures determined by the diagonal actions
of $\pi$.

{\rm 
(ii)}
Define 
\[
\Sigr(X,f,\xi)= 
(S_*(\widetilde{X}),
\Z^w\otimes_R S_*(\widetilde{X}\times \widetilde{X}),
\beta,
\varphi),
\]
where $\beta$ is the composite
\[
S_*(\widetilde{X})^t\otimes_R
S_*(\widetilde{X})
\cong
\Z^w\otimes_R (S_*(\widetilde{X})\otimes S_*(\widetilde{X}))
\xrightarrow{1\otimes \times}
\Z^w\otimes_R S_*(\widetilde{X} \times \widetilde{X} )
\]
and $\varphi$ is the image of $\xi$ under the map
\[
S_*(X,\Z^f)
=
\Z^w \otimes_R S_*(\widetilde{X})
\to
\Z^w \otimes_R S_*(\widetilde{X}\times \widetilde{X})
\]
(where the unmarked arrow is induced by the diagonal map).
\end{definition}

\begin{remark}
(i) 
For set-theoretic reasons one should modify this definition as in 
\cite[Section 10]{LM12}; we leave this to the reader.

(ii)  For strict ball complexes, $\Sigr$ takes ads to ads, because the 
composite of the cross product with the Alexander-Whitney map is naturally 
chain homotopic to the map induced by the diagonal.
\end{remark}

Next we compare $\Sig$ to $\Sigr$.

Let us denote by $\delta$ both the map 
$\A^R\to \A_{\text{rel}}^R$
of Remark \ref{m45} and the map
$\bM^R\to \bM_{\text{rel}}^R$ which it induces. 

\begin{prop}
\label{homprop}
The diagram 
\[
\xymatrix{
&
\bM_{\pi, Z,w}
\ar[ld]_\Sig
\ar[rd]^\Sigr
&
\\
\bM^R
\ar[rr]^-\delta
&
&
\bM_{\mathrm{rel}}^R
}
\]
commutes in the homotopy category of symmetric spectra.
\end{prop}

We will derive this from a more general result.  Recall Definition
\ref{quasi}.

\begin{prop}
\label{m56}
Let $\Theta$ be an ad theory, let $R$ be a ring with involution, and let 
$S_1$ and $S_2$ be morphisms of ad theories $\Theta\to 
\ad_\mathrm{rel}^R$.  Suppose that there is a natural quasi-isomorphism
$\nu:S_1\to S_2$.  Then the maps $\bM_\Theta\to \bM_\mathrm{rel}^R$ induced by
$S_1$ and $S_2$ are homotopic.
\end{prop}

\begin{proof}[Proof of Proposition \ref{homprop}].
The extended Eilenberg-Zilber map
\[
W\otimes S_*(Y\times Z)
\to
S_*(Y)\otimes S_*(Z)
\]
(\cite[proof of Proposition 5.8]{GBF31})
gives a map
\[
S_*(Y\times Z)
\to
((S_*(Y)\otimes S_*(Z))^W,
\]
and this
gives a natural quasi-isomorphism
\[
\nu: \Sigr\to\delta\circ\Sig.
\]
\end{proof}

The rest of this section is devoted to the proof of Proposition \ref{m56}.
The basic idea is similar to the proof of Theorem \ref{Nov12.3}.

\begin{definition}
Let $P$ be the poset whose two elements are the functors $S_1$ and
$S_2$, with $S_1\leq S_2$.
\end{definition}

Recall Definition \ref{m22}(i), and note that $U(K)$ has a poset structure 
given by inclusions of cells.  Our next definition is analogous to
Definition \ref{m22}(ii).

\begin{definition}
Let $k\geq 0$, let $\bfn$ be a $k$-fold multi-index, and let
$F\in \pre_{\Theta}^k(\Delta^\bfn)$.
Let
\[
b:U(\Delta^{\bfn})\to P
\]
be a map of posets.

(i)
For an object $(\sigma,o)$ of $\Cell(\Delta^{\bfn})$ define
the object $ b_*(F)(\sigma,o) $ of $\A_{\text{rel}}^R$ to be
$ b(\sigma)(F(\sigma,o))$.

(ii)
For a morphism $f:(\sigma,o)\to (\sigma',o')$ of $\Cell(\Delta^{\bfn})$
define the morphism
\[
b_*(F)(f): b_*(F)(\sigma,o)\to b_*(F)(\sigma',o')
\]
to be
\[
\begin{cases}
S_1(f) & \text{if $b(\sigma',o')=S_1$},
\\
S_2(f) & \text{if $b(\sigma,o)=S_2$},
\\
\nu\circ S_1(f) &\text{otherwise}.
\end{cases}
\]

\end{definition}

\begin{lemma}
$b_*$ takes ads to ads.
\end{lemma}

\begin{proof}
Suppose $F$ is a $K$-ad in $\Theta$. Then $\nu$ provides a quasi-isomorphism $\kappa$ from $S_1(F)$ to $b_*(F)$: let $\kappa(F)_{(\sigma,o)}$ be the identity if $b(\sigma,o)=S_1$ and let it coincide with $\nu(F)_{(\sigma,o)}$ else. Since $S_1(F)$ is quasi-isomorphic to a balanced $K$-ad the same is true for $b_*(F)$. This implies that $b_*(F)$ is a $K$-ad.
\end{proof}

Recall that we write $\bR$ for the object of $\sss$ associated to an ad theory
(Example \ref{Aug29.1}).
Then $b_*$ gives a map 
\[
((\bR_\Theta)_k)_\bfn
\to
((\bR_{\text{rel}}^R)_k)_\bfn.
\]


\begin{definition}
(i) For $k\geq 0$ define
an object $\P_k$ of $\Sigma_k\ssk$ by
\[
(\P_k)_\bfn=\Map_{\text{posets}}(U(\Delta^\bfn),P)_+
\]
(where the $+$ denotes a disjoint basepoint);  the morphisms in
$(\Delta_\inj^{\text{op}})^{\times k}$ act in the evident way, and the
morphisms of the form $(\alpha,\mathrm{id})$ with $\alpha\in \Sigma_k$ act by
permuting the factors in $\Delta^\bfn$.

(ii) Define $\P$ to be the object of $\ss$ with $k$-th term
$\P_k$.
\end{definition}

Next we give $\P\owedge \bR_{\Theta}$ the structure of a multisemisimplicial
symmetric spectrum (cf.\ Definition \ref{m29}).  Recall Definition
\ref{m24} and let $s$ be the 1-simplex of $S^1$.
Define
\[
\omega:S^1\wedge (\P\owedge \bR_{\Theta})_k
\to
(\P\owedge \bR_{\Theta})
\]
as follows: for
$b\in (\P_k)_\bfn$ and $x\in
((\bR_{\Theta})_k)_\bfn$, let
\[
\omega(s\wedge (b\wedge x))
=
(b\circ \Pi)\wedge \omega(s\wedge x).
\]

It follows from the definitions that we obtain a map
\[
\beta:\P\owedge \bR_{\Theta}
\to 
\bR_{\text{rel}}^R
\]
in $\sss$ by 
\[
\beta(b\wedge F)=b_*(F).
\]

For each $k$ and $\bfn$, define elements $c_{k,\bfn},d_{k,\bfn}\in (\P_k)_\bfn$
to be the constant functions $U(\Delta^\bfn)\to P$ whose values are 
respectively $S_1$ and $ S_2$. Then define maps
\[
c,d:\boS\to \P
\]
in $\ss$
by taking the nontrivial simplex of $(\boS_k)_\bfn$ to $c_{k,\bfn}$, resp.,
$d_{k,\bfn}$. Finally, define maps
\[
{\mathbf c},
{\mathbf d}:
\bR_{\Theta}\to \P\owedge \bR_{\Theta}
\]
in $\sss$ by letting $\mathbf c$ be the composite
\[
\bR_{\Theta}\cong
\boS\owedge \bR_{\Theta}
\xrightarrow{c\wedge 1}
\P\owedge \bR_{\Theta}
\]
and similarly for $\mathbf d$.  

Now $\beta\circ\mathbf c$ is the map $S_1$ and $\beta\circ\mathbf d$
is the map $ S_2$, so to complete the proof of Proposition \ref{m56}
it suffices to show: 

\begin{lemma}
\label{m57}
$\mathbf c$ and $\mathbf d$ are homotopic in $\sss$.
\end{lemma}

\begin{proof}[Proof of Lemma \ref{m57}]
For each $k\geq 0$ and each $\bfn$ let $e_{k,\bfn}:(\P_k)_\bfn\to S^0$ be the
map which takes every simplex except the basepoint to the nontrivial element of
$S^0$, and let 
\[
e:\P\to\boS
\]
be the map given by the $e_{k,\bfn}$. Let
\[
{\mathbf e}: \P\owedge \bR_{\Theta}
\to \bR_{\Theta}
\]
be the composite
\[
\P\owedge \bR_{\Theta}
\xrightarrow{e\wedge 1}
\boS\owedge \bR_{\Theta}
\cong
\bR_{\Theta}.
\]
Then ${\mathbf e}\circ{\mathbf c}$ and ${\mathbf e}\circ{\mathbf d}$ are both
equal to the identity.
But $\mathbf e$ is a weak equivalence by  
Proposition \ref{Aug29.3} and Lemma \ref{m39}, and the result follows.
\end{proof}

\section{Background for the proof of Theorem \ref{m54}}
\label{colored}

\begin{notation}
In order to distinguish the product in $\Ag$ from the Cartesian product of
categories, we will denote the former by $\boxtimes$ from now on.
\end{notation}

We now turn to the proof of Theorem \ref{m54}, which will follow the
general outline of the proof of Theorem \ref{Nov12.3}.  The key ingredient in
that proof was the action of the monad $\bbP$ on $\bR$. That action was
constructed
from the family of operations given in Definition \ref{m22}(ii), and this
family in turn was constructed from the family of functors $\eta_\bigstar$ 
given in
Definition \ref{Nov13}(ii).  
For our present purpose we need 
the functors $\eta_\bigstar$
and also a 
family of functors 
\[
\bfd_\bs:\A_1\times\cdots\times\A_j
\to
\Ar,
\]
where each $\A_i$ is equal to $\Ag$ or $\Ar$; these will be
built from the symmetric monoidal structures of  $\Ag$ and $\Ar$ and 
the functor
\[
\Sigr: \A_{e,*,1}\to  \Ar.
\]
It is convenient to represent this situation by a function $r$ from
$\{1,\ldots,j\}$ to a two element set $\{u,v\}$, with $\A_i=\Ag$ if $r(i)=u$
and $\A_i=\Ar$ if $r(i)=v$.

%

\begin{example}
\label{m68}
A typical example is the functor
\[
\Ar\times(\Ag)^{\times 5}\times \Ar
\to
\Ar
\]
which takes $(x_1,\ldots,x_7)$ to
\[
i^\epsilon \Sigr(x_4\boxtimes x_3)\otimes x_7\otimes
\Sigr(x_6\boxtimes x_2\boxtimes x_5)\otimes x_1,
\]
where $i^\epsilon$ is the sign that arises from permuting $(x_1,\ldots,x_7)$
into the order $(x_4,x_3,x_7,x_6,x_2,x_5,x_1)$.
In Definition \ref{m67}(iv) we will represent such a functor by a surjection 
$h$ which
keeps track of which inputs go to which output factors and a permutation $\eta$
which keeps track of the order in which the inputs to each $\Sigr$ factor are 
multiplied. In the present example $h$ is the surjection $\{1,\ldots,7\}\to
\{1,2,3,4\}$ with 
\[
h^{-1}(1)=\{3,4\}, h^{-1}(2)=7, h^{-1}(3)=\{2,5,6\},
h^{-1}(4)=1
\]
and $\eta$ is the permutation $(256)(34)$.
\end{example}


In order to get the signs right we need a preliminary definition.

\begin{definition}
\label{m71}
(i)
For totally ordered sets $S_1,\ldots,S_m$, define
\[
\coprod_{i=1}^n S_i
\]
to be the disjoint union with the order relation given as follows:
$s<t$ if either $s\in S_i$ and $t\in S_j$ with $i<j$, or $s,t\in S_i$ with $s<t$
in the order of $S_i$.

(ii)
For a surjection 
\[
h:\{1,\ldots,j\}\to\{1,\ldots,m\}
\]
define $\theta(h)$ to be the 
permutation
\[
\{1,\ldots,j\}
\cong
h^{-1}(1)\coprod\cdots\coprod h^{-1}(m)
\cong
\{1,\ldots,j\};
\]
here the first map restricts to the identity on each $h^{-1}(i)$ and the second
is the unique ordered bijection.
\end{definition}

In Example \ref{m68}. $\theta(h)$ takes $1,\ldots,7$ respectively to
$7,4,1,2,5,6,3$.

\begin{definition}
\label{m67}
Let $j\geq 0$ and let $r:\{1,\ldots,j\}\to \{u,v\}$.  Let $\A_i$ denote $\Ag$
if $r(i)=u$ and $\Ar$ if $r(i)=v$.

(i) Let $1\leq m\leq j$. A surjection 
\[
h:\{1,\ldots,j\}\to\{1,\ldots,m\}
\]
is {\it adapted} to $r$ if 
$r$ is constant on 
each set $h^{-1}(i)$ and  
$h$ is  monic on $r^{-1}(v)$.

(ii) Given a surjection 
\[
h:\{1,\ldots,j\}\to\{1,\ldots,m\}
\]
which is adapted to $r$, define
\[
h_\blacklozenge: 
\A_1\times\cdots\times\A_j
\to
(\Ar)^{\times m}
\]
by
\[
h_\blacklozenge(x_1,\ldots,x_j)
=(i^{\epsilon}y_1,\ldots,y_m),
\]
where
$i^\epsilon$ is the sign that arises from putting the objects
$x_1,\ldots,x_j$ into the order $x_{\theta(h)^{-1}(1)},\ldots,
x_{\theta(h)^{-1}(j)}$
and 
\[
y_i=
\begin{cases}
\Sigr(\boxtimes_{l\in h^{-1}(i)} \,x_l) & \text{if $h^{-1}(i)\subset 
r^{-1}(u)$},
\\
x_{h^{-1}(i)} & \text{if $h^{-1}(i)\in r^{-1}(v)$}.
\end{cases}
\]

(iii)
A {\it datum} of type $r$ is a pair
\[
(h,\eta),
\]
where $h$ is a surjection which is 
adapted to $r$ and $\eta$ is an element of $\Sigma_j$ with the property
that $h\circ\eta=h$.

(iv) Given a datum 
\[
\bfd=(h,\eta),
\]
of type $r$,
define 
\[
\bfd_\bs :
\A_1\times\cdots\times\A_j
\to
\Ar
\]
to be the composite
\[
\A_1\times\cdots\times\A_j
\xrightarrow{\eta}
\A_{\eta^{-1}(1)}\times\cdots\times\A_{\eta^{-1}(j)}
=
\A_1\times\cdots\times\A_j
\xrightarrow{h_\blacklozenge}
(\Ar)^{\times m}
\xrightarrow{\otimes}
\Ar,
\]
where $\eta$ permutes the factors with the usual sign.
\end{definition}

We also need natural transformations between the functors $\bfd_\bs$.
First observe that $\mbox{sig}_{rel}$ is lax monoidal: there is a natural transformation
from the functor
\[
(\Ag)^{\times l}
\xrightarrow{\Sigr^{\times l}}
(\Ar)^{\times l}
\xrightarrow{\otimes}
\Ar
\]
to the functor
\[
(\Ag)^{\times l}
\xrightarrow{\boxtimes}
\Ag
\xrightarrow{\Sigr}
\Ar
\]
given by the maps
\[
S_*({X_1})
\otimes\cdots\otimes
S_*({X_l})
\xrightarrow{\times}
S_*({X_1}\times\cdots\times
{X_l})
\]
and 
\begin{multline*}
S_*(X_1\times X_1)
\otimes\cdots\otimes
S_*(X_l\times X_l)
\xrightarrow{\times}
S_*(X_1\times X_1\times\cdots\times X_l\times X_l)
\\
\cong
S_*((X_1\times\cdots\times X_l)\times (X_1\times\cdots\times X_l)).
\end{multline*}
Combining this with the symmetric monoidal structures of $\Ag$ and 
$\Ar$, we obtain a natural transformation $\bfd_\bs\to\bfd'_\bs$ whenever
$\bfd\leq\bfd'$, as defined in:

\begin{definition}
\label{m63}
For data of type $r$, define
\[
(h,\eta)
\leq
(h',\eta')
\]
if each set $h^{-1}(i)$ is contained in some set 
${h'}^{-1}(l)$.
\end{definition}

Our next definition is analogous to Definition \ref{m16} (the presence of the
letter $v$ in the symbols $P_{r;v}$ and $\O(r;v)$ will be explained in a 
moment).

\begin{definition}
\label{m60}
Let $r:\{1\ldots,j\}\to \{u,v\}$. 

(i) Let $P_{r;v}$ be the preorder whose elements are the data of type $r$, with
the order relation given by Definition \ref{m63}.

(ii) 
Define
an object $\O(r;v)_k$ of $\Sigma_k\ssk$ by
\[
(\O(r;v)_k)_\bfn=\Map_{\mathrm{preorder}}(U(\Delta^\bfn),P_{r;v})_+
\]
(where the $+$ denotes a disjoint basepoint);  the morphisms in
$(\Delta_\inj^{\text{op}})^{\times k}$ act in the evident way, and the
morphisms of the form $(\alpha,\mathrm{id})$ with $\alpha\in \Sigma_k$ act by
permuting the factors in $\Delta^\bfn$.

(iii) Define $\O(r;v)$ to be the object of $\ss$ with $k$-th term
$\O(r;v)_k$.
\end{definition}

\begin{remark}
\label{m59}
(i)
Given $r:\{1\ldots,j\}\to \{u,v\}$, let 
$m=|r^{-1}(v)|+1$, let
\[
h:\{1,\ldots,j\}\to \{1,\ldots,m\}
\]
be any surjection which is adapted to $r$, and let $e$ be the identity element
of $\Sigma_j$.
Then 
the datum $(h,e)$ 
is $\geq$ every element in $P_{r;v}$.

(ii)
Lemma \ref{m39} shows that each of the objects
$\O(r;v)_k$ has compatible degeneracies and is weakly equivalent to a point.
\end{remark}

We also need a preorder corresponding to the family of functors
\[
\eta_\bigstar:(\Ag)^{\times j}\to \Ag:
\]

\begin{notation}
\label{m69}
Let $r_u(j)$ (resp., $r_v(j)$) denote the constant function 
$\{1,\ldots,j\}\to\{u,v\}$ with value $u$ (resp., $v$).
\end{notation}

\begin{definition}
\label{m64}
Let $r:\{1\ldots,j\}\to \{u,v\}$.

(i) If $r=r_u(j)$, let 
$P_{r;u}$ be the set $\Sigma_j$ with the preorder in which every element is
$\leq$ every other, and let $\O(r;u)_k$ be the object $\O(j)_k$ 
of Definition \ref{m16}.

(ii)  Otherwise let $P_{r;u}$ be the empty set and let $\O(r;u)_k$ be the
multisemisimplicial set with a point in every multidegree. 

(iii) In either case, let 
$\O(r;u)$ be the object of $\ss$ with $k$-th term
$\O(r;u)_k$.
\end{definition}

In the next section we will use the objects $\O(r;v)$ and $\O(r;u)$ to 
construct a monad.
In preparation for that, we show that the collection of preorders $P_{r;v}$ 
and $P_{r;u}$ 
has suitable
composition maps.  Specifically, we show that it is a colored
operad (also called a multicategory) in the category of preorders. 

We refer the reader to 
\cite[Section2]{EM} for the definition of multicategory; we will mostly follow
the notation and terminology given there. In our case there are
two objects $u$ and $v$, and we think of a function $r:\{1,\ldots,j\}\to
\{u,v\}$ as a sequence of $u$'s and $v$'s.  

\begin{remark}
\label{m65}
Let $r:\{1,\ldots,j\}\to \{u,v\}$ and let 
$\A_i$ denote $\Ag$
if $r(i)=u$ and $\Ar$ if $r(i)=v$.
Let us write $\A_r$ for the category $\A_1\times\cdots\times\A_j$ and
$\A_{r;u}$ (resp., $\A_{r;v}$) for the category of functors $\A_r\to \Ag$
(resp., $\A_r\to \Ar$).  Define a multicategory with objects  $u$ and $v$ as follows. The multimorphisms  from $r$ to $u$ are given by    $\A_{r;u}$ and from $r$ to $v$ are given by $\A_{r;v}$. The actions of the symmetric groups are the obvious ones and the composition
is provided by the composition of functors.
Moreover, Definitions 
\ref{Nov13}(ii)
and \ref{m67}(iv) give inclusion functors
\[
\Phi_{r;u}:P_{r;u}\to \A_{r;u}
\]
and
\[
\Phi_{r;v}:P_{r;v}\to \A_{r;v}.
\]
We like to promote these to a functor between multicategories.  The source also has two objects, $u$, $v$ say and the functor is the identity on objects.  The set of  multimorphisms with target $u$ are empty unless the source only involves $u$'s in which case it is $P_{r;u}$ with $r=r_u(j)$. The set of multimorphisms with target $v$ and source $r$ is the  $P_{r;v}$.  The following definitions are chosen so that the inclusion functors
preserve the $\Sigma_j$ actions and the composition operations.
\end{remark}

We define the right $\Sigma_j$ action on a
multimorphism with $j$ sources as follows.
Let $\alpha\in\Sigma_j$ and $r:\{1,\ldots,j\}\to
\{u,v\}$.  
Define $r^\alpha$ to be the composite
\[
\{1,\ldots,j\}
\xrightarrow{\alpha}
\{1,\ldots,j\}
\xrightarrow{r}
\{u,v\}.
\]
If $r=r_u(j)$ then $r^\alpha=r$ 
and the map
\[
\alpha: P_{r;u} \to P_{r^\alpha;u}
\]
is the right action of $\Sigma_j$ on itself.
If $(h,\eta)\in P_{r,v}$ define $(h,\eta)\alpha\in P_{r^\alpha;v}$ to be 
$(h\circ\bar{\alpha},\bar{\alpha}^{-1}\circ\eta\circ\alpha)$, where 
$\bar{\alpha}\in \Sigma_j$ is the permutation whose restriction to each 
$\alpha^{-1}h^{-1}(i)$ is the order-preserving bijection to $h^{-1}(i)$.

We define the composition operation as follows.  
If the composition involves only $u$'s then it is the
composition in the operad ${\mathcal M}$ of \cite[Definition
3.1(i)]{MR0420610}.  Otherwise
let $i,j_1,\ldots,j_i\geq 0$,
let $r:\{1,\ldots,i\}\to\{u,v\}$, and for $1\leq l\leq i$ let
$r_l:\{1,\ldots,j_l\}\to\{u,v\}$; assume that if $r(l)=u$ then $r_l$ is
$r_u(j_l)$.
Let $(h,\eta)\in P_{r;v}$ and for $1\leq l\leq i$ let $x_l\in P_{r_l;r(l)}$.
If $r(l)=v$ then $x_l$ has the form $(h_l,\eta_l)$,
otherwise $x_l$ is an element $\eta_l\in\Sigma_{j_l}$ and we write $h_l$ for 
the map 
$\{1,\ldots, j_l\} \to \{1\}$.
Define the composition operation $\Gamma$ by
\begin{equation}
\label{m70}
\Gamma((h,\eta),x_1,\ldots,x_l)=(H,\theta),
\end{equation}
where $\theta$
is the composite $\gamma_{\mathcal M}(\eta,\eta_1,\ldots,\eta_i)$ in the operad
${\mathcal M}$ of \cite[Definition 3.1(i)]{MR0420610} and $H$ is the following multivariable composite $h\circ(h_1,\ldots,h_i)$: 
we are going to make the composite $H$ explicit with $h$ as in Example \ref{m68}, the general formula should be clear then. The source of the morphism $(h,\eta)$ is $r=(v,u,u,u,u,u,v)$ and the target is $v$. It is convenient to write $h$ in the form
$$ h=(h^{-1}\{ 1\} , h^{-1}\{ 2 \} ,  h^{-1}\{ 3\}, h^{-1}\{ 4\} )= (\{ 3,4\} ,\{ 7\} , \{ 2,5,6\} , \{1\} ).$$
Then the  multivariable composite $H$ for arbitrary $h_1, h_2, \ldots, h_7$ takes the form
$$ ( h^{-1}_3\{ 1\} \cup h^{-1}_4\{ 1\} , h_7, h^{-1}_2\{ 1\} \cup  h^{-1}_5\{ 1\} \cup  h^{-1}_6\{ 1\} ,h_1).$$
In this notation, the $j_4$ elements in the source of $h_4$ have to be shifted by $j_3$, the ones of $h_7$ should be shifted by $j_3+j_4$ and so on. Moreover,  the round parentheses of $h_7$ and $h_1$  should be ignored. The formula may look strange but notice that $h_k$ is the constant map to $\{ 1\}$ for all $2\leq k \leq 6$ because for these we have $r(k)=u$.

\begin{prop}
With these definitions, the collection of preorders $P_{r;u}$ and $P_{r;v}$ are the multimorphisms in 
a multicategory from $r$ to $u$ and from $r$ to $v$ respectively. The functors $\Phi_{r;u}$ and $\Phi_{r;v}$ define a multifunctor between multicategories.
\end{prop}

\begin{proof}
This is immediate from Remark \ref{m65}.
\end{proof}

\section{A monad in $\ss\times\ss$}
\label{proof}

In this section we construct a monad in $\ss\times \ss$ which acts on the
pair $(\bR_{e,*,1},\bR_{\text{rel}}^\Z)$. The arguments in this and in the next sections are along the same lines with $\A ssoc$ and $\C omm$ replaced by monads encoding (commutative) monoid maps or (symmetric) monoidal transformations.

\begin{definition}
\label{m82}
Let $j\geq 0$
and let $\bX,\bY\in\ss$.

(i)
For $\alpha\in\Sigma_j$ and $r:\{1,\ldots,j\}\to \{u,v\}$, 
define 
\[
\tilde{\alpha}:
\O(r;v)\to
\O(r^\alpha;v)
\]
by
\[
(\tilde{\alpha}(a))(\sigma)=(a(\sigma))\alpha
\]
where
$a\in\Map_{\mathrm{preorder}}(U(\Delta^\bfn),P_{r;v})_+$ and 
$\sigma\in U(\Delta^\bfn)$.

(ii)  Define 
\[
(\bX,\bY)^{\otimes r}=\bZ_1\otimes\cdots\otimes \bZ_j,
\]
where $\bZ_i$ denotes $\bX$ if $r(i)=u$ and $\bY$ if $r(i)=v$.

(iii) 
For $\alpha\in\Sigma_j$ define
\[
\bar{\alpha}:
\bigvee_r
\O(r;v)
\owedge
(\bX,\bY)^{\otimes r}
\to
\bigvee_r
\O(r;v)
\owedge
(\bX,\bY)^{\otimes r}
\]
to be the map which takes the $r$-summand to the $r^\alpha$-summand by means of
the map
\[
\O(r;v)
\owedge
(\bX,\bY)^{\otimes r}
\xrightarrow{\tilde{\alpha}\owedge \alpha}
\O(r^\alpha;v)
\owedge
(\bX,\bY)^{\otimes r^\alpha}.
\]
\end{definition}

Note that the maps $\bar{\alpha}$ give $\bigvee_r
\O(r;v)
\owedge
(\bX,\bY)^{\otimes r}$
a right $\Sigma_j$ action.

Recall Notation \ref{m69}.

\begin{definition}
(i)
Define a functor $\bbO:\ss\times\ss\to\ss\times \ss$ by
\[
\bbO(\bX,\bY)
=
(\bbO_1(\bX),\bbO_2(\bX,\bY)),
\]
where
\[
\bbO_1(\bX)=
\bigvee_{j\geq 0}
\bigl(
\O(r_u(j);u)
\owedge
\bX^{\otimes j}
\bigr)/\Sigma_j
\]
and
\[
\bbO_2(\bX,\bY)=
\bigvee_{j\geq 0}
\bigvee_{r :\{1,\ldots ,j\} \ra \{u,v\}}
\bigl(\O(r;v)
\owedge
(\bX,\bY)^{\otimes r}
\Bigr)/\Sigma_j.
\]

(ii) Define a natural transformation
\[
\iota:(\bX,\bY)\to \bbO(\bX,\bY)
\]
to be $(\iota_1,\iota_2)$,
where $\iota_1$ is the composite
\[
\bX\xrightarrow{\cong}
\boS\owedge \bX
=
\O(r_u(1);u)\owedge \bX
\hookrightarrow
\bbO(\bX)
\]
and $\iota_2$ is the composite
\[
\bY\xrightarrow{\cong}
\boS\owedge \bY
=
\O(r_v(1);v)\owedge \bY
\hookrightarrow
\bbO(\bY).
\]
\end{definition}

For the structure map $\mu:\bbO\bbO\to\bbO$ we need a composition operation for
the collection of objects $\O(r;u)$ and $\O(r;v)$.
Recall Definition \ref{m71}(i) and the map $\Gamma$ defined in Equation
\eqref{m70}.

\begin{definition}
\label{m73}
Let $i,j_1,\ldots,j_i\geq 0$,
let $r:\{1,\ldots,i\}\to\{u,v\}$, and for $1\leq l\leq i$ let
$r_l:\{1,\ldots,j_l\}\to\{u,v\}$; assume that if $r(l)=u$ then $r_l$ is
$r_u(j_l)$.
Let 
\[
R:\{1,\ldots,\sum j_l\}\to\{u,v\}
\]
be the composite
\[
\{1,\ldots,\sum j_l\}
\cong
\coprod_{l=1}^i  \{1,\ldots,j_l\}
\to
\{u,v\},
\]
where the first map is the unique order-preserving bijection and the second
restricts on each $\{1,\ldots,j_l\}$ to $r_l$.  
Define a map
\[
\gamma:
\O(r;v)
\owedge
\bigl(\O(r_1;r(1))\otimes\cdots\otimes O(r_i;r(i))\bigr)
\to
\O(R;v)
\]
in $\ss$ by the formulas
\[
\gamma(a\wedge [e, b_1\wedge \ldots\wedge b_i])
(\sigma_1\times\cdots\times \sigma_i)
=
\Gamma(a(\sigma_1\times\cdots\times
\sigma_i),b_1(\sigma_1),\ldots, b_i(\sigma_i))
\]
(where $e$ is the identity element of the relevant symmetric group)
and
\[
\gamma(a\wedge [\alpha, b_1\wedge \ldots\wedge b_i])
=
(\alpha,\mathrm{id})_*\gamma((\alpha^{-1},\mathrm{id})_*a\wedge [e, b_1\wedge
\ldots\wedge b_i]).
\]
\end{definition}

This operation satisfies the analogues of Lemmas \ref{m14}, \ref{m20}, and
\ref{m72}.

Now we can define 
\[
\mu:\bbO\bbO\to\bbO
\]
to be $(\mu_1,\mu_2)$, where $\mu_1$ is given by Definition \ref{m25}(iv) and
$\mu_2$ is defined in a similar way using Definition \ref{m73}.

\begin{prop}
The transformations $\mu$ and $\iota$ define a monad structure on $\bbO$.
\qed
\end{prop}

We conclude this section by giving the action of $\bbO$ on the pair 
$(\bR_{e,*,1},\bR_{\text{rel}}^\Z)$.  Recall Remark \ref{m65}.

\begin{definition}
\label{m101}
Let $k_1,\ldots,k_j$ be non-negative integers and let $\bfn_i$ be a
$k_i$-fold multi-index for $1\leq i\leq j$.  Let $r:\{1,\ldots,j\}\to \{u,v\}$,
and for $1\leq i\leq j$ let $\pre_i$ denote $\pre_{e,*,1}$ if $r(i)=u$ and
$\pre_{\text{rel}}^\Z$ if $r(i)=v$.
For any map of preorders
\[
a: U(\Delta^{\bfn_1}\times\cdots\times \Delta^{\bfn_j})\to P_{r;v}
\]
define
\[
a_*:
\pre_1^{k_1}(\Delta^{\bfn_1})
\times
\cdots
\times
\pre_j^{k_j}(\Delta^{\bfn_j})
\to
(\pre_{\text{rel}}^\Z)^{k_1+\cdots+k_j}(\Delta^{(\bfn_1,\ldots,\bfn_j)})
\]
by 
\begin{multline*}
a_*(F_1,\ldots, F_j)(\sigma_1\times\cdots\times\sigma_j,o_1\times\cdots\times 
o_j)
\\
=
i^{\epsilon(\zeta)}
\Phi_{r;v}(a(\sigma_1\times\cdots\times\sigma_j))
(F_1(\sigma_1,o_1), \ldots, F_j(\sigma_j,o_j)),
\end{multline*}
where $\Phi_{r;v}$ was defined in Remark \ref{m65} and $\zeta$ is the block permutation that takes blocks 
${\mathbf b}_1$, \dots, ${\mathbf b}_j$, ${\mathbf c}_1$, \dots, ${\mathbf c}_j$
of size
$k_1$, \dots, $k_j$, $\dim \sigma_1$, \dots, $\dim \sigma_j$ into the order 
${\mathbf b}_1$, ${\mathbf c}_1$, \dots, ${\mathbf b}_j$, ${\mathbf c}_j$.
\end{definition}

\begin{lemma}
If $F_i\in \ad_i^{k_i}(\Delta^{\bfn_i})$ for $1\leq i\leq j$ then
$$a_*(F_1,\ldots, F_j)
\in (\ad_{\mathrm{rel}}^\Z)^{k_1+\cdots+k_j}(\Delta^{(\bfn_1,\ldots,\bfn_j)}).$$
\end{lemma}

\begin{proof}
This is a straightforward consequence of the fact that 
the natural transformation
from the functor
\[
(\Ag)^{\times l}
\xrightarrow{\Sigr^{\times l}}
(\Ar)^{\times l}
\xrightarrow{\otimes}
\Ar
\]
to the functor
\[
(\Ag)^{\times l}
\xrightarrow{\boxtimes}
\Ag
\xrightarrow{\Sigr}
\Ar
\]
given by the cross product is a quasi-isomorphism.
\end{proof}

\begin{definition}
\label{m85}
Let $j\geq 0$ and let $r:\{1,\ldots,j\}\to\{u,v\}$.
Define a map
\[
\phi_r: \O(r;v)
\owedge 
(\bR_{e,*,1},\bR_{\text{rel}}^\Z)^{\otimes r}
\to
\bR_{\text{rel}}^\Z
\]
in $\ss$ by the formulas
\[
\phi_r(a\wedge [e, F_1\wedge\cdots\wedge F_j])
=
a_*(F_1\wedge\cdots\wedge F_j)
\]
(where $e$ denotes the identity element of the relevant symmetric group)
and
\[
\phi_r(a\wedge [\alpha, F_1\wedge\cdots\wedge F_j])
=
(\alpha,\mathrm{id})_*\phi_r((\alpha^{-1},\mathrm{id})_*a\wedge [e,
F_1\wedge\cdots\wedge F_j]).
\]
\end{definition}

Next observe that the maps $\phi_r$ induce a map

\begin{equation}
\label{m74}
\Bigl(
\bigvee_r
\bigl(\O(r;v)
\owedge
(\bR_{e,*,1},\bR_{\mathrm{rel}}^\Z)^{\otimes r}
\Bigr)/\Sigma_j
\to
\bR_{\mathrm{rel}}^\Z
\end{equation}
for each $j\geq 0$.
We define 
\[
\nu:\bbO(\bR_{e,*,1},\bR_{\text{rel}}^\Z)
\to
(\bR_{e,*,1},\bR_{\text{rel}}^\Z)
\]
to be the pair $(\nu_1,\nu_2)$, where
$\nu_1$ is given by Definition \ref{m75} and $\nu_2$ is given by
the maps \eqref{m74}.

\begin{prop}
$\nu$ is an action of $\bbO$ on $(\bR_{e,*,1},\bR_{\mathrm{rel}}^\Z)$.
\end{prop}

\begin{proof}
This is a straightforward consequence of Remark \ref{m65}.
\end{proof}

\section{A monad in $\sss\times\sss$}
\label{msss}

First we give $ \O(r;v) \owedge (\bX,\bY)^{\wedge r} $ the structure of a
multisemisimplicial symmetric spectrum when $\bX,\bY\in \sss$.  The definition
is analogous to Definition \ref{m29}.  Recall Definition \ref{m24}.

\begin{definition}
Let $j,k\geq 0$. 
Let $s$ be the 1-simplex of $S^1$.   
Define
\[
\omega:S^1\wedge (\O(r;v)\owedge (\bX,\bY)^{\wedge r})_k
\to
(\O(r;v)\owedge (\bX,\bY)^{\wedge r})_{k+1}
\]
as follows: for
$a\in (\O(r;v)_k)_\bfn$ and $x\in
(((\bX,\bY)^{\wedge r})_k)_\bfn$, let
\[
\omega(s\wedge (a\wedge x))
=
(a\circ \Pi)\wedge \omega(s\wedge x).
\]
\end{definition}

\begin{definition}
\label{m100}
Define a functor $\bbP:\ss\times\ss\to\ss\times \ss$ by
\[
\bbP(\bX,\bY)
=
(\bbP_1(\bX),\bbP_2(\bX,\bY)),
\]
where
\[
\bbP_1(\bX)=
\bigvee_{j\geq 0}
\bigl(
\O(r_u(j);u)
\owedge
\bX^{\wedge j}
\bigr)/\Sigma_j
\]
and
\[
\bbP_2(\bX,\bY)=
\bigvee_{j\geq 0}
\Bigl(
\bigvee_r 
\bigl(\O(r;v)
\owedge
(\bX,\bY)^{\wedge r}
\Bigr)/\Sigma_j.
\]
\end{definition}

The proof that $\bbP$ inherits a monad structure and an action on
$(\bR_{e,*,1},\bR_{\text{rel}}^\Z)$ is the same as the corresponding proof in Section
\ref{monadsss}.

For use in the next section we record a lemma.  Let $\C$ be the category whose
objects are triples $(\bX,\bY,f)$, where $\bX$ and $\bY$ are monoids in $\sss$
and $f$ is a map $\bX\to\bY$ in $\sss$ which is not required to be a monoid 
map; the morphisms are commutative diagrams
\[
\xymatrix{
\bX
\ar[d]_f
\ar[r]
&
\bX'
\ar[d]^{f'}
\\
\bY
\ar[r]
&
\bY',
}
\]
where the horizontal arrows are monoid maps.

\begin{lemma}
\label{m78}
{\rm (i)} There is a functor $\Upsilon$ from $\bbP$ algebras to $\C$ which takes
$(\bX,\bY)$ to a map $\bX\to\bY$; in particular, $\bX$ and $\bY$ have natural
monoid structures.

{\rm (ii)} $\Upsilon(\bR_{e,*,1},\bR_{\mathrm{rel}}^\Z)$ is the map
\[
\Sigr:\bR_{e,*,1}\to \bR_{\mathrm{rel}}^\Z,
\]
where $\bR_{e,*,1}$ and $\bR_{\mathrm{rel}}^\Z$ have the monoid structures given by Lemma
\ref{m35}(i).
\end{lemma}

\begin{proof}  
Recall Notation \ref{m69}, and 
let $e$ denote the identity element of $\Sigma_j$. 

Part (i).
The map $f:\bX\to\bY$ is the composite
\[
\bX
\cong
\boS\owedge \bX
=
\O(r_u(1);v) \owedge \bX
\hookrightarrow
\bbP_2(\bX,\bY)
\to \bY.
\]

The monoid structure on $\bX$ is given by Lemma \ref{m35}; it remains to give
the monoid structure on $\bY$.   It suffices to give an action on 
$\bY$ of the monad $\mathbb A$ defined in the proof of Lemma \ref{m35}, and for
this in turn it suffices to give a suitable natural transformation $\mathbb
A\to\bbP_2$.

For each $j\geq 0$ let $h_0$ be the identity map of $\{1,\ldots,j\}$.  Then 
$h_0$
is adapted to $r_v(j)$, so we obtain an element $(h_0,e)\in P_{r_v(j);v}$.

For each $j,k\geq 0$ and each $k$-fold multi-index $\bfn$, define an
element
\[
b_{j,k,\bfn}\in (\O(r_v(j);v)_k)_\bfn
\]
to be the constant function $U(\Delta^\bfn)\to P_{r_v(j);v}$ whose value is 
$(h_0,e)$.  Next define a map
\[
\boS\to \O(r_v(j);v)
\]
by taking the nontrivial simplex of $(\boS_k)_\bfn$ to $b_{j,k,\bfn}$.

Now the composite
\begin{multline*}
{\mathbb A}(\bY)=\bigvee_{j\geq 0} \bY^{\wedge j}
\cong
\bigvee_{j\geq 0} \boS\owedge \bY^{\wedge j}
\\
\to
\bigvee_{j\geq 0}
\Bigl(
\bigvee_r 
\bigl(\O(r;v)
\owedge
(\bX,\bY)^{\wedge r}
\Bigr)/\Sigma_j
=\bbP_2(\bX,\bY)
\end{multline*}
is the desired map.

Part (ii) is an easy consequence of the definitions.
\end{proof}

\section{Rectification}
\label{rect2}

In this section we prove Theorem \ref{m54}.  The argument is
analogous to that in Section \ref{rect}.

First we consider a monad in $\sss\times\sss$ which is simpler than $\bbP$.

\begin{definition}
\label{m86}
(i) Define $\bbP'(\bX,\bY)$ to be 
\[
(
\bigvee_{j\geq 0} \,  \bX^{\wedge j}/\Sigma_j
,
\bigvee_{j\geq 0} \,
\bigl(\bigvee_r \,
(\bX,\bY)^{\wedge r}
\bigr)
/\Sigma_j
).
\]

(ii)  For each $j\geq 0$ and each $r:\{1,\ldots,j\}\to\{u,v\}$, let
\[
\xi_j:\O(r_u(j);u)\to \boS
\]
and 
\[
\zeta_r:\O(r;v)\to \boS
\]
be the maps which take each nontrivial simplex of the $k$-th object to the 
nontrivial simplex of $\oS_k$ in the same multidegree.
Define a natural transformation 
\[
\Xi: \bbP\to\bbP'
\]
to be the pair $(\Xi_1,\Xi_2)$, where $\Xi_1$ is
the wedge of the composites
\[
\bigl(\O(r_u(j);u)\owedge \bX^{\wedge j}\bigr)/\Sigma_j
\xrightarrow{\xi_j\owedge 1}
\bigl(\boS\owedge \bX^{\wedge j}\bigr)/\Sigma_j
\xrightarrow{\cong}
\bX^{\wedge j}/\Sigma_j.
\]
and $\Xi_2$ is the wedge of the composites 
\begin{multline*}
\Bigl(
\bigvee_r 
\O(r;v)
\owedge
(\bX,\bY)^{\wedge r}
\Bigr)/\Sigma_j
\xrightarrow{ \bigvee \zeta_r \owedge 1}
\Bigl(
\bigvee_r 
\boS
\owedge
(\bX,\bY)^{\wedge r}
\Bigr)/\Sigma_j
\\
\xrightarrow{\cong}
\Bigl(
\bigvee_r 
(\bX,\bY)^{\wedge r}
\Bigr)/\Sigma_j.
\end{multline*}
\end{definition}

\begin{prop}
\label{m79}
{\rm (i)} An algebra over $\bbP'$ is the same thing as a pair of 
commutative monoids $(\bX,\bY)$ in $\sss$ together with a monoid map
$\bX\to\bY$.

{\rm (ii)} $\Xi$ is a map of monads.

{\rm (iii)} Suppose that each $X_k$ and each $Y_k$ has compatible 
degeneracies (see Definition \ref{m31}).  Let $\bbP^q$ denote the $q$-th 
iterate of $\bbP$.  Then each map
\[
\Xi:\bbP^q(\bX,\bY)\to \bbP'\bbP^{q-1}(\bX,\bY)
\]
is a weak equivalence.
\end{prop}

\begin{proof}
Part (i).  Let $(\bX,\bY)$ be an algebra over $\bbP'$.  The fact that $\bX$
and $\bY$ are commutative monoids is immediate from the definitions.  The map
$f:\bX\to\bY$ is constructed as in the proof of Lemma \ref{m78}(i).  

To show
that $f$ is a monoid map, we first observe that there are two inclusions of
$\bX^{\wedge j}$ into $\bbP'_2\bbP'(\bX,\bY)$.  Let $i_1$ be the composite
\[
\bX^{\wedge j}
\hookrightarrow
\bbP'_1(\bX)
\hookrightarrow
\bbP'_2\bbP'(\bX,\bY),
\]
where the second arrow is the inclusion of the summand indexed by
$j=1,r=r_u(1)$.
Let $i_2$ be the composite
\[
\bX^{\wedge j}
\hookrightarrow
\bbP'_2(\bX,\bY)^{\wedge j}
\hookrightarrow
\bbP'_2\bbP'(\bX,\bY),
\]
where the first arrow is the $j$-fold smash of the inclusion of the 
$r_u(1)$ summand, and the second arrow is the inclusion of the $r_v(j)$
summand.

Consider the commutative diagram
\[
\xymatrix{
\bX^{\wedge j}
\ar[rr]^-=
\ar[rd]^{i_1}
\ar[dd]
&&
\bX^{\wedge j}
\ar[d]
\\
& 
\bbP'_2\bbP'(\bX,\bY)
\ar[d]_{\bbP'_2\nu}
\ar[r]^-\mu
&
\bbP'_2(\bX,\bY)
\ar[d]^\nu
\\
\bX
\ar[r]
&
\bbP'_2(\bX,\bY)
\ar[r]^-\nu
&
\bY.
}
\]
Let $H$ denote the composite of the right-hand vertical arrows.  Then the
diagram shows that the composite
\begin{equation}
\label{m89}
\bX^{\wedge j}
\to\bX
\xrightarrow{f}
\bY
\end{equation}
is $H$.

Next consider the commutative diagram
\[
\xymatrix{
\bX^{\wedge j}
\ar[rr]^-=
\ar[rd]^{i_2}
\ar[dd]_{f^{\wedge j}}
&&
\bX^{\wedge j}
\ar[d]
\\
& 
\bbP'_2\bbP'(\bX,\bY)
\ar[d]_{\bbP'_2\nu}
\ar[r]^-\mu
&
\bbP'_2(\bX,\bY)
\ar[d]^\nu
\\
\bY^{\wedge j}
\ar[r]
&
\bbP'_2(\bX,\bY)
\ar[r]^-\nu
&
\bY.
}
\]
This diagram shows that the composite
\begin{equation}
\label{m90}
\bX^{\wedge j}
\xrightarrow{f^{\wedge j}}
\bY^{\wedge j}
\to
\bY
\end{equation}
is also $H$. 
Therefore the composites \eqref{m89} and \eqref{m90} are equal as reqired.

Part (ii) is immediate from the definitions, and the proof of part
(iii) is the same as for Proposition \ref{m34}(iii) (but using Remark 
\ref{m59}(ii)).
\end{proof}

\begin{proof}[Proof of Theorem \ref{m54}]
The proof follows the outline of 
the proof of Theorem \ref{Nov12.3} (given in 
Section \ref{rect}); we refer the reader to that proof for some of the 
details.  We have a diagram of  simplicial $\bbP$-algebras
\[
(\bR_{e,*,1},\bR_{\text{rel}}^\Z)_\bu \xleftarrow{\varepsilon}
B_\bu(\bbP,\bbP,(\bR_{e,*,1},\bR_{\text{rel}}^\Z))
\xrightarrow{\Xi_\bu}
B_\bu(\bbP',\bbP,(\bR_{e,*,1},\bR_{\text{rel}}^\Z)).
\]
By Lemma \ref{m78}, this gives a diagram
\begin{equation}
\label{m76}
\xymatrix{
(\bR_{e,*,1})_\bu
\ar[d]_\Sigr
&
B_\bu(\bbP_1,\bbP,(\bR_{e,*,1},\bR_{\text{rel}}^\Z))
\ar[l]_-\varepsilon
\ar[r]^-{(\Xi_1)_\bu}
\ar[d]
&
B_\bu(\bbP'_1,\bbP,(\bR_{e,*,1},\bR_{\text{rel}}^\Z))
\ar[d]
\\
(\bR_{\text{rel}}^\Z)_\bu
&
B_\bu(\bbP_2,\bbP,(\bR_{e,*,1},\bR_{\text{rel}}^\Z))
\ar[l]_-\varepsilon
\ar[r]^-{(\Xi_2)_\bu}
&
B_\bu(\bbP'_2,\bbP,(\bR_{e,*,1},\bR_{\text{rel}}^\Z)),
}
\end{equation}
in which all objects are simplicial monoids and all horizontal arrows are
monoid maps.  By Proposition \ref{m79}(i),
the right column is a simplicial monoid map between simplicial commutative
monoids.  Moreover, each map $\varepsilon$ is a homotopy 
equivalence of simplicial objects, and (using Proposition 
\ref{m79}(iii)) $(\Xi_1)_\bu$ and $(\Xi_2)_\bu$ are weak
equivalences in each simplicial degree.

The objects of the diagram \eqref{m76} are simplicial objects in $\sss$.  We
obtain a diagram 
\begin{equation}
\label{m77}
\xymatrix{
|(\bR_{e,*,1})_\bu|
\ar[d]_{|\Sigr|}
&
|B_\bu(\bbP_1,\bbP,(\bR_{e,*,1},\bR_{\text{rel}}^\Z))|
\ar[l]_-{|\varepsilon|}
\ar[r]^-{|(\Xi_1)_\bu|}
\ar[d]
&
|B_\bu(\bbP'_1,\bbP,(\bR_{e,*,1},\bR_{\text{rel}}^\Z))|
\ar[d]
\\
|(\bR_{\text{rel}}^\Z)_\bu|
&
|B_\bu(\bbP_2,\bbP,(\bR_{e,*,1},\bR_{\text{rel}}^\Z))|
\ar[l]_-{|\varepsilon|}
\ar[r]^-{|(\Xi_2)_\bu|}
&
|B_\bu(\bbP'_2,\bbP,(\bR_{e,*,1},\bR_{\text{rel}}^\Z))|
}
\end{equation}
of simplicial objects in $\S p$ (the category of symmetric spectra)
by applying the geometric realization functor $\sss\to \S p$
to the diagram \eqref{m76} in each simplicial degree.
All objects are simplicial monoids and all horizontal arrows are
monoid maps, and the right column is a simplicial monoid map between simplicial
commutative monoids.  The maps $|\varepsilon|$ are homotopy equivalences of 
simplicial objects and the maps $|(\Xi_1)\bu|$ and $|(\Xi_1)\bu|$ are weak
equivalences in each simplicial degree. 

Finally, we apply geometric realization to the diagram \eqref{m77}.
We define $\mathbf A$ to be
$||B_\bu(\bbP_1,\bbP,(\bR_{e,*,1},\bR_{\text{rel}}^\Z))||$,
$\mathbf B$ to be
$||B_\bu(\bbP_2,\bbP,(\bR_{e,*,1},\bR_{\text{rel}}^\Z))||$,
$\mathbf C$ to be
$||B_\bu(\bbP'_1,\bbP,(\bR_{e,*,1},\bR_{\text{rel}}^\Z))||$,
and
$\mathbf D$ to be
$||B_\bu(\bbP'_2,\bbP,(\bR_{e,*,1},\bR_{\text{rel}}^\Z))||$.
This gives the diagram of Theorem \ref{m54}.
\end{proof}

\begin{remark}
\label{m80}
The symmetric ring spectrum $\mathbf C$ is the same as the symmetric ring
spectrum $\bMc_{e,*,1}$ given by Theorem \ref{Nov12.3}.  There is a ring map 
\[
(\bM_{\text{rel}}^\Z)^{\mathrm{comm}}
\to
\mathbf D
\]
which is a weak equivalence (because there is a commutative diagram whose first
row is Diagram \eqref{m81} and whose second row is the second row of the
diagram in Theorem \ref{m54}).
\end{remark}

\section{Improved versions of geometric and symmetric Poincar\'e bordism.}
\label{Nov12.1}

In order to state our next theorem we need some background. 

Let $\S p$ denote the category of symmetric spectra. 

Recall (from \cite[Definition 13.2(i) and the second paragraph of Section 
19]{LM12}) the strict
monoidal category $\T$ whose objects are the
triples $(\pi,Z,w)$, where $\pi$ is a group, $Z$ is a simply-connected free
$\pi$-space, and $w$ is a homomorphism $\pi\to \{\pm 1\}$.
There is a monoidal 
functor 
\[
\Mgeom: \T\to \S p
\]
which takes $(\pi,Z,w)$ to $\bM_{\pi, Z,w}$ 
(\cite[Definition 19.1 and Theorem 19.2]{LM12}).

Let $\R$ be the category of rings with involution. We like to say that there is 
 a functor 
\[
\Msym: \R\to\S p
\]
which takes $R$ to $\bM_{\text{rel}}^R$. As explained in Appendix C, the ad theory $\ad_{\text{rel}}^R$ does not depend on $R$ in a functorial way. However, there is functorial  refinement which is constructed in Appendix C and which we may use instead: there is an ad theory $\ad^R_{\mathrm{Rel,sch}}$ which depends on $R$ in a natural way. 
The same proof as in 
\cite[Theorem 
19.2, Theorem 18.5]{LM12}) then shows that its geometric realization $\Msym$ is monoidal.

There is a functor $\rho:\T\to\R$ which takes $(\pi,Z,w)$ to $\Z[\pi]$ 
with the $w$-twisted involution (\cite[Definition 13.2(ii)]{LM12}).
In Section \ref{revis} we constructed a natural transformation
\[
\Sigr: \Mgeom\to \Msym\circ \rho.
\]
More precisely, in the notation of appendix $C$ for objects in $\ad^R_{\mathrm{Rel,sch}}$, we have
$$\Sigr (X,f,\xi, \Phi )=(S_*(\tilde{X}),(S_*(\tilde{X})^t\otimes S_*(\tilde{X}))^W, \gamma, \phi)$$ where $\gamma$ is the obvious map and $\phi\in (\Z^\omega \otimes_R(S_*(\tilde{X})^t\otimes S_*(\tilde{X}))^W)^{\Z/2}$ is induced by the Alexander-Whitney map. The lifting function $\Phi$ gives an isomorphism between $S_*(\tilde{X})$ and the free $R$-module on the set of all singular simplexes in $X$ (see \cite[Section 10]{LM12}). This means that $\Sigr (X,f,\xi, \Phi )$ refines to a schematic Relaxed quasi-symmetric complex in a functorial way (for the action of morphisms in $\T$ on the lifting structure see \cite[Section 13]{LM12}).  The new map $\Sigr$ coincides with the map $\Sigr$ of Section \ref{revis} under the natural transformation from $\A^R_{\text{Rel,sch}}$ to $\A^R_{\text{rel}}$.\par 
However, $\Sigr$ is not a monoidal transformation, 
and $\Mgeom$ and $\Msym$ are not symmetric monoidal functors
(we recall the definitions of monoidal transformation and symmetric monoidal
functor below).  Our next result shows that 
there is a monoidal transformation between symmetric monoidal functors which is
weakly equivalent to $\Sigr$.

\begin{thm}
\label{Nov13.3}
There are symmetric monoidal functors 
$\bP:\T\to\S p$, $\bL:\R\to 
\S p$, and a monoidal natural transformation $\Sig:\bP\to \bL\circ \rho$
such that

{\rm (i)} $\bP$ is weakly equivalent as a monoidal functor to $\Mgeom$; specifically,
there is a monoidal functor $\bA:\T\to\S p$ and monoidal weak equivalences
\[
\Mgeom
\leftarrow \bA\to 
\bP.
\]

{\rm (ii)} $\bL$ is weakly equivalent as a monoidal functor to $\Msym$; specifically,
there is a monoidal functor $\bB:\R\to\S p$ and monoidal weak equivalences
\[
\Msym
\leftarrow \bB\to 
\bL.
\]

{\rm (iii)} The natural transformations $\Sig:\bP\to\bL\circ \rho$ and
$\Sigr:\Mgeom\to\Msym\circ \rho$ are weakly equivalent in $\S p$; specifically, there is a
natural transformation $\bA\to\bB\circ\rho$ which makes the following
diagram strictly commute
\[
\xymatrix{
\Mgeom
\ar[d]_{\Sigr}
&
\bA
\ar[l]
\ar[r]
\ar[d]
&
\bP
\ar[d]^\Sig
\\
\Msym\circ \rho
&
\bB\circ \rho
\ar[l]
\ar[r]
&
\bL\circ \rho.
}
\]
\end{thm}

\begin{remark}
\label{Aug27}
(i) Theorem \ref{Nov13.3} implies that $\bL(R)$ is a strictly commutative 
symmetric ring spectrum when $R$ is commutative. Also, $\bP(e,*,1)$ is a 
strictly commutative symmetric ring spectrum and $\Sig:\bP(e,*,1)\to \bL(\Z)$ 
is a map of symmetric ring spectra.  This is compatible with Theorem \ref{m54}:
there is a commutative diagram
\[
\xymatrix{
\mathbf{C}
\ar[d]
\ar[r]
&
\bP(e,*,1)
\ar[d]_\Sig
\\
\mathbf{D}
\ar[r]
&
\bL(\Z)
}
\]
in which the horizontal arrows are ring maps, and they are weak equivalences 
by the argument given in Remark \ref{m80}.

(ii) The fact that $\Sig$ is a monoidal functor is a spectrum-level version of 
Ranicki's multiplicativity formula for the symmetric signature
(\cite[Proposition 8.1(i)]{MR566491}).  It seems
likely that his multiplicativity formula for the surgery obstruction 
(\cite[Proposition 8.1(ii)]{MR566491})
can also
be given a spectrum-level interpretation.
\end{remark}

We recall the definitions of symmetric monoidal functor and monoidal
transformation.  The theorem says that
$\bL$ (and similarly $\bP$) is a monoidal functor with the additional property
that the diagram
\[
\xymatrix{
\bL(R)\wedge \bL(S)
\ar[r]
\ar[d]
&
\bL(R\otimes S)
\ar[d]
\\
\bL(S)\wedge \bL(R)
\ar[r]
&
\bL(S\otimes R)
}
\]
strictly commutes.  Moreover, $\Sig$ has the property that the diagrams
\[
\xymatrix{
&
\bS
\ar[ld]
\ar[rd]
&
\\
\bP(e,*,1)
\ar[rr]^\Sig
&
&
\bL(\Z)
}
\]
and
\[
\xymatrix{
\bP(\pi,Z,w)\wedge\bP(\pi',Z',w')
\ar[r]^-{\Sig\wedge\Sig}
\ar[dd]
&
\bL(\Z[\pi]^w)\wedge\bL(\Z[\pi']^{w'})
\ar[d]
\\
&
\bL(\Z[\pi]^w\otimes\Z[\pi']^{w'})
\ar[d]
\\
\bP(\pi\times \pi',Z\times Z', w\cdot w')
\ar[r]^-\Sig
&
\bL(\Z[\pi\times\pi']^{w\cdot w'})
}
\]
strictly commute.

\section{Proof of Theorem \ref{Nov13.3}}
\label{l2}

The proof is a modification of the proof of Theorem \ref{m54}; the main
difference is that we need more elaborate notation.

\begin{notation}
\label{m83}
(i)
For an object $x$ of $\T$ or $\R$, write $\A_x$ (meaning $\A^x_{\text{Rel,sch}}$ for $x\in \R$) for the corresponding
$\Z$-graded category and $\bR_x$ for the associated object of $\sss$.

(ii)
Given a $j$-tuple $(x_1,\ldots,x_j)$, where each $x_i$ is an object of $\T$ or
$\R$, write 
\[
[x_1,\ldots,x_j]
\]
for $y_1\otimes\cdots\otimes y_j$, where $y_i$ is $x_i$ if $x_i$ is
an object of $\R$ and $\rho(x_i)$ if $x_i$ is an object of $\T$.

(iii)  Given a $j$-tuple $(f_1,\ldots,f_j)$, where each $f_i$ is a morphism in
$\T$ or $\R$, write
\[
[f_1,\ldots,f_j]
\]
for $g_1\otimes\cdots\otimes g_j$, where $g_i$ is $f_i$ if $f_i$ is a morphism
in $\R$ and $\rho(f_i)$ if $f_i$ is a morphism in $\T$.
\end{notation}

The reader should see
Proposition \ref{m88}(i)
below for motivation for the following
definitions.

\begin{definition}
(i) Let $y$ be an object of $\T$.  An {\it entity} 
of type $(r_u(j),y)$ is a
$j+1$-tuple $(x_1,\ldots,x_j,f)$, where $f$ is a morphism in $\T$ from
$x_1\boxtimes\cdots\boxtimes x_j$ to $y$.

(ii) Let $\sE_{r_u(j),y}$ denote the set of entities of type $(r_u(j),y)$.

(iii)
Let $z$ be an object of $\R$ and let $r:\{1,\ldots,j\}\to\{u,v\}$ be a
function.  An {\it entity} 
of type $(r,z)$ is a $j+1$-tuple 
$(x_1,\ldots,x_j,f)$,
where each $x_i$ is an object of $\T$ or $\R$ and $f$ is a morphism in $\R$
from $[x_1,\ldots,x_j]$ to $z$.

(iv) Let $\sE_{r,z}$ denote the set of entities of type $(r,z)$.
\end{definition}

\begin{notation}
(i)
Let $\fS$ denote the union of the set of objects of $\T$ and the set of objects
of $\R$.

(ii) Let $\Pi\sss$ be the infinite product of copies of $\sss$, indexed over
$\fS$.
\end{notation}

We will define a monad in $\Pi\sss$.

First we need to define the relevant right $\Sigma_j$ actions. Recall 
Definition \ref{m82}(i).

\begin{definition}
Let $\{\bX_x\}_{x\in\fS}$ be an object of $\Pi\sss$ and let $j\geq 0$.

(i) Given an object $y$ of $\T$ and $\alpha\in\Sigma_j$, define
a map $\bar{\alpha}$ from
\[
\bigvee_{(x_1,\ldots,x_j,f)\in\sE_{r_u(j),y}}
\O(r_u(j);u)
\owedge
(\bX_{x_1}\wedge\cdots\wedge\bX_{x_j})
\]
to itself
to be the map which takes the summand indexed by $(x_1,\ldots,x_j,f)$ to the
summand indexed by $(x_{\alpha(1)},\ldots,x_{\alpha(1)},f\circ\alpha)$ by means
of the map
\[
\O(r_u(j);u)
\owedge
(\bX_{x_1}\wedge\cdots\wedge\bX_{x_j})
\xrightarrow{\alpha\owedge\alpha}
\O(r_u(j);u)
\owedge
(\bX_{x_{\alpha(1)}}\wedge\cdots\wedge\bX_{x_{\alpha(j)}}).
\]

(ii) Given an object $z$ of $\R$ and $\alpha\in\Sigma_j$, define
a map $\bar{\alpha}$ from
\[
\bigvee_r
\bigvee_{(x_1,\ldots,x_j,f)\in\sE_{(r,z)}}
\O(r;v)
\owedge
(\bX_{x_1}\wedge\cdots\wedge\bX_{x_j})
\]
to itself to be the map which takes the summand indexed by
$(x_1,\ldots,x_j,f)\in \sE_{(r,z)}$ 
to the
summand indexed by $(x_{\alpha(1)},\ldots,x_{\alpha(1)},f\circ\alpha)\in
\sE_{(r^\alpha,z)}$ by means
of the map
\[
\O(r;v)
\owedge
(\bX_{x_1}\wedge\cdots\wedge\bX_{x_j})
\xrightarrow{\tilde{\alpha}\owedge\alpha}
\O(r^\alpha;v)
\owedge
(\bX_{x_{\alpha(1)}}\wedge\cdots\wedge\bX_{x_{\alpha(j)}}).
\]
\end{definition}

Note that this definition gives right $\Sigma_j$ actions on the objects
mentioned. 

\begin{definition}
Let $\{\bX_x\}_{x\in\fS}$ be an object of $\Pi\sss$.

(i) 
Given an object $y$ of $\T$,
define
\[
\bbP_y(\{\bX_x\}_{x\in\fS})
\]
to be
\[
\bigvee_{j\geq 0}
\Bigl(
\bigvee_{(x_1,\ldots,x_j,f)\in\sE_{r_u(j),y}}
\O(r_u(j);u)
\owedge
(\bX_{x_1}\wedge\cdots\wedge\bX_{x_j})
\Bigr)/\Sigma_j.
\]

(ii) Given an object $z$ of $\R$, define
\[
\bbP_z(\{\bX_x\}_{x\in\fS})
\]
to be
\[
\bigvee_{j\geq 0}
\Bigl(
\bigvee_r
\bigvee_{(x_1,\ldots,x_j,f)\in\sE_{(r,z)}}
\O(r;v)
\owedge
(\bX_{x_1}\wedge\cdots\wedge\bX_{x_j})
\Bigr)/\Sigma_j.
\]

(iii) Define $\bbP:\Pi\sss\to \Pi\sss$ to be the functor whose projection on
the $y$ factor (where $y$ is an object of $\T$) is $\bbP_y$ and whose 
projection on the $z$ factor (where $z$ is an object of $\R$) is $\bbP_z$.
\end{definition}

\begin{definition}
Let $\{\bX_x\}_{x\in\fS}$ be an object of $\Pi\sss$.

(i) For an object $y$ of $\T$, define
\[
\iota_y:\bX_y\to \bbP_y(\{\bX_x\}_{x\in\fS})
\]
to be the composite
\[
\bX_y
\cong
\boS\owedge \bX_y
=
\O(r_u(1);u)\owedge \bX_y
\hookrightarrow
\bbP_y(\{\bX_x\}_{x\in\fS}),
\]
where the last map is the inclusion of the summand corresponding to the entity 
$(y,\id)$.

(ii) For an object $z$ of $\R$, define
\[
\iota_z:\bX_z\to \bbP_z(\{\bX_x\}_{x\in\fS})
\]
to be the composite
\[
\bX_z
\cong
\boS\owedge \bX_z
=
\O(r_v(1);v)\owedge \bX_y
\hookrightarrow
\bbP_z(\{\bX_x\}_{x\in\fS}),
\]
where the last map is the inclusion of the summand corresponding to the entity
$(z,\id)$.

(iii) Define 
\[
\iota:\{\bX_x\}_{x\in\fS}\to \bbP(\{\bX_x\}_{x\in\fS})
\]
to be the map whose projection on the $y$ factor (where $y$ is an object
of $\T$) is $\iota_y$ and whose
projection on the $z$ factor (where $z$ is an object of $\R$) is $\iota_z$.
\end{definition}

In order to define the structure map $\mu:\bbP\bbP\to\bbP$ we need a
composition operation on entities.  For part (ii) we use Notation 
\ref{m83}(iii) and the notation of Definition \ref{m73}.

\begin{definition}
\label{m84}
Let $i\geq 0$, and for each $l$ with $1\leq l\leq i$ let $j_l\geq 0$.

(i) Let $y$ be an object of $\T$ and let 
\[
\bfe=(x_1,\ldots,x_i,f)\in \sE_{r_u(i),y}.
\]
For each $l$ with $1\leq l\leq i$ let
\[
\bfe_l=(x^{(l)}_1,\ldots,x^{(l)}_{j_l},f^{(l)})\in \sE_{r_u(j_l),x_l}.
\]
Define 
\[
\bfe\circ(\bfe_1,\ldots,\bfe_i)\in \sE_{r_u(j_1+\cdots+j_i),y}
\]
to be
\[
(x^{(1)}_1,\ldots,x^{(i)}_{j_i},g),
\]
where $g$ is the composite
\[
x^{(1)}_1\boxtimes\cdots\boxtimes x^{(i)}_{j_i}
\xrightarrow{f^{(1)}\boxtimes\cdots \boxtimes f^{(i)}}
x_1\boxtimes\cdots\boxtimes x_i
\xrightarrow{f} 
y.
\]

(ii) Let $z$ be an object of $\R$, let
$r:\{1,\ldots,i\}\to\{u,v\}$, and let 
\[
\bfe=(x_1,\ldots,x_i,f)\in \sE_{r,z}.
\]
For each $l$ with $1\leq l\leq i$ let 
$r_l:\{1,\ldots,j_l\}\to\{u,v\}$; assume that if $r(l)=u$ then $r_l$ is
$r_u(j_l)$.  Let
\[
\bfe_l=(x^{(l)}_1,\ldots,x^{(l)}_{j_l},f^{(l)})\in \sE_{r_l,x_l}.
\]
Define
\[
\bfe\circ(\bfe_1,\ldots,\bfe_i)\in \sE_{R,z}
\]
to be
\[
(x^{(1)}_1,\ldots,x^{(i)}_{j_i},g),
\]
where $g$ is the composite
\[
[x^{(1)}_1,\ldots,x^{(i)}_{j_i}]
\xrightarrow{[f^{(1)},\ldots,f^{(i)}]}
[x_1,\ldots,x_i]
\xrightarrow{f}
z.
\]
\end{definition}

Now we can define $\mu:\bbP\bbP\to\bbP$.  We begin with the projection on the
$y$-factor,
\[
\mu_y:\bbP_y\bbP\to\bbP_y,
\]  
where $y$ is an object of $\T$. 
A collection of entities $\bfe,\bfe_1,\ldots,\bfe_i$ as in Definition
\ref{m84}(i) determines a summand 
\[
\O(r_u(i);u)
\owedge
\Bigl(
\bigl(
\O(r_u(j_1);u)\owedge
(\bX_{x^{(1)}_1}\wedge\cdots\wedge \bX_{x^{(1)}_{j_1}})
\bigr)
\wedge\cdots
\Bigr)
\]
in $\bbP_y\bbP(\{\bX_x\}_{x\in\fS})$.  We define the restriction of $\mu_y$ 
to this summand to be the map
to the summand of $\bbP_y(\{\bX_x\}_{x\in\fS})$ 
indexed by $\bfe\circ(\bfe_1,\ldots,\bfe_i)$ which is induced (after
passage to quotients) by the composite
\begin{multline*}
\O(r_u(i);u)
\owedge
\Bigl(
\bigl(
\O(r_u(j_1);u)\owedge
(\bX_{x^{(1)}_1}\wedge\cdots\wedge \bX_{x^{(1)}_{j_1}})
\bigr)
\wedge\cdots
\Bigr)
\\
\to
\Bigl(\O(r_u(i);u)\owedge
\bigl(
\O(r_u(j_1);u)\otimes \cdots
\bigr)
\Bigr)
\owedge
(\bX_{x^{(1)}_1}\otimes\cdots\otimes \bX_{x^{(i)}_{j_i}})
\\
\xrightarrow{\gamma\owedge 1}
\O(r_u(j_1+\cdots+j_i);u)\owedge
(\bX_{x^{(1)}_1}\otimes\cdots\otimes \bX_{x^{(i)}_{j_i}}).
\end{multline*}
The projection of $\mu$ on the $z$ factor (where $z$ is an object of $\R$) is
defined similarly (using Definition \ref{m73}).

Next we give the action of $\bbP$ on the object $\{\bR_x\}_{x\in\fS}$.
Let $y$ be an object of $\T$ and let $(x_1,\ldots,x_j,f)$ be an entity of type
$(r_u(j),y)$.  
A slight modification of Definition \ref{m17}
gives a map
\[
\O(r_u(j);u)
\owedge
(\bR_{x_1}\wedge\cdots\wedge\bR_{x_j})
\to
\bR_{x_1\boxtimes\cdots\boxtimes x_j},
\]
and composing with the map induced by $f$ gives a map
\begin{equation}
\label{m87}
\O(r_u(j);u)
\owedge
(\bR_{x_1}\wedge\cdots\wedge\bR_{x_j})
\to
\bR_y.
\end{equation}
We define
\[
\nu_y:\bbP_y(\{\bR_x\}_{x\in\fS})
\to \bR_y
\]
to be the map whose restriction to the summand indexed by $(x_1,\ldots,x_j,f)$
is the map \eqref{m87}.
We define
\[
\nu_z:\bbP_z(\{\bR_x\}_{x\in\fS})
\to \bR_z
\]
similarly when $z$ is an object of $\R$
(using a slight modification of Definition \ref{m85}), and we define
\[
\nu:\bbP(\{\bR_x\}_{x\in\fS})
\to
\{\bR_x\}_{x\in\fS}
\]
to be the map with projections $\nu_y$ and $\nu_z$.

\begin{lemma}
$\nu$ is an action of $\bbP$ on $\{\bR_x\}_{x\in\fS}$.
\qed
\end{lemma}

Now we need the analogue of Lemma \ref{m78}.
Let $\C$ be the category whose
objects are triples $({\mathbf F},{\mathbf G},{\mathbf t})$, where ${\mathbf
F}$ is a monoidal functor $\T\to \sss$, 
${\mathbf G}$ is a monoidal functor $\R\to \sss$, and $\mathbf t$ is
a natural transformation ${\mathbf F}\to {\mathbf G}\circ \rho$
which is not required to be a monoidal transformation;
the morphisms are commutative diagrams
\[
\xymatrix{
{\mathbf F}
\ar[d]_{\mathbf t}
\ar[r]
&
{\mathbf F}'
\ar[d]^{{\mathbf t}'}
\\
{\mathbf G}\circ\rho
\ar[r]
&
{\mathbf G}'\circ\rho,
}
\]
where the horizontal arrows are monoidal transformations.

Let us write $\bR_{\mathrm{geom}}$ (resp., $\bR_{\mathrm{sym}}$) for the 
functor $\T\to\sss$ (resp., $\R\to\sss$) which takes $x$
to $\bR_x$.

\begin{lemma}
\label{m93}
{\rm (i)} There is a functor $\Upsilon$ from $\bbP$ algebras to $\C$ which
takes $\{\bX_x\}_{x\in\fS}$ to a triple $({\mathbf F},{\mathbf G},{\mathbf t})$
with ${\mathbf F}(y)=\bX_y$ and ${\mathbf G}(z)=\bX_z$.

{\rm (ii)} $\Upsilon(\{\bR_x\}_{x\in\fS})$ is the triple
\[
(\bR_{\mathrm{geom}},
\bR_{\mathrm{sym}},
\Sigr)
\]
\end{lemma}

\begin{proof}
Part (i).  Let $\bX= \{\bX_x\}_{x\in\fS}$ be a $\bbP$ algebra.  Define a functor
\[
{\mathbf F}: \T\to \sss
\]
on objects by ${\mathbf F}(y)=\bX_y$  and on morphisms 
by letting ${\mathbf F}(f:y\to y')$ be the composite
\[
\bX_y
\cong
\O(r_u(1);u)\owedge\bX_y
\hookrightarrow
\bbP_{y'}(\{\bX_x\}_{x\in\fS})
\xrightarrow{\nu_{y'}}
\bX_{y'},
\]
where the unlabeled arrow is the inclusion of the summand indexed by the entity
$(y,f)$. The functoriality of ${\mathbf F}$ follows from the commutativity of the diagram
\[
\xymatrix{ 
\bbP \bbP \bX \ar[r]^{\mu_{\bX}}\ar[d]^{\bbP \nu}& \bbP \bX \ar[d]^\nu\\ \bbP \bX \ar[r]^\nu & \bX
}
\]
and the definition of $\mu$ in Definition \ref{m84}.
We define 
\[
{\mathbf G}: \R\to \sss
\]
similarly.  The proof that ${\mathbf F}$ and $\mathbf G$ are monoidal functors
is similar to the argument, in
the proof in Lemma \ref{m78}(i), that $\bX$ and $\bY$ are monoids. The functoriality is as above.

It remains to give the natural transformation 
\[
{\mathbf t}: {\mathbf F} \to {\mathbf G}\circ\rho.
\]
For an object $y$ of $\T$, let ${\mathbf t}_y$ be the composite
\[
\bX_y
\cong
\O(r_u(1);v)\owedge\bX_y
\hookrightarrow
\bbP_{\rho(y)}(\{\bX_x\}_{x\in\fS})
\xrightarrow{\nu_{\rho(y)}}
\bX_{\rho(y)},
\]
where the unlabeled arrow is the inclusion of the summand indexed by the entity
$(y,\id)\in\sE_{r_u(1),\rho(y)}$.
To show
that $\mathbf t$ is a natural transformation, let $f:y\to y'$ be a morphism in
$\T$, and let $z=\rho(y)$, $z'=\rho(y')$.
Let $i_1$ be the composite
\begin{multline*}
\bX_y
\cong
\O(r_u(1);u)\owedge \bX_y
\hookrightarrow
\bbP_{y'}(\{\bX_x\}_{x\in \fS})
\cong
\O(r_u(1);v)\owedge \bbP_{y'}(\{\bX_x\}_{x\in \fS})
\\
\hookrightarrow
\bbP_{z'}\bbP(\{\bX_x\}_{x\in \fS}),
\end{multline*}
where the first arrow is the inclusion of the summand indexed by the
entity $(y,f)$ and the second is the inclusion of the summand indexed by
$(y',\id)$.
Let $j_1$ be the composite
\[
\bX_{y'}
\cong
\O(r_u(1);v)\owedge\bX_{y'}
\hookrightarrow
\bbP_{z'}(\{\bX_x\}_{x\in \fS}),
\]
where the inclusion is indexed by $(y',\id)$,
and let $j_2$ be the composite
\[
\bX_y
\cong
\O(r_u(1);v)\owedge\bX_y
\hookrightarrow
\bbP_{z'}(\{\bX_x\}_{x\in \fS}),
\]
where the inclusion is indexed by $(y,\rho(f))$.

Consider the commutative diagram
\[
\xymatrix{
\bX_y
\ar[rr]^-=
\ar[rd]^{i_1}
\ar[dd]_{{\mathbf F}(f)}
&&
\bX_y
\ar[d]^{j_2}
\\
& 
\bbP_{z'}\bbP(\{\bX_x\}_{x\in \fS})
\ar[d]_{\bbP_{z'}\nu}
\ar[r]^-\mu
&
\bbP_{z'}(\{\bX_x\}_{x\in \fS})
\ar[d]^\nu
\\
\bX_{y'}
\ar[r]^-{j_1}
&
\bbP_{z'}(\{\bX_x\}_{x\in \fS})
\ar[r]^-\nu
&
\bX_{z'}.
}
\]
Let $H$ denote the composite of the right-hand vertical arrows.  Then the
diagram shows that the composite
\begin{equation}
\label{m91}
\bX_y
\xrightarrow{{\mathbf F}(f)}
\bX_{y'}
\xrightarrow{\mathbf t}
\bX_{z'}
\end{equation}
is $H$.

Let $i_2$ be the composite
\begin{multline*}
\bX_y
\cong
\O(r_u(1);v)\owedge \bX_y
\hookrightarrow
\bbP_{z}(\{\bX_x\}_{x\in \fS})
\cong
\O(r_v(1);v)\owedge \bbP_{z}(\{\bX_x\}_{x\in \fS})
\\
\hookrightarrow
\bbP_{z'}\bbP(\{\bX_x\}_{x\in \fS}),
\end{multline*}
where the first inclusion is indexed by $(y,\id)$ and
the second is indexed by $(z,\rho(f))$.
Let $j_2$ be as above and let $j_3$ be the composite
\[
\bX_z
\cong
\O(r_v(1);v)\owedge \bX_z
\hookrightarrow
\bbP_{z}(\{\bX_x\}_{x\in \fS}),
\]
where the inclusion is indexed by $(z,\rho(f))$.

Consider the commutative diagram
\[
\xymatrix{
\bX_y
\ar[rr]^-=
\ar[rd]^{i_1}
\ar[dd]_{{\mathbf t}}
&&
\bX_y
\ar[d]^{j_2}
\\
& 
\bbP_{z'}\bbP(\{\bX_x\}_{x\in \fS})
\ar[d]_{\bbP_{z'}\nu}
\ar[r]^-\mu
&
\bbP_{z'}(\{\bX_x\}_{x\in \fS})
\ar[d]^\nu
\\
\bX_z
\ar[r]^-{j_3}
&
\bbP_{z'}(\{\bX_x\}_{x\in \fS})
\ar[r]^-\nu
&
\bX_{z'}.
}
\]
This diagram shows that the composite
\begin{equation}
\label{m92}
\bX_y
\xrightarrow{{\mathbf t}}
\bX_z
\xrightarrow{{\mathbf G}(f)}
\bX_{z'}
\end{equation}
is also $H$, so the composites \eqref{m91} and \eqref{m92} are equal as reqired.

Part (ii) is an easy consequence of the definitions.
\end{proof}

Finally, we have the analogues of Definition \ref{m86} and Proposition
\ref{m79}.

\begin{definition}
Let $\{\bX_x\}_{x\in\fS}$ be an object of $\Pi\sss$.

(i)
Given an object $y$ of $\T$,
define
\[
\bbP'_y(\{\bX_x\}_{x\in\fS})
\]
to be
\[
\bigvee_{j\geq 0}
\Bigl(
\bigvee_{((x_1,\ldots,x_j,f)\in\sE_{r_u(j),y}}
\bX_{x_1}\wedge\cdots\wedge\bX_{x_j}
\Bigr)/\Sigma_j.
\]

(ii) Given an object $z$ of $\R$, define
\[
\bbP'_z(\{\bX_x\}_{x\in\fS})
\]
to be
\[
\bigvee_{j\geq 0}
\Bigl(
\bigvee_r
\bigvee_{(x_1,\ldots,x_j,f)\in\sE_{(r,z)}}
\bX_{x_1}\wedge\cdots\wedge\bX_{x_j}
\Bigr)/\Sigma_j.
\]

(iii) Define $\bbP':\Pi\sss\to \Pi\sss$ to be the functor whose projection on
the $y$ factor (where $y$ is an object of $\T$) is $\bbP'_y$ and whose
projection on the $z$ factor (where $z$ is an object of $\R$) is $\bbP'_z$.
\end{definition}

A routine modification of Definition \ref{m86}(ii) gives a natural
transformation
\[
\Xi:\bbP\to\bbP'.
\]

\begin{prop}
\label{m88}
{\rm (i)} An algebra over $\bbP'$ is the same thing as a pair of
symmetric monoidal functors ${\mathbf F}$ and $\mathbf G$ with a monoidal
transformation ${\mathbf F}\to {\mathbf G}\circ \rho$.

{\rm (ii)} $\Xi$ is a map of monads.

{\rm (iii)} Suppose that each $(X_x)_k$ has compatible
degeneracies (see Definition \ref{m31}).  Let $\bbP^q$ denote the $q$-th 
iterate of $\bbP$.  Then each map
\[
\Xi:\bbP^q(\{\bX_x\}_{x\in\fS})\to \bbP'\bbP^{q-1}(\{\bX_x\}_{x\in\fS})
\]
is a weak equivalence.
\end{prop}

\begin{proof}
Part (i). 
Let $\{\bX_x\}_{x\in\fS}$ be an algebra over $\bbP'$.  The fact that ${\mathbf
F}$ and $\mathbf G$ are symmetric monoidal functors
is an easy consequence of the definitions.  The natural transformation
${\mathbf t}:{\mathbf F}\to {\mathbf G}\circ \rho$
is constructed as in the proof of Lemma \ref{m93}(i).  The proof that $\mathbf
t$ is monoidal is similar to the proofs of Proposition \ref{m79}(i) and Lemma
\ref{m93}(i), using the maps
\begin{multline*}
i_1: \bX_{y_1} \wedge \cdots\wedge \bX_{y_j}
\hookrightarrow
\bbP'_{y_1\boxtimes\cdots\boxtimes y_j}(\{\bX_x\}_{x\in\fS})
\hookrightarrow
\bbP'_{\rho(y_1\boxtimes\cdots\boxtimes y_j)}
\bbP'(\{\bX_x\}_{x\in\fS}),
\end{multline*}
(where the first inclusion is indexed by $(y_1,\ldots,y_j,\id)$ and the second
by $(y_1\boxtimes\cdots\boxtimes y_j,\id))$,
and 
\begin{multline*}
i_2:  
\bX_{y_1} \wedge \cdots\wedge \bX_{y_j}
\hookrightarrow
\bbP'_{\rho(y_1)}(\{\bX_x\}_{x\in\fS})\wedge \cdots\wedge 
\bbP'_{\rho(y_j)}(\{\bX_x\}_{x\in\fS})
\\
\hookrightarrow
\bbP'_{\rho(y_1\boxtimes\cdots\boxtimes y_j)}
\bbP'(\{\bX_x\}_{x\in\fS}),
\end{multline*}
(where the first map is the smash product of the inclusions indexed by
$(y_i,\id)$ and the second is indexed by $(\rho(y_1),\ldots,\rho(y_j),
\rho(y_1)\otimes\cdots\otimes\rho(y_j)\to \rho(y_1\boxtimes\cdots\boxtimes
y_j))$).

Part (ii) is immediate from the definitions, and the proof of part
(iii) is the same as for Proposition \ref{m34}(iii) (but using Remark
\ref{m59}(ii)).
\end{proof}

Now the proof of Theorem \ref{Nov13.3} is the same as the proof of Theorem
\ref{m54} given in Section \ref{rect2}, with only the notation changed.

\appendix

\section{ A property of the smash product of symmetric spectra.}\label{free}

For an object $\bX$ of $\sss$ let 
\[
\op:X_k\to X_{k+1}
\]
be the map
which takes $x$ to $\omega(s\wedge x)$ (where $s$ is the
1-simplex of $S^1$). More precisely, $\op$ is the $k$-fold multisemisimplicial map which takes $(X_k)_{n_1,\ldots, n_k}$ to 
$(X_k)_{1,n_1,\ldots, n_k}$ obtained from the suspension map of the spectrum $\bX$.
\begin{definition}
An object $\bX$ of $\sss$ is {\it monomorphic} if 
$\op$ is a monomorphism for every $k$.  It is {\it strongly monomorphic} if it
is monomorphic and
has the following property: if
$x\in X_k$ and $\alpha\in \Sigma_{k+1}$ with $\alpha(1)\neq
1$, and if $\alpha\op x$ is in the image of $\op$, then there is a
$\beta\in \Sigma_k$ such that $\beta^{-1}(1)=\alpha^{-1}(1)-1$ and $\beta x$
is in the image of $\op$.
\end{definition}

The purpose of this appendix is to prove the following fact, which is used
in Section \ref{rect}:


\begin{prop}\label{PropA2}
\label{f32}
Let $\bX$ be strongly monomorphic and suppose that the only element of $X_0$ is
the basepoint.  Then the $\Sigma_j$ action on $\bX^{\wedge j}$ which permutes
the factors is free away from the basepoints.
\end{prop}

\begin{remark}\label{RemA3}
(i) The main  object 
$\bR$ given in Example \ref{Aug29.1} is strongly monomorphic: for an ad $F\in R_k$, the suspension map was defined by $\omega (s \wedge F)=\lambda^*(F)$ where $\lambda $ is the incidence-compatible isomorphism from $\C ell(\Delta^1 \times K, \partial \Delta^1 \times K)$ to $\C ell( K)$. Suppose we are  given a multisimplex of the form $\alpha \bar{\omega}F$ which is in the image of $\bar{\omega}$ as above. Then it  necessarily defines a functor with source 
$$\C ell(\Delta^1 \times K_1 \times \Delta^1 \times K_2, (\partial \Delta^1 \times K_1 \times \Delta^1 \times K_2)\cup (\Delta^1 \times K_1 \times \partial \Delta^1 \times K_2) )$$
for some products of simplices $K_1$,$K_2$. Thus a permutation $\beta$ which exchanges the factors $K_1$ and $\Delta^1$ has the property that $\beta F$ is in the image of $\bar{\omega}$.

(ii) The analog of Proposition \ref{f32} for simplicial or topological symmetric
spectra is also true, with essentially the same proof.
\end{remark}

The proof of Proposition \ref{f32} will be given after the proof of our next
result, which is the main ingredient in the proof of Proposition \ref{f32}.

\begin{prop}
\label{f1}
If $\bX_1,\ldots,\bX_j$ are strongly monomorphic then 
$
\bX_1\wedge\cdots\wedge\bX_j
$
is monomorphic.
\end{prop}

We will give an example at the end of this appendix to show that
Propositions \ref{f32} and \ref{f1} both fail if ``strongly monomorphic'' is 
replaced by ``monomorphic.''

Before we can give the proof of
Proposition \ref{f1} we will need quite a bit of background, culminating with
Lemma \ref{f8}.

If $A$ is a multisemisimplicial set and $a\in A_{\mathbf n}$ for some $\mathbf
n$, we will say {\it $a$ is a point of $A$}.  If $\bX$ is an object of $\sss$
and $x$ is a point of $X_k$ for some $k$ we will say {\it $x$ is a point of 
$\bX$} and write $|x|=k$.

We will use the fact that (by Remark \ref{Aug29.2}) 
points in the $k$-fold smash product
$
\bX_1\wedge\cdots\wedge\bX_j
$
are equivalence classes of symbols of the form
\[
\theta\wedge x_1\wedge\cdots\wedge x_j
\]
where $x_i$ is a point of $\bX_i$ for each $i$ and $\theta\in
\Sigma_{|x_1|+\cdots+|x_j|}$.
Our first task is to describe
the equivalence relation 
$\sim$ explicitly, and for this we need the operations in Definition \ref{f3} 
below.

\begin{notation}\label{not1}
(i)
Let $\alpha\in\Sigma_l$, let $L\geq 1$, and let $k\geq L+l-1$.  We will 
write 
$\alpha^{[L]}$ for the element of $\Sigma_k$ which permutes
$L,L+1,\ldots,L+l-1$ in the same way that $\alpha$ permutes $1,\ldots,l$.

(ii)
Let $\bX$ be an object of $\sss$. 
Let 
\[
\psi:\bX\wedge \bS\to\bX
\]
be the right action (that is, $\psi=\omega\circ\tau$, where $\tau$ is defined
in Equation \eqref{m8}), 
and let
\[
\pp:X_k\to X_{k+1}
\]
be the map that takes $x$ to $\psi(x\wedge s)$, where $s$ is the
1-simplex of $S^1$.  Note that 
\begin{equation}
\label{f5}
\pp x=\rho_{1,k}\op x, 
\end{equation}
where 
$\rho_{1,k}$ is
defined after Equation \eqref{m8}.
\end{notation}

\begin{remark}
\label{f11}
(i)
If $x$ is a point of $\bX$ and $\alpha\in\Sigma_{|x|}$ then 
$\op \alpha x=\alpha^{[2]}\op x$ and $\pp\alpha x=\alpha^{[1]}\pp x$.

(ii) $\pp$ commutes with $\op$.

(iii)
Suppose that $\bX$ is strongly monomorphic $\sss$.
Using Equation \ref{f5} we see that if $x$ is a point of $\bX$ and
$\gamma$ is an element of $\Sigma_{|x|+1}$ for which $\gamma(1)\neq |x|+1$ and
$\gamma\op x$ is in the image of $\pp$, then there is a $\delta\in\Sigma_{|x|}$
for which $\delta^{-1}(|x|)=\gamma^{-1}(|x|+1)-1$ and $\delta x$ is in the image
of $\pp$.
\end{remark}

\begin{definition}
\label{f3}
Let $\bX_1,\ldots,\bX_j$ be monomorphic objects of $\sss$.
Let $x_i$ be a point of $\bX_i$ for $1\leq i\leq j$, and let
$\theta\in\Sigma_{|x_1|+\cdots+|x_j|}$.

(i) For $1\leq m\leq j$ and $\alpha$ an element of $\Sigma_{|x_m|}$ other than
the identity define
\[
A_{m,\alpha}(\theta\wedge x_1\wedge\cdots\wedge x_j)=
(\theta \circ (\alpha^{[1+\sum_{i<m} |x_i|]})^{-1})\wedge y_1\wedge\cdots \wedge y_j,
\]
where $\alpha^L$ was defined in Notation \ref{not1} and
\[
 y_i=
\begin{cases}
x_i & \text{if $i\neq m$}\\
\alpha x_m & \text{if $i=m$.}
\end{cases}
\]

(ii) Suppose that
$m<j$ and that $x_m=\pp z$ for some $z$ (in which case $z$ is uniquely
determined since $\bX_m$ is monomorphic) define
\[
B_m(\theta\wedge x_1\wedge\cdots\wedge x_j)=
\theta\wedge y_1\wedge\cdots \wedge y_j,
\]
where
\[
 y_i=
\begin{cases}
x_i & \text{if $i\neq m,m+1$}\\
z & \text{if $i=m$} \\
\op x_{m+1} & \text{if $i=m+1$.}
\end{cases}
\]

(iii) If $m<j$ and $x_{m+1}=\op z$ for some $z$ define
\[
C_m(\theta\wedge x_1\wedge\cdots\wedge x_j)=
\theta\wedge y_1\wedge\cdots \wedge y_j,
\]
where 
\[
 y_i=
\begin{cases}
x_i & \text{if $i\neq m,m+1$}\\
\pp x_{m} & \text{if $i=m$} \\
z & \text{if $i=m+1$}
\end{cases}
\]
\end{definition}

Now we can give an explicit description of the equivalence relation $\sim$, as
follows:
$\theta\wedge 
x_1\wedge\cdots\wedge x_j\sim
\kappa\wedge y_1\wedge\cdots\wedge y_j$
if and only if there is a composable sequence $D_1,\ldots,D_n$ of operations of 
the types given in Definition \ref{f3} with 
\[
\kappa\wedge y_1\wedge\cdots\wedge y_j
=
D_n\cdots D_1(\theta\wedge x_1\wedge\cdots\wedge x_j).
\]
We will only indicate why these operation generate: the operations $A_{m,\alpha}$ take care of the equivalences described in Remark \ref{Aug29.2}, the operations $B_m$ and $C_m$ generate the equivalences which come from the coequalizer  diagram in Definition \ref{m19}.

In the situation of an equivalence, we will say that the $n$-tuple $P=(D_1,\ldots,D_n)$ is 
a {\it path} from $\theta\wedge x_1\wedge\cdots\wedge x_j$ to 
$\kappa\wedge y_1\wedge\cdots\wedge y_j$ and we will write
\[
P(\theta\wedge x_1\wedge\cdots\wedge x_j)
=
\kappa\wedge y_1\wedge\cdots\wedge y_j.
\]  
The {\it length} of $P$ is the
number of $A$-operations in $P$ plus twice the number of $B$-operations and 
twice the number of $C$-operations.

\begin{notation}
If $\alpha$ is the identity element then $A_{m,\alpha}$ will be interpreted as
the empty path.
\end{notation}

We record some useful calculations in our next two lemmas.
We define the {\it standard path of length
$3p$ starting at
$\theta\wedge \op x_1\wedge x_2\wedge\cdots\wedge x_j$} with $p<j$  to be the path
\[
A_{1,\rho_{1,|x_1|}},B_1,A_{2,\rho_{1,|x_2|}},B_2,\ldots,A_{p,\rho_{1,|x_p|}},B_p
\]
Note that the $B$-operations in this sequence are always possible because of
Equation \eqref{f5} and that $A_{i,\rho_{1,|x_i|}}$ is interpreted as
the empty path when $|x_i|=0$.

\begin{lemma}
\label{f4}
Let $P$ be the standard path of length $3p$ starting at $\theta\wedge \op
x_1\wedge x_2\wedge\cdots\wedge x_j$.  Then
\[
P(\theta\wedge\op x_1\wedge\cdots\wedge x_j)
=
\theta(\rho_{1,|x_1|+\cdots+|x_p|}^{[1]})^{-1}
\wedge x_1\wedge\cdots\wedge \op x_{p+1}\wedge
\cdots
\]
\end{lemma}

\begin{proof}
This follows by induction on $p$ from Equation \ref{f5} and the equation
$
\rho_{1,t}^{[s+1]}
\rho_{1,s}^{[1]}
=\rho_{1,s+t}^{[1]}
$ in the group $\Sigma_{s+t+1}$.
\end{proof}

\begin{remark}
\label{f6}
The operations in Definition \ref{f3} often
commute with each other.  Specifically, $A_{m,\alpha}$ commutes with
$A_{n,\beta}$ for $m\neq n$ and with $B_n$ and $C_n$ for $m\neq n,n-1$;
moreover $B_m$ commutes with every $B_n$ and every $C_n$ and $C_m$
commutes with every $C_n$.
\end{remark}

\begin{lemma}
\label{f4.1}
Let $P$ be as in Lemma \ref{f4}, and let $Q$ be any path starting at
$\theta\wedge \op
x_1\wedge x_2\wedge\cdots\wedge x_j$ of the form
\[
A_{1,\alpha_1},B_1,A_{2,\alpha_2},B_2,\ldots,A_{p,\alpha_p},B_p
\]
where $\alpha_i(1)=|x_i|+1$ for each $i$.  For each $i$ let $\beta_i$ be the
element of $\Sigma_{|x_i|}$ with $\beta_i^{[2]}=
\rho_{1,|x_i|}^{-1} \alpha_i$ (which exists because $\rho_{1,|x_i|}^{-1}
\alpha_i$ takes 1 to 1).  Then

{\rm (i)}
the path 
\[
A_{1,\beta_1^{[2]}},A_{2,\beta_2},\ldots,
A_{p,\beta_p},P
\]
has the same effect on $\theta\wedge \op
x_1\wedge x_2\wedge\cdots\wedge x_j$ as $Q$, and
\end{lemma}

{\rm (ii)} 
the path
\[
P,A_{1,\beta_1},A_{2,\beta_2},\ldots,A_{p,\beta_p}
\]
also has the same effect on $\theta\wedge \op
x_1\wedge x_2\wedge\cdots\wedge x_j$ as $Q$.

\begin{proof}
The proof in each case is by induction on $p$.

For part (i), the case $p=1$ is immediate.  For $p>1$, we observe that
\[
A_{1,\beta_1^{[2]}},A_{2,\beta_2},\ldots,
A_{p-1,\beta_{p-1}},P,A_{p,\alpha_p},B_p
\]
has the same effect as 
\[
A_{1,\beta_1^{[2]}},A_{2,\beta_2},\ldots,
A_{p-1,\beta_{p-1}},P,A_{p,\beta_p^{[2]}},A_{p,\rho_{1,|x_p|}},B_p
\]
and by Lemma \ref{f4} this has the same effect as
\[
A_{1,\beta_1^{[2]}},A_{2,\beta_2},\ldots,
A_{p,\beta_{p}},P,A_{p,\rho_{1,|x_p|}},B_p
\]
as required.

For part (ii), we use the equation
$\rho_{1,|x_p|}\beta_p^{[2]}= \beta_p^{[1]}\rho_{1,|x_p|}$. 
This shows that 
\[
P,A_{1,\beta_1},A_{2,\beta_2},\ldots,A_{{p-1},\beta_{p-1}},A_{p,\alpha_{p}},
B_{p}
\]
has the same effect as 
\[
P,A_{1,\beta_1},A_{2,\beta_2},\ldots,A_{{p-1},\beta_{p-1}},
A_{p,\rho_{1,|x_p|}},
A_{p,\beta_p^{[1]}},B_p
\]
which has the same effect as
\[
P,A_{1,\beta_1},A_{2,\beta_2},\ldots,A_{{p-1},\beta_{p-1}},
A_{p,\rho_{1,|x_p|}},
B_p,
A_{p,\beta_p}
\]
and by Remark \ref{f6} this has the same effect as
\[
P, A_{p,\rho_{1,|x_p|}},
B_p,
A_{1,\beta_1},A_{2,\beta_2},\ldots,A_{{p},\beta_{p}}
\]
as required.
\end{proof}

\begin{notation}
Given a symbol $\mathbf x=\theta\wedge x_1\wedge\cdots\wedge x_j$ as above, meaning a representative of the equivalence relation, we
write $\op \mathbf x$ for $\theta^{[2]}\wedge \op 
x_1\wedge\cdots\wedge x_j$.
\end{notation}

\begin{definition}
Given symbols
$\mathbf x$
and $\mathbf y$, and a path $P$ from 
$\op \mathbf x$
to 
$\op \mathbf y$, we will say that $\mathbf x$ and $\mathbf y$ are {\it 
$P$-related}.
\end{definition}

With this terminology, Proposition \ref{f1} is true if the following statement
is true for all $P$:
\[
\tag{*} \text{
if $\mathbf x$ and $\mathbf y$ are $P$-related then $\mathbf x\sim
\mathbf y$.}
\] 

Before proving Proposition \ref{f1} we need two more lemmas.

\begin{lemma}
\label{f7}
Let $\mathbf x=\theta\wedge x_1\wedge\cdots\wedge x_j$ and let $P$ be a path
starting at $\mathbf x$.
Suppose that $P$ consists of a single operation $D$ and that

{\rm (i)} if $D=A_{1,\alpha}$ then $\alpha(1)=1$, and

{\rm (ii)} if $D=B_1$ then $|x_1|\neq 0$.

Then
there is a unique $\mathbf y$ with $P(\op\mathbf x)=\op\mathbf y$, and
$\mathbf x\sim \mathbf y$.  In particular, Statement {\rm (*)} is true for
$P$.
\end{lemma}

\begin{proof}
If $D=A_{1,\alpha}$ with $\alpha(1)=1$ then
there is a $\beta$ with $\alpha=\beta^{[2]}$, and
\[
P(\op\mathbf x)=
P(\theta^{[2]}\wedge \op 
x_1\wedge\cdots\wedge x_j)
=
\theta^{[2]}(\beta^{[2]})^{-1}\wedge\op \beta x_1\wedge\cdots\wedge x_j,
\]
so we can let 
\[
\mathbf y=
\theta\beta^{-1}\wedge\beta x_1\wedge\cdots\wedge x_j
\]
which is the unique choice since $\bX_1$ is monomorphic.
We have $\mathbf y=A_{1,\beta}\mathbf x$, so $\mathbf x\sim
\mathbf y$.

Next suppose $D=B_1$ and that $|x_1|\neq 0$. Since $B_1$ can be applied to 
$\mathbf x$ and we have (by definition of $B_1$) that $x_1=\pp z$, and so 
\begin{equation}
\label{f33}
\op x_1= \op \pp z= \pp \op z
\end{equation}
for some $z$, and then
\[
P(\op\mathbf x)= \theta^{[2]}\wedge z\wedge \op x_2\cdots
\]
Combining Equations
\ref{f33} and \ref{f5} we have $\rho_{1,|x_1|}\op z=\op x_1$, so since 
$\bX_1$ is strongly monomorphic there is a $\beta\in\Sigma_{|x_1|}$ with 
$\beta^{-1}(1)=1$ and 
\[
\beta z=\op w
\]
for some $w$. There is a 
$\tilde{\beta}\in\Sigma_{|x_1|-1}$ with $\beta=\tilde{\beta}^{[2]}$ and we have 
\begin{equation}
\label{f35}
z=(\tilde{\beta}^{[2]})^{-1}\op w=\op \tilde{\beta}^{-1}w.
\end{equation}
Now let
\[
\mathbf y= \theta\wedge \tilde{\beta}^{-1}w \wedge \op x_2\cdots
\]
We have $\op \mathbf y=P(\op \mathbf x)$ by Equation \ref{f35}.
We claim that
\[
x_1=\pp\tilde{\beta}^{-1}w
\]
Assuming this for the moment, we have $\mathbf y=B_1\mathbf x$, so $\mathbf
x\sim
\mathbf y$ as required.  
Since $\bX_1$ is monomorphic the claim follows from the equations
\[
\op x_1=\pp z=\pp \op \tilde{\beta}^{-1}w
=\op\pp\tilde{\beta}^{-1}w,
\]
where we have used Equations \ref{f33} and \ref{f35} and Remark \ref{f11}(ii).

The remaining cases are easy.
\end{proof}

\begin{lemma}
\label{f8}
Let 
$Q$ be a path which can be written as a composite $Q_1,Q_2$, where Statement
{\rm (*)} is true for $Q_2$ and every operation in
$Q_1$ satisfies the hypothesis of Lemma \ref{f7}.
Then Statement {\rm (*)} is true for $Q$. 
\end{lemma}

\begin{proof}
By an iterated application of Lemma \ref{f7}, 
$Q_1(\op\mathbf x)$ has the form $\op\mathbf z$ for some $\mathbf z$,
and $\mathbf x\sim\mathbf z$.  But also $\mathbf z\sim\mathbf y$ since
the symbols
$\mathbf z$ and $\mathbf y$ are $Q_2$-related.
\end{proof}

\begin{proof}[Proof of Proposition \ref{f1}]

We will prove that Statement (*) holds for all $P$. So suppose that $P$ is a
path of length $r$ from $\op\mathbf x$ to
$\op\mathbf y$.

If $r=0$ then $\mathbf x=\mathbf y$ since $\bX_1$ is monomorphic.

Suppose that $r>0$ and that the result holds for all paths of length $<r$.

Let \[
\mathbf x=\theta\wedge x_1\wedge\cdots\wedge x_j
\]
and let 
\[
\mathbf y=\kappa\wedge y_1\wedge\cdots\wedge y_j.
\]
Let $p$ be the
largest number for which $P$ begins with a path of the form
\[
A_{1,\alpha_1},B_1,A_{2,\alpha_2},B_2,\ldots,A_{p,\alpha_p},B_p
\]
where $\alpha_i(1)=|x_i|+1$ for each $i$.
Denote this path by $P_1$  
($p$ is allowed to be 0 in which case $P_1$ is the empty path). 

Lemmas \ref{f4.1}(i) and \ref{f4} imply that $P\neq
P_1$:
if $|x_1|=0$ this is
because $x_1$ cannot equal $\op y_1$, and 
otherwise it's because
$\theta^{[2]}(\beta_1^{[2]})^{-1} \cdots
(\beta_p^{2+\sum_{i< p}|x_i|]})^{-1}
(\rho_{1,\sum_{i\leq p}|x_i|}^{[1]})^{-1}$
does not take 1 to 1 and so cannot be
equal to $\kappa^{[2]}$.
Let $D$ be the next operation in $P$ and let $P_2$ be the part of $P$ 
after $D$.
There are three cases.

\medskip
\noindent
{\bf Case 1.} Suppose that $D=A_{m,\gamma}$ for some $\gamma$, with $m\neq 
p+1$.  

If $m>p+1$ then $D$ commutes with all the earlier operations in $P$ (if any), 
so $P$ has the same effect as
\[
D,P_1,P_2.
\]
Since the length of $P_1,P_2$ is $r-1$ 
we are done by Lemma \ref{f8} (taking $Q_1=D$ in that lemma).

If $m\leq p$, Lemma \ref{f4.1}(ii) shows that $P_1,D$ has the same effect
as
\[
Q=(A_{1,\alpha'_1},B_1,A_{2,\alpha'_2},B_2,\ldots,A_{p,\alpha'_p},B_p)
\]
where 
\[
\alpha'_i=
\begin{cases}
\alpha_i & \text{if $i\neq m$,} \\
\rho_{1,|x_m|}\gamma^{[2]}\rho_{1,|x_m|}^{-1}\alpha_m & \text{if $i=m$.}
\end{cases}
\]
Then $P$ has the same effect as $Q,P_2$, which has length $r-1$.

\medskip
\noindent
{\bf Case 2.} Suppose that $D$ is $B_m$ or $C_m$ for some $m$.

If $m>p$ then $D$ commutes with all earlier operations in $P$ (if any), so we
are done by Lemma \ref{f8} (taking $Q_1=D$; note that if $p=0$ and $|x_1|=0$ 
then we cannot have $D=B_1$ because of the way $p$ was chosen).

Let $P_1'$ be the standard path of length $3p$ starting at $\op \mathbf x$ and
let $Q$ be the path $P_1',D,P_2$.
Lemma \ref{f4.1}(i) says that $P_1$ has the same effect as
\[
A_{1,\beta_1^{[2]}},A_{2,\beta_2},\ldots,
A_{p,\beta_p},P_1'
\]
By Lemma \ref{f7}, there is a symbol $\mathbf z$ such that the path 
\[
A_{1,\beta_1^{[2]}},A_{2,\beta_2},\ldots,
A_{p,\beta_p}
\]
takes $\op \mathbf x$ to $\op\mathbf z$ and $\mathbf
x\sim \mathbf z$.  Then $\mathbf z$ is $Q$-related to $\mathbf y$, and to
complete Case 2 it suffices to show that $\mathbf z\sim\mathbf y$.

If $D$ is $B_m$ or $C_m$  with $m<p$, Lemma \ref{f4} shows that $P_1',D$ 
has the same effect as $D,P_1'$.  Hence $Q$ has the same effect as 
$D,P_1',P_2$, and we are done by Lemma \ref{f8} since $P_1',P_2$ has length 
$<r$ (note that if $D=B_1$ in this situation then $|x_1|$ cannot be 0
since it would not be possible to apply $B_1$ to $P_1'(\mathbf z)$).

If $D=C_p$ then $P_1',D$ has the same effect as
\[
R=
(A_{1,\rho_{1,|x_1|}},B_1,A_{2,\rho_{1,|x_2|}},B_2,\ldots,A_{p,\rho_{1,|x_p|}})
\]
since $C_p$ and $B_p$ are inverses of each other.  Then $Q$ has the same 
effect as $R,P_2$, which has length $r-2$.

If $D=B_p$, Lemma \ref{f4} shows that $P_1',D$
has the same effect as $D,P_1',A_{p+1,\gamma}$, where $\gamma$ is the
transposition $(12)$.  Hence $Q$ has the same effect as
$D,P_1',A_{p+1,\gamma},P_2$, and we are done by Lemma \ref{f8} since 
$P_1',A_{p+1,\gamma},P_2$ has length $r-1$ (this is why operations of type 
$B$
count for 2 in the definition of length); note that we cannot have $D=B_1$ and
$|x_1|=0$ in this situation since $p$ would be 1 and then $P_1$ would be
$B_1$, and it would not be possible to apply the sequence $P_1,D$ to 
$\mathbf x$.

\medskip
\noindent
{\bf Case 3.}  Suppose that $D=A_{p+1,\gamma}$ for some $\gamma$.

First suppose $\gamma(1)=1$, which implies $\gamma=\tilde\gamma^{[2]}$ for some
$\tilde\gamma$.  If $p=0$ we're done by Lemma \ref{f7}.  
If $p>0$ 
Lemmas \ref{f4.1}(i) and \ref{f4} show that $P_1,D$ has the same effect as 
$A_{p+1,\gamma'},P_1$, so $P$ has the same effect as $A_{p+1,\gamma'},P_1,P_2$
and we're done by Lemma \ref{f8}.

For the rest of Case 3 we suppose that
\begin{equation}
\label{f12}
\gamma(1)\neq 1.
\end{equation}

Lemmas \ref{f4.1}(i) and \ref{f4}
show that $P$ cannot equal $P_1,D$: if $|x_1|=0$ this is because 
$x_1$ cannot equal $\op y_1$, and
otherwise it's because the permutation
\[
\theta^{[2]}(\beta_1^{[2]})^{-1} \cdots
(\beta_p^{2+\sum_{i< p}|x_i|]})^{-1}
(\rho_{1,\sum_{i\leq p}|x_i|}^{[1]})^{-1}(\gamma^{[1+\sum_{i\leq 
p}|x_i|]})^{-1}
\]
does not take 1 to 1 and so cannot be
equal to $\kappa^{[2]}$.
Let $E$ be the next operation in $P$ and let $P_2'$ be the part of $P$
after $E$; then we have
\begin{equation}
\label{f26}
P=(P_1,D,E,P_2')
\end{equation}
There are three cases to consider.

\medskip
\noindent
{\bf Case 3.1.}  $E=A_{m,\delta}$ for some $\delta$.

If $m=p+1$ we can combine $D$ and $E$ into a single operation, which gives
a path of length $r-1$ from $\mathbf x$ to $\mathbf y$.

Otherwise we can commute $E$ past $D$, which reduces to Case 1.

\medskip
\noindent
{\bf Case 3.2.} $E=B_m$ for some $m$.

If $m\neq p,p+1$ we can commute $E$ past $D$, which reduces to Case 2.

If $m=p$ then $D,E$ has the same effect as $B_m,A_{p+1,\gamma^{[2]}}$, which
reduces to Case 2.

So suppose $m=p+1$.  
Then 
\begin{equation}
\label{f13}
\gamma(1)\neq |x_{p+1}|+1,
\end{equation}
because of the way $p$ was chosen.
Denote $P_1(\mathbf x)$ by  $\mathbf z$, and let
\begin{equation}
\label{f28}
\mathbf z=\lambda\wedge z_1\wedge\cdots\wedge z_j.
\end{equation}
Then
\begin{equation}
\label{f17}
z_{p+1}=\op x_{p+1}
\end{equation}
by Lemmas \ref{f4.1} and \ref{f4}.

Because the operation $B_{p+1}$ can be applied to
$A_{p+1,\gamma}(\mathbf z)$, we have that
\begin{equation}
\label{f14}
\gamma z_{p+1}=\pp w
\end{equation}
for some $w$, and then
\begin{equation}
\label{f16}
ED(\mathbf z)
=
\lambda(\gamma^{[|z_1|+\cdots+|z_p|+1]})^{-1}\wedge \cdots w\wedge \op 
z_{p+2}\cdots
\end{equation}

Equations \eqref{f13}, \eqref{f17} and \eqref{f14} allow us to 
apply Remark \ref{f11}(iii) to get a $\delta\in\Sigma_{|x_{p+1}|}$ with the
properties that
$\delta^{-1}(|x_{p+1}|)=\gamma^{-1}(|x_{p+1}|+1)-1$ and 
\begin{equation}
\label{f15}
\delta x_{p+1}=\pp v
\end{equation}
for some $v$.
Let 
\begin{equation}
\label{f18}
\varepsilon=\gamma(\delta^{[2]})^{-1}.
\end{equation}
Then $\varepsilon$ takes
$|x_{p+1}|+1$ to itself, so there is an 
$\tilde\varepsilon\in\Sigma_{|x_{p+1}|}$ with
$\varepsilon=\tilde\varepsilon^{[1]}$.

We claim that the operation $B_{p+1}$ can be applied to
$A_{p+1,\delta^{[2]}}(\mathbf z)$ and that 
$A_{p+1,\delta^{[2]}},B_{p+1},A_{p+1,\tilde\varepsilon}$ has the same effect on 
$\mathbf z$ as $D,E$.  
Let us assume this for the moment.  If $p=0$ then (by Equation
\eqref{f26}) $P$ has the same effect as 
$A_{1,\delta^{[2]}},B_1,A_{1,\tilde\varepsilon},P_2'$; since the length of
$A_{1,\tilde\varepsilon},P_2'$ is $r-2$ we're done by Lemma \ref{f8} (note that $|x_1|\neq 0$ because of Equation
\ref{f12}).
If $p\neq 0$ then $P$ has the same effect as 
$P_1,A_{p+1,\delta^{[2]}},B_{p+1},A_{p+1,\tilde\varepsilon},P_2'$, 
which 
(by Lemmas \ref{f4.1} and \ref{f4})
has the same effect as 
$A_{p+1,\delta},P_1,B_{p+1},A_{p+1,\tilde\varepsilon},P_2'$. Case 2 applies 
to part of this path after the first operation,
so we're done by Lemma \ref{f8}.

It remains to verify the claim. By Equations \eqref{f17} and \eqref{f15} and
Remark \ref{f11}(ii) we have 
\[
\delta^{[2]}z_{p+1}=\op\delta x_{p+1}= \op \pp v=\pp \op v,
\]
so (using Equation \eqref{f28}) $B_{p+1}$ can be applied to
$A_{p+1,\delta^{[2]}}(\mathbf z)$ and we have
\begin{multline*}
A_{p+1,\tilde\varepsilon}B_{p+1}A_{p+1,\delta^{[2]}}(\mathbf z)
\\ =
\lambda((\delta^{[2]})^{-1}(\tilde\varepsilon^{[1]})^{-1})^{[|z_1|+\cdots+|z_p|+1]}
\wedge \cdots \tilde\varepsilon \op v\wedge \op z_{p+2}\cdots.
\end{multline*}
By Equation \eqref{f16} it suffices to show 
$\tilde\varepsilon^{[1]}\delta^{[2]}=\gamma$ (which follows
from Equation \ref{f18}) and $w=\tilde\varepsilon\op v$.  Because $\bX_{p+1}$ 
is monomorphic, for the latter equation it suffices to show 
$\pp w=\pp \tilde\varepsilon^{[1]}\op v$, and this in turn follows from 
\begin{multline*}
\pp w= \gamma z_{p+1}= \gamma \op x_{p+1}= \varepsilon \delta^{[2]} \op x_{p+1}
\\ =\varepsilon \op \delta x_{p+1}=
\varepsilon \op \pp v =\varepsilon \pp\op v=
\pp \tilde\varepsilon^{[1]}\op v,
\end{multline*}
where we have used (in this order) Equations \eqref{f14}, \eqref{f17}, 
\eqref{f18}, Remark \ref{f11}(i), Equation \ref{f15}, and Remark \ref{f11}(ii).

\medskip
\noindent
{\bf Case 3.3.} $E=C_m$ for some $m$.

If $m\neq p,p+1$ we can commute $E$ past $D$, which reduces to Case 2.

If $m=p+1$ then $D,E$ has the same effect as $C_m,A_{p+1,\gamma^{[1]}}$, which
reduces to Case 2.

So suppose $m=p$ (which implies $p\neq 0$).
Let $P_0$ be 
$A_{1,\alpha_1},B_1,A_{2,\alpha_2},B_2,\ldots,A_{p-1,\alpha_p},B_{p-1}$
if $p>1$ and the empty path if $p=1$;
note that
\begin{equation}
\label{f24}
P=(P_0,A_{p,\alpha_p},B_p,D,E,P_2').
\end{equation}
We
denote $A_{p,\alpha_p}P_0(\mathbf x)$ by  $\mathbf z$ and let
\begin{equation}
\label{f27}
\mathbf z=\lambda\wedge z_1\wedge\cdots\wedge z_j.
\end{equation}

Because the operation $B_p$ can be applied to $\mathbf z$, we have
\begin{equation}
\label{f19}
z_p=\pp w
\end{equation}
for some $w$.
Because $C_p$ can be applied to $A_{p+1,\gamma}B_p{\mathbf z}$, we have
\begin{equation}
\label{f25}
\gamma\op z_{p+1}=\op v
\end{equation}
for some $v$, and 
\begin{equation}
\label{f20}
EDB_p(\mathbf z)
=
\lambda(\gamma^{-1})^{[|z_1|+\cdots+|z_p|+1]})\wedge \cdots \pp w \wedge v
\cdots
\end{equation}

Since $\bX_{p+1}$ is strongly monomorphic, Equations \eqref{f12}
and \eqref{f25} imply that there is 
a $\delta\in\Sigma_{|x_{p+1}|}$ with the
properties that
$\delta^{-1}(1)=\gamma^{-1}(1)-1$ and
\begin{equation}
\label{f22}
\delta z_{p+1}=\op u
\end{equation}
for some $u$.

Let $\varepsilon$ be the transposition $(12)$.
Let
\begin{equation}
\label{f23}
\eta=\gamma(\delta^{[2]})^{-1}(\varepsilon^{[1]})^{-1}.
\end{equation}
Then $\eta$ takes
$1$ to 1, so there is an $\tilde\eta\in\Sigma_{|x_{p+1}|}$
with
$\eta=\tilde\eta^{[2]}$.
We claim that the sequence of operations
\[
A_{p+1,\delta},C_p,A_{p,\varepsilon^{[|z_p|+1]}},B_p,A_{p+1,\tilde\eta}
\]
can be applied to $\mathbf z$ and that it has the same effect as
$B_p,D,E$.
Assuming this for the moment, we see (using Equation \eqref{f24}) that $P$ 
has the
same effect as 
\[
P_0,A_{p,\alpha_p},
A_{p+1,\delta},C_p,A_{p,\varepsilon^{[|z_p|+1]}},B_p,A_{p+1,\tilde\eta},P_2'
\]
which (by commuting $A_{p+1,\delta}$ past the earlier operations) has the 
same effect as
\[
A_{p+1,\delta},P_0,A_{p,\alpha_p},C_p,A_{p,\varepsilon^{[|z_p|+1]}},B_p,A_{p+1,\tilde\eta},P_2'
\]
which (moving $C_p$ past $A_{p,\alpha_p}$) has the same effect as
\[
A_{p+1,\delta},P_0,C_p,A_{p,\alpha_p^{[1]}},A_{p,\varepsilon^{[|z_p|+1]}},B_p,A_{p+1,\tilde\eta},P_2'
\]
which (since $C_p$ commutes with the operations in $P_0$) has the same effect
as
\[
A_{p+1,\delta},C_p,P_0,A_{p,\alpha_p^{[1]}},A_{p,\varepsilon^{[|z_p|+1]}},B_p,A_{p+1,\tilde\eta},P_2'
\]
The part of this path after the first two operations 
has length $r-1$, so we're done by Lemma \ref{f8}.

It remains to verify the claim. 
Equations \eqref{f27}, \eqref{f19} and \eqref{f22}  give
\begin{equation}
\label{f29}
A_{p+1,\delta}(\mathbf z)
= \lambda((\delta^{[2]})^{-1})^{[|z_1|+\cdots+|z_p|+1]}
\wedge
\cdots \psi w \wedge \op u \cdots
\end{equation}
so $C_p$ can be applied to $A_{p+1,\delta}(\mathbf z)$.  Now
$\varepsilon^{[|z_p|+1]}(\pp\pp w)=\pp\pp w$, so Equation \eqref{f29} gives
\begin{multline}
\label{f30}
A_{p,\varepsilon^{[|z_p|+1]}}C_pA_{p+1,\delta}(\mathbf z)
\\ =
\lambda((\delta^{[2]})^{-1}(\varepsilon^{[1]})^{-1})^{[|z_1|+\cdots+|z_p|+1]}
\wedge \cdots \pp\pp w\wedge u\cdots
\end{multline}
Then $B_p$ can be applied to this and we have
\begin{multline*}
\label{f31}
A_{p+1,\tilde\eta}B_pA_{p,\varepsilon^{[|z_p|+1]}}C_pA_{p+1,\delta}(\mathbf z)
\\ =
\lambda((\delta^{[2]})^{-1}(\varepsilon^{[1]})^{-1}
(\tilde\eta^{[2]})^{-1})^{[|z_1|+\cdots+|z_p|+1]}
\wedge \cdots \pp w\wedge \tilde\eta\op u\cdots
\end{multline*}
Comparing this to Equation \eqref{f20}, we see that it suffices to show
$\gamma=\tilde\eta^{[2]}\varepsilon^{[1]}\delta^{[2]}$ (which follows from
Equation \eqref{f23}) and $\tilde\eta\op u=v$.  Since $\bX_{p+1}$ is 
monomorphic,
the latter equation follows from the equations
\[
\op\tilde\eta\op u=\eta\op\op u=\gamma\op \delta^{-1}\op u=\gamma\op z_{p+1}=\op v
\]
where we have used Remark \ref{f11}(i) and Equations \ref{f23}, \ref{f22} and
\ref{f25}.

This completes the proof of Proposition \ref{f1}
\end{proof}

\begin{proof}[Proof of Proposition \ref{f32}]
First we need some notation to distinguish the
action of $\Sigma_j$ on $\bX^{\wedge j}$ from the action of $\Sigma_k$ on the
$k$-th object of $\bX^{\wedge j}$: given $\nu\in\Sigma_j$, $\eta\in\Sigma_k$  
and a point 
$\mathfrak x$ of $\bX^{\wedge j}$ with $|\mathfrak x|=k$, we write $\mathfrak 
x^\nu$ (resp., $\eta\mathfrak x$) for the point obtained by applying $\nu$
(resp., $\eta$) to $\mathfrak x$.

Now suppose there is a point $\mathfrak x$ in
$\bX^{\wedge j}$ which
is not a basepoint and a nontrivial
$\nu\in\Sigma_j$ such that $\mathfrak x^\nu=\mathfrak x$.
Let $\mathfrak x$ be represented by
\[
\mathbf x=\theta\wedge x_1\wedge\cdots\wedge x_j
\]
and note that $|x_i|>0$ for all $i$ by our hypothesis on $X_0$.
Then $\mathfrak x^\nu$ is represented by
\[
\mathbf y=\theta\eta^{-1}\wedge x_{\nu^{-1}(1)}\wedge \cdots\wedge 
x_{\nu^{-1}(j)},
\]
where $\eta$ permutes blocks of size $|x_1|,\ldots,|x_j|$ in the same way that
$\nu$ permutes $1,\ldots,j$.  There must be a path from $\mathbf x$ to $\mathbf
y$, and it cannot consist entirely of operations of type $A$, since $\eta$ is
not in $\Sigma_{|x_1|}\times\cdots\times\Sigma_{|x_j|}$.  Thus we may assume without loss of generality that some $x_i$ is in
the image of $\op$, and this implies (using operations of type A and C) that 
there is a point $\mathfrak w$ of
$\bX^{\wedge j}$ and a $\zeta\in\Sigma_{|\mathfrak x|}$ with $\mathfrak
x=\zeta\op\mathfrak w$. 
Now $\nu$ commutes with both $\zeta$ and $\op$, so the fact that $\mathfrak
x^\nu=\mathfrak x$ implies that
\[
\zeta\op(\mathfrak w^\nu)=\zeta\op\mathfrak w.
\]
Using Proposition \ref{f1} we see that $\mathfrak w^\nu=\mathfrak w$, and thus
$\mathfrak w$ is a nontrivial fixed point of the $\Sigma_j$ action with 
$|\mathfrak w|<|\mathfrak x|$.  Continuing in this way would give a nontrivial
fixed point in the 0-th object of $\bX^{\wedge j}$, which is impossible by our
hypothesis on $X_0$.
\end{proof}

We conclude with an example which shows that Propositions \ref{f32} and \ref{f1}
fail if ``strongly monomorphic'' is replaced by ``monomorphic.''  

Recall Definition \ref{z1}.  Let $\bX$ be the subobject of $\bS$ with $X_0=*$
and $X_k=S^k$ for $k>0$.  

We will denote the point of $\bX\wedge \bX$ represented by a
symbol $\theta\wedge x\wedge y$ by $[\theta\wedge x\wedge y]$.
Let $x$ be the nontrivial simplex in $X_1$.  Note that
$\bX$ is not strongly monomorphic because the transposition $(12)$ takes $\op
x$ to itself, so $[(12)\op x]$ is in the image of $\op$, but $x$ is not in the
image of $\op$.

Let $\tau$ be the operation which switches the two factors of $\bX\wedge \bX$. Now we claim that 
\[
([\op x\wedge x])^\tau=[\op x\wedge x],
\]
which gives a counterexample for Proposition \ref{f32}.  We have
\begin{align*}
([\op x\wedge x])^\tau &=[\rho_{1,2} \wedge x\wedge \op x]
=[(23)(12) \wedge x\wedge \op x]
= [(23)(12) \wedge \pp x\wedge x]
\\ &= [(23) \wedge \pp x\wedge x] \text{\ because $(12)$ acts trivially on $(X_2)_{1,1}$}
\\
&= [(23) \wedge x\wedge \op x]
= [x\wedge \op x]
=[ \pp x\wedge x]
= [\op x\wedge x].
\end{align*}

Next we observe that 
$
[(12)\wedge x\wedge x]
\neq
[x\wedge x]
$, because there is
no nontrivial path beginning at $(12)\wedge x\wedge x$.  But
we claim that 
\[
\op[((12)\wedge x\wedge x)]
= 
\op([x\wedge x]),
\]
which gives a counterexample for Proposition \ref{f1}.  We have
\begin{align*}
\op([(12)\wedge x\wedge x])
&=[(23)\wedge \op x\wedge x]
=[(23)\wedge \pp x\wedge x]
\\
&=[(23)\wedge x\wedge \op x]
=[x\wedge \op x]
=[\pp x\wedge x]
\\
&=[\op x\wedge x]
=\op([x\wedge x]).
\end{align*}

\section{The proof of Lemma \ref{m99}}\label{Lemma m99}

Let $\bW$ denote the multisemisimplicial spectrum whose $k$-th object is 
$S_\bu^{k-\mathrm{multi,}{\mathord{\pitchfork}}}(T(\STop(k)))$, so that
$|\bW|=\bY$.
We begin by showing that the monad $\bbP$ of Definition \ref{m100} acts on the
pair $(\bW,\bR_\STop)$.

Let us define a $\Z$-graded category $\B$ as follows.  The 
objects of $\B$ are pairs 
\[
(g:\Delta^\bfn\to T(\STop(k)),o),
\]
where both $\bfn$ and $k$ are allowed to vary and $o$ is an orientation of
$\Delta^\bfn$; the grading is given by $d(g,o)=\dim(\Delta^\bfn)-k$. We assume that the preimage $g^{-1}S$ of the zero section $S$ is a topological manifold and write
$ \Inv(g,o)$ for the resulting oriented manifold in $\A_{\STop}$.
The  morphisms are commutative diagrams
\[
\xymatrix{
\Delta^\bfn
\ar[d]^\phi
\ar[r]^-g
&
T(\STop(k))
\ar[d]^\alpha
\\
\Delta^{\bfn'}
\ar[r]^-{g'}
&
T(\STop(k))
}
\]
in which $\phi$ is a composite of coface maps and permutations of the 
factors and $\alpha$ is a permutation; we require $\phi$ to be orientation 
preserving if the dimensions are equal.
$\B$ is a symmetric monoidal $\Z$-graded category with product $\bx$, where
\[
(g,o)\bx (g',o')
\]
is the pair consisting of the composite
\[
\Delta^\bfn\times\Delta^{\bfn'}
\xrightarrow{g\times g'}
T(\STop(k))
\times
T(\STop(k'))
\to
T(\STop(k+k'))
\]
and the orientation $o\times o'$.
The symmetry isomorphism $\gamma$ is
\[
\xymatrix{
\Delta^\bfn\times\Delta^{\bfn'}
\ar[d]^\phi
\ar[r]^-{g\times g'}
&
T(\STop(k))
\times
T(\STop(k'))
\ar[r]
&
T(\STop(k+k'))
\ar[d]^\alpha
\\
\Delta^{\bfn'}\times\Delta^\bfn
\ar[r]^-{g'\times g}
&
T(\STop(k'))
\times
T(\STop(k))
\ar[r]
&
T(\STop(k'+k))
}
\]
where $\phi$ and $\alpha$ are the evident permutations.

In the construction of Definition \ref{m67}, if we replace $\Ag$ by $\B$, $\Ar$
by $\A_\STop$, $\boxtimes$ by $\bx$ and $\Sigr$ by $\Inv$ we 
obtain a functor
\[
\bfd_\bs:\A_1\times\cdots\times\A_j\to \A_\STop,
\]
for each datum $\bfd$,
where $\A_i$ denotes $\B$ if $r(i)=u$ and $\A_\STop$ if $r(i)=v$.

Next we have the analogue of Definition \ref{m101}.

\begin{definition}
Let $k_1,\ldots,k_j$ be non-negative integers and let $\bfn_i$ be a
$k_i$-fold multi-index for $1\leq i\leq j$.  Let $r:\{1,\ldots,j\}\to \{u,v\}$,
and for $1\leq i\leq j$ let $\bZ_i$ denote $\bW$ if $r(i)=u$ and
$\bR_\STop$ if $r(i)=v$.  
Then for each map of
preorders
\[
a: U(\Delta^{\bfn_1}\times\cdots\times \Delta^{\bfn_j})\to P_{r;v}
\]
we define
\[
a_*: 
((\bZ_1)_{k_1})_{\bfn_1}
\times
\cdots\times
((\bZ_j)_{k_j})_{\bfn_j}
\to
((\bR_\STop)_{k_1+\cdots+k_j})_{(\bfn_1,\ldots,\bfn_j)}
\]
by 
\begin{multline*}
a_*(z_1,\ldots,z_j)(\sigma_1\times\cdots\times\sigma_j,o_1\times\cdots\times
o_j)
\\
=
i^{\epsilon(\zeta)}
a(\sigma_1\times\cdots\times\sigma_j)_\bs
(z_1(\sigma_1,o_1), \ldots, z_j(\sigma_j,o_j)),
\end{multline*}
where 
\begin{itemize}
\item
if $r(i)=u$ then
$z_i(\sigma_i,o_i))$ denotes $(z_i|_{\sigma_i},o_i)$, and
\item
$\zeta$ is the block permutation that takes blocks
${\mathbf b}_1$, \dots, ${\mathbf b}_j$, ${\mathbf c}_1$, \dots, ${\mathbf
c}_j$
of size
$k_1$, \dots, $k_j$, $\dim \sigma_1$, \dots, $\dim \sigma_j$ into the order
${\mathbf b}_1$, ${\mathbf c}_1$, \dots, ${\mathbf b}_j$, ${\mathbf c}_j$.
\end{itemize}
\end{definition}

As in Section \ref{msss}, this definition leads to a map
\begin{equation}
\label{m102}
\bbP_2(\bW,\bR_\STop)
\to
\bR_\STop.
\end{equation}
Since $\bW$ is a commutative multisemisimplicial symmetric ring spectrum, we
have a map
\begin{equation}
\label{m103}
\bbP_1(\bW)
\xrightarrow{\Xi_1}
\bigvee_{j\geq 0} \,  \bW^{\wedge j}/\Sigma_j
\to
\bW,
\end{equation}
where $\Xi_1$ is given in Definition \ref{m86}(ii).  

The maps \eqref{m102} and
\eqref{m103} give the required action of $\bbP$ on $(\bW,\bR_\STop)$.  Now the
proof of Theorem \ref{m54} (given in Section \ref{rect2}) gives a map of
commutative symmetric ring spectra
\[
|B_\bu(\bbP_1',\bbP,(\bW,\bR_\STop))|
\to
|B_\bu(\bbP_2',\bbP,(\bW,\bR_\STop))|
\]
which is a weak equivalence by Lemma \ref{m104}.  As in Remark \ref{m80}, there
is a weak equivalence of commutative symmetric ring spectra
\[
(\bM_\STop)^{\mathrm{comm}}
\to
|B_\bu(\bbP_2',\bbP,(\bW,\bR_\STop))|.
\]
To complete the proof, we observe that there is a weak equivalence of
commutative symmetric ring spectra
\[
|B_\bu(\bbP_1',\bbP,(\bW,\bR_\STop))|
=
|B_\bu(\bbP_1',\bbP_1,\bW)|
\to
|B_\bu(\bbP_1',\bbP_1',\bW)|
\to
|\bW|=\bY,
\]
where the first arrow is a weak equivalence by Proposition \ref{m79}(iii) and 
the second by \cite[Proposition 9.8 and Corollary 11.9]{MR0420610}.

\section{A functorial version of $\ad^R_{\mathrm{rel}}$}
\label{ff14}

In 
\cite[Section 13]{LM12} we explained why $\ad^R$ (as defined in
\cite[Section 9]{LM12}, which is the definition we have used in the present
paper) is not a functor of
$R$, and how to to modify the
definition of $\ad^R$ to make it a functor (unfortunately, in \cite{LM12} we
also denoted the modified version by $\ad^R$; in this appendix we will be more
careful with the notation).  Our goal in this appendix is to give a 
functorial version of $\ad^R_{\mathrm{rel}}$, which is needed in Theorem \ref{Nov13.3}.  

Unfortunately, it seems that we cannot just adapt the method of \cite[Section 
13]{LM12} to this situation, because $\ad^R_{\mathrm{rel}}$ isn't even
``approximately'' functorial: given a ring homomorphism $R\to S$ and 
an object $(C,D,\beta,\varphi)$ of $\A_{\text{rel}}^R$, there does not seem 
to be a reasonable way to create an object $(C',D',\beta',\varphi')$ of 
$\A_{\text{rel}}^S$ from this data
(we could let $C'=S\otimes_R C$, but the obvious candidate for $D'$ does not come with a quasi-isomorphism).
So in Subsection \ref{ff2} (after some preliminary motivation in Subsection
\ref{ff1}) we give a variant of $\ad^R_{\mathrm{rel}}$, which we denote by
$\ad_{\text{Rel}}^R$, and in 
Subsection \ref{ff2.5} we show that $\ad_{\text{Rel}}^R$ is
approximately functorial (i.e., functorial up to isomorphism).
In Subsection \ref{ff3} we show that the ad
theories $\ad^R$ and $\ad_{\text{Rel}}^R$ are equivalent, that is, there is a
morphism of ad theories from
$\ad_{\text{Rel}}^R$ to
$\ad^R_{\mathrm{rel}}$ which induces an isomorphism of bordism groups.  
In 
Subsection \ref{ff4} we give an enhanced version
of the material in \cite[Section 13]{LM12}, and in Subsection \ref{ff5}
we use this 
to create a variant of 
$\ad_{\text{Rel}}^R$, which we
denote by $\ad^R_\Rel$.  In Subsection \ref{ff51} we 
show (using
Subsection \ref{ff2.5})
that $\ad^R_\Rel$ is
a functor of $R$.

\begin{remark}
We could have used $\ad^R_{\mathrm{Rel}}$ throughout this paper instead of 
$\ad^R_{\mathrm{rel}}$, but that would have added extra complexity and  functoriality is only an issue at the end of the paper in Theorem \ref{Nov13.3}.
\end{remark}

\subsection{Background}
\label{ff1}

As motivation for the definition of $\ad_{\text{Rel}}^R$, we need

\begin{lemma}
\label{ff15}
Let $M$ be a right $R$ module and $N$ a left $R$ module, and let
$R^{\mathrm{op}}\otimes R$ act on $R$ on the right in the usual way.  Then the
map
\[
a:R\otimes_{R^{\mathrm{op}}\otimes R} (M\otimes N)
\to
M\otimes_R N
\]
given by $a(r\otimes m\otimes n)=mr\otimes n$ is an isomorphism.
\end{lemma}

\begin{proof}
This follows immediately from the isomorphisms
\[
R\otimes_{R^{\mathrm{op}}\otimes R} (M\otimes N)
\cong
M\otimes_R R\otimes_R N\cong M\otimes_R N.
\]
\end{proof}

Now fix a ring $R$ with involution.

\begin{definition}
Let $P$ and $Q$ be left $R^{\mathrm{op}}\otimes R$ modules.  A map 
\[
b:P\to Q
\]
is {\it quasi-linear} if $b((r\otimes s)p)=(\bar{s}\otimes \bar{r})b(p)$.
\end{definition}

\begin{lemma}
\label{aa1}
Let $b:P\to Q$ be a quasi-linear map of left $R^{\mathrm{op}}\otimes R$
modules.  Then the map
\[
\hat{b}:R\otimes_{R^{\mathrm{op}}\otimes R} P
\to
R\otimes_{R^{\mathrm{op}}\otimes R} Q
\]
given by $\hat{b}(r\otimes p)=\bar{r}\otimes b(p)$
is well-defined. \qed
\end{lemma}

\subsection{The ad theory $\ad_{\text{Rel}}^R$}
\label{ff2}

Recall the definition of homotopy finite
(\cite[Definition 9.2(iv)]{LM12}).

\begin{definition}  
\label{111} A {\it Relaxed quasi-symmetric complex of dimension $n$}
is a quadruple $(C,E,\gamma,\phi)$, where $C$ is a homotopy finite%
\footnote{In \cite[Section 9]{LM12} we also required $C$ to be free, but that
turns out not to be necessary; see Remark \ref{aa100}.}
chain complex over $R$, $E$ is a
homotopy finite chain complex over $R^{\mathrm{op}}\otimes R$ with a 
$\Z/2$ action for which the generator acts quasi-linearly, $\gamma$ is a 
$\Z/2$ equivariant
$R^{\mathrm{op}}\otimes R$-linear 
quasi-isomorphism $C^t\otimes C\to E$, and $\phi$ is an $n$-dimensional 
element of 
$(R\otimes_{R^{\mathrm{op}}\otimes R} E)^{\Z/2}$ (where the $\Z/2$ action is
given by Lemma \ref{aa1}).
\end{definition}

For the following example, 
note that if $A$ is a left $R^{\mathrm{op}}\otimes R$ module which is
nonzero in only finitely many dimensions then $A^W$ is (additively)
a direct sum of copies of $A$,
and hence the natural map
\[
R\otimes_{R^{\mathrm{op}}\otimes R} (A^W)
\to
(R\otimes_{R^{\mathrm{op}}\otimes R} A)^W
\]
is an isomorphism because tensor product preserves direct sums.

\begin{example}
\label{aa2}
If $(C,\varphi)$ is a quasi-symmetric complex as defined in \cite[Definition
9.3]{LM12}, and if $C$ is nonzero in only finitely many dimensions, 
then the quadruple $(C,(C^t\otimes C)^W,\gamma,\phi)$ is a 
Relaxed quasi-symmetric complex, where $\gamma:C^t\otimes C\to 
(C^t\otimes C)^W$ 
is induced by the augmentation $W\to \Z$ and $\phi$ is the image of $\varphi$
under the composite
\[
((C^t\otimes_R C)^W)^{\Z/2}
\cong
((R\otimes_{R^{\mathrm{op}}\otimes R} (C^t\otimes C))^W)^{\Z/2}
\cong
(R\otimes_{R^{\mathrm{op}}\otimes R} ((C^t\otimes C)^W))^{\Z/2}
\]
(where the first isomorphism is Lemma \ref{ff15}).
\end{example}

\begin{definition}
We define a category $\A_{\text{Rel}}^R$ as 
follows.  The objects of $\A_{\text{Rel}}^R$ are the Relaxed quasi-symmetric
complexes.  A morphism $(C,E,\gamma,\phi)\to (C',E',\gamma',\phi')$ is a
pair $(f:C\to C',g:E\to E')$, where $f$ is an $R$-linear chain map and $g$ is a
$\Z/2$ equivariant $R^{\mathrm{op}}\otimes R$-linear chain map, such that
$g\gamma=\gamma'(f\otimes f)$, 
and (if $\dim \phi=\dim\phi'$)
$(1\otimes g)_*(\phi)=\phi'$.
\end{definition}

$\A_{\text{Rel}}^R$ is a balanced (\cite[Definition 5.1]{LM12}) $\Z$-graded 
category, where $i$ takes $(C,E,\gamma, \phi)$ 
to $(C,E,\gamma, -\phi)$
and $\emptyset_n$ is the $n$-dimensional object 
for which $C$ and $E$ are zero in all degrees.

\begin{remark}
\label{aa4}
There is a morphism $\A_{\text{Rel}}^R\to \A_{\text{rel}}^R$ of $\Z$-graded
categories which takes
$(C,E,\gamma,\phi)$ to $(C, R\otimes_{R^{\mathrm{op}}\otimes R} E,
\beta,\phi)$, where $\beta$ is the composite
\[
C^t\otimes_R C
\cong
R\otimes_{R^{\mathrm{op}}\otimes R} (C^t\otimes C)
\xrightarrow{1\otimes \gamma}
R\otimes_{R^{\mathrm{op}}\otimes R} E.
\]
\end{remark}

\begin{remark}
\label{aa3}
Let $\A^R_{\mathrm{fin}}$ be the full subcategory of $\A^R$ consisting of
objects $(C,\varphi)$ with $C$ finite (not just homotopy finite).
Let $\A^R_{\mathrm{Rel,fin}}$ be the full subcategory of $\A^R_\mathrm{Rel}$ 
consisting of
objects $(C,E,\gamma,\phi)$ with $C$ and $E$ finite. 
The construction of Example \ref{aa2} gives a morphism 
\[
\A^R_{\mathrm{fin}}\to \A_{\text{Rel,fin}}^R
\]
of $\Z$-graded categories. 
\end{remark}

Next we must say what the $K$-ads with values in $\A_{\text{Rel}}^R$ are.  
For a balanced pre $K$-ad $F$ 
we will use the notation
\[
F(\sigma,o)=(C_\sigma,E_\sigma,\gamma_\sigma,\phi_{\sigma,o}).
\]

Recall \cite[Definition 9.7]{LM12}.

\begin{definition}
A balanced pre $K$-ad $F$ is 
{\it well-behaved} if $C$ and $E$ are well-behaved.
\end{definition}

\begin{definition}
(i) 
A {\it balanced $K$-ad} is a pre $K$-ad with the following properties.

(a) it is balanced, closed, and well-behaved, and 

(b) the composite of $F$ with the morphism of Remark \ref{aa4} satisfies part
(i)(b) of Definition \ref{m47}

(ii) A {\it $K$-ad} is a pre $K$-ad which is naturally isomorphic to a balanced
$K$-ad.
\end{definition}

We write $\ad_{\text{Rel}}^R(K)$ for the set of $K$-ads with values in
$\A_{\text{Rel}}^R$.

\begin{thm}
\label{ff9}
$\ad_{\mathrm{Rel}}^R$ is an ad theory.
\end{thm}

This follows from the proof of Theorem \ref{m42} with minor changes.

\begin{remark}
\label{aa5}
The morphisms of Remarks \ref{aa4} and \ref{aa3} take ads to ads.
\end{remark}

\subsection{$\ad_{\mathrm{Rel}}^R$ is approximately functorial}
\label{ff2.5}

The result in this subsection will be used in Subsection \ref{ff51}.

\begin{definition}
\label{ff12}
Let $h:R\to S$ be a ring homomorphism.  Define a functor 
\[
h_\mathrm{Rel}:\A^R_{\mathrm{Rel}} \to \A^S_{\mathrm{Rel}}
\]
as follows.  For an object
$(C,E,\gamma,\phi)$ 
of $\A^R_{\mathrm{Rel}}$, let 
\[h_\mathrm{Rel}(C,E,\gamma,\phi)=
(S\otimes_R C,(S^{\mathrm{op}}\otimes
S)\otimes_{R^{\mathrm{op}}\otimes R} E,\gamma',\phi'),
\]
where 
$\gamma'$
is the composite
\[
(S\otimes_R C)^t\otimes (S\otimes_R C)
\cong
(S^{\mathrm{op}}\otimes
S)\otimes_{R^{\mathrm{op}}\otimes R} (C^t\otimes C)
\xrightarrow{1\otimes \gamma}
(S^{\mathrm{op}}\otimes
S)\otimes_{R^{\mathrm{op}}\otimes R} E
\]
(which is a quasi-isomorphism by the K\"unneth spectral
sequence \cite[Theorem 5.6.4]{WEIB}, using the fact that 
$C$ and $E$ are homotopy finite) and 
$\phi'$
is the image of $\phi$ under the
composite
\[
(R\otimes_{R^{\mathrm{op}}\otimes R} E)^{\Z/2}
\to
(S\otimes_{R^{\mathrm{op}}\otimes R} E)^{\Z/2}
\cong
(S\otimes_{S^{\mathrm{op}}\otimes S} (S^{\mathrm{op}}\otimes
S)\otimes_{R^{\mathrm{op}}\otimes R} E)^{\Z/2}.
\]
\end{definition}

The reader can
check
that if $k:S\to T$ is another ring homomorphism then 
$(kh)_\mathrm{Rel}(C,E,\gamma,\phi)$
is isomorphic to but not equal to 
$k_\mathrm{Rel}h_\mathrm{Rel}(C,E,\gamma,\phi)$.

\begin{prop}
\label{ff10}
$h_\mathrm{Rel}$ takes ads to ads.
\end{prop}

\begin{proof}
Let $F\in \ad^R_{\mathrm{Rel}}(K)$.  We may assume that $F$ is balanced. Write
\[
F(\sigma,o)=({C}_\sigma,{E}_\sigma,\gamma_\sigma,\phi_{\sigma,o}).
\]
Let $\Psi$ be the functor of Remark \ref{aa4}.  Let $G=\Psi\circ F$; then $G$
is an element of $\ad^R_{\mathrm{rel}}(K)$.
Write 
\[
G(\sigma,o)=({C}_\sigma,D_\sigma,\beta_\sigma,\phi_{\sigma,o}).
\]
Let $H=\Psi\circ h_\mathrm{Rel}\circ F$; we need to show that $H$ is an ad.
It's immediate that $H$ is balanced, well-behaved and closed, so it only
remains to show that it satisfies part (ii) of Definition \ref{m47}.
Write
\[
H(\sigma,o)=(S\otimes_R{C}_\sigma,D^S_\sigma,\beta^S_\sigma,
\psi_{\sigma,o})
\]
and fix an 
oriented cell $(\sigma,o)$ of $K$.
Recall Notation \ref{ff11} and let
\begin{multline*}
k_\sigma: ((S\otimes_R C)^t\otimes_S (S\otimes_R C))_\sigma/
((S\otimes_R C)^t\otimes_S (S\otimes_R C))_{\partial\sigma}
\\
\to
((S\otimes_R C_\sigma)/(S\otimes_R C_{\partial\sigma}))
\otimes_S (S\otimes_R C_\sigma)
\end{multline*}
be the analogous map.
We need to show that the slant product with 
$(k_\sigma)_*(\beta^S_\sigma)_*^{-1}([\psi_{\sigma,o}])$
is an isomorphism
\[
H^*(\Hom_S(S\otimes_R {C}_\sigma,S))
\to
H_{\dim \sigma-\deg F-*}(S\otimes_R {C}_\sigma/S\otimes_R
{C}_{\partial\sigma}).
\]
First we observe that
(using the definition of $h_\mathrm{Rel}$) the 
image of 
$(j_\sigma)_*(\beta_\sigma)_*^{-1}([\phi_{\sigma,o}])$
in $H_*((S\otimes_R {C}_\sigma)/(S\otimes_R {C})_{\partial\sigma})\otimes S\otimes_R
{C}_\sigma$
is
$(k_\sigma)_*(\beta^S_\sigma)_*^{-1}([\psi_{\sigma,o}])$.  Now the desired 
isomorphism 
follows from our next lemma.

\end{proof}

\begin{lemma}
\label{aa10}
Let $h:R\to S$ be a homomorphism of rings with involution.  
Let $A$ and $B$ be homotopy finite chain complexes over $R$.
Let $x$ be a cycle in $A^t\otimes_R B$ with 
the property that the 
slant product with $x$ is an isomorphism
\[
H^*(\Hom_R(B,R))\to H_{\dim x-*} A^t.
\]
Let $y$ be the image of $x$ under the map
\[
A^t\otimes_R B
\to
(S\otimes_R A)^t
\otimes_S
(S\otimes_R B).
\]
Then the slant product with $y$ is an isomorphism
\[
H^*(\Hom_S(
S\otimes_R B ,S))\to H_{\dim x-*}  
(S\otimes_R{A})^t.
\]
\end{lemma}

\bigskip

\begin{proof}[Proof of Lemma \ref{aa10}]

By naturality of the slant product, we may assume that $A$ and $B$ are finite.
Because $B$ is 
additively a direct sum of
finitely many copies of $R$, the map
\[
\Upsilon:\Hom_R({B},R)\otimes_R S\to \Hom_S(S\otimes_R {B},S)
\]
defined by $\Upsilon(f\otimes s)(t\otimes b)=tf(b)s$
is an isomorphism. 

Consider the diagram
\[
\xymatrix{
H^*(\Hom_S(
S\otimes_R B ,S))
\ar[r]^-{\setminus y}
&
H_{\dim x-*}  
(S\otimes_R A)^t
\\
H^*(\Hom_R({B},R)\otimes_R S)
\ar[r]
\ar[u]_\cong^{\Upsilon}
&
H_{\dim x-*} 
(A^t\otimes_R S)
\ar[u]_\cong
}
\]
where the bottom arrow is induced by the chain map 
$(\setminus x)\otimes 1$.
It's straightforward to check that the diagram commutes, and the lower arrow
is an isomorphism by the K\"unneth spectral sequence 
\cite[Theorem 5.6.4]{WEIB}.
\end{proof}

\subsection{Comparison of $\ad_{\mathrm{Rel}}^R$ and
$\ad_{\mathrm{rel}}^R$}
\label{ff3}

\begin{prop}
\label{aa6}
The morphism $\ad_{\mathrm{Rel}}^R\to \ad_{\mathrm{rel}}^R$ induces an 
isomorphism of bordism groups.
\end{prop}

\begin{proof}
Let $\A_{\text{Rel,fin}}^R$ be the full subcategory of $\A_{\text{Rel}}^R$
consisting of objects $(C,E,\gamma,\phi)$ for which $C$ is finite.
Consider the diagram
\[
\xymatrix{
\A_{\mathrm{Rel}}^R
\ar[r]
&
\A_{\mathrm{rel}}^R
\\
\A_{\text{Rel,fin}}^R
\ar[u]^a
&
\\
\A^R_{\mathrm{fin}}
\ar[u]^b
\ar[r]^c
&
\A^R
\ar[uu]^d
}
\]
where $a$ is induced by the inclusion of categories, $b$ is given by 
Remark \ref{aa3},
$c$ is induced by the inclusion of categories, and
$d$ is given by Remark \ref{m45}.  This diagram commutes up to natural
isomorphism, so it induces a commutative diagram of bordism groups:
\[
\xymatrix{
(\Omega_{\mathrm{Rel}}^R)_*
\ar[r]
&
(\Omega_{\mathrm{rel}}^R)_*
\\
(\Omega_{\text{Rel,fin}}^R)_*
\ar[u]^{\Omega^a_*}
&
\\
(\Omega^R_{\mathrm{fin}})_*
\ar[u]^{\Omega^b_*}
\ar[r]^{\Omega^c_*}
&
(\Omega^R)_*
\ar[uu]^{\Omega^d_*}
}
\]

$\Omega_*^d$ is an isomorphism by Proposition \ref{m48}, and the proof of that
proposition, with minor modifications, shows that $\Omega_*^b$ is an
isomorphism.

To see that $\Omega_*^a$ is onto, let $(C,E,\gamma,\phi)$ represent an element
of $\ad_{\mathrm{Rel}}^R(*)$.  Since $C$ is homotopy finite, there is a chain
homotopy equivalence $f:C'\to C$ with $C'$ finite.  Let $\gamma'=\gamma\circ
(f\otimes f)$.  Then $(C',E,\gamma',\phi)$ is an element of
$\ad^R_{\mathrm{Rel,fin}}(*)$, and $(f,\id)$ is a
morphism from $(C',E,\gamma',\phi)$ to $(C,E,\gamma,\phi)$.  Let $F$ be the
cylinder object of $(C,E,\gamma,\phi)$ (\cite[Definition 3.10(g)]{LM12}).
Let $0,1,\iota$ denote the three cells of the unit interval $I$, with their
standard orientations.
Replacing $F(0)$ with $(C',E,\gamma',\phi)$ gives an $I$-ad which is a bordism
between $(C,E,\gamma,\phi)$ and $(C',E,\gamma',\phi)$.

To complete the proof it suffices to show that $\Omega_*^c$ is an isomorphism,
since this will imply that $\Omega_*^a$ is a monomorphism.

To see that $\Omega_*^c$ is onto, let $(C,\varphi)$ represent an element of
$\ad^R(*)$.  There is a chain homotopy equivalence $f:C\to C'$ with $C'$
finite. Then $(C',(f\otimes f)\circ \varphi)$ represents an element of
$\ad^R_{\mathrm{fin}}(*)$, and $(f,\phi)$ is a morphism from $(C,\varphi)$ to
$(C',(f\otimes f)\circ \varphi)$. Let $F$ be the cylinder object of
$(C',(f\otimes f)\circ \varphi)$.  Replacing $F(0)$ with $(C,\varphi)$ gives a
bordism between $(C,\varphi)$ and $(C',(f\otimes f)\circ \varphi)$.

To see that $\Omega_*^c$ is a monomorphism, let $F\in \ad^R(I)$
with $F(0)$ and $F(1)$ in $\A_{\text{fin}}^R$.  
Write 
\[
F(0)= (C_0,\varphi_0), \quad F(1)= (C_1,\varphi_1),\quad
F(\iota)= (C_\iota,\varphi_\iota).
\]
Since $C_\iota$ is homotopy finite, there is a chain homotopy equivalence
$f:C_\iota\to B$ with $B$ finite.  Let $g_0$ be the composite $C_0\to
C_\iota\xrightarrow{f} B$ and similarly for $g_1:C_1\to B$.  Let 
$\mathrm{cl}(I)$ be the
cellular chain complex of $I$ and let $j_0$ (resp., $j_1$)
be the composite
$\Z\cong \mathrm{cl}(0)\to \mathrm{cl}(I)$ (resp., $\Z\cong \mathrm{cl}(1)\to
\mathrm{cl}(I)$, where the second map is the
inclusion.
Let $B'$ be the colimit 
\[
\xymatrix{
& 
C_0 
\ar[ld]_{j_1\otimes \mathrm{id}}
\ar[rd]^{g_0}
&
& 
C_1 
\ar[ld]_{g_1}
\ar[rd]^{j_1\otimes \mathrm{id}}
&
\\
\mathrm{cl}(I)\otimes C_0
&
&
B
&
&
\mathrm{cl}(I)\otimes C_1
}
\]
The composites
$C_0\xrightarrow{j_0\otimes \mathrm{id}}\mathrm{cl}(I)\otimes C_0\to B'$ and 
$C_1\xrightarrow{j_0\otimes \mathrm{id}}\mathrm{cl}(I)\otimes C_1\to B'$
are strong monomorphisms
(\cite[Definition
9.6]{LM12}).
Let $n$ be the degree of $F$ and let 
$\psi\in B'_{1-n}$ be the image of 
$\iota\otimes\varphi_0+f_*(\phi_\iota)-\iota\otimes \varphi_1$.

Define an $I$-ad $G$ by
\[
G(0)=F(0),\quad G(1)=F(1),\quad G(\iota)=(B',\psi).
\]
Then $G$ is the desired bordism.
\end{proof}

\begin{remark}
\label{aa100}
The proof that $\Omega_*^c$ is an isomorphism also shows that the requirement
in \cite[Section 9]{LM12} that $C$ should be free over $R$ is not needed.  
That is, if we define $\ad^R$ as in \cite[Section 9]{LM12} and $(\ad^R)^\dagger$
by requiring only that $C$ be homotopy finite then the forgetful map
$\ad^R\to (\ad^R)^\dagger$ induces an isomorphism of bordism groups.
\end{remark}

\subsection{Enhanced version of \cite[Section 13]{LM12}}
\label{ff4}

In \cite[Section
13]{LM12} we gave a model for the category of free $R$ modules which is
functorial in $R$.  In this subsection give a similar model for the category 
of all $R$ modules.  Our terminology and notation will be different from
\cite[Section 13]{LM12}.

We define the category of {\it schematic free $R$ modules} as follows.  An 
object is
a set $\mathbb M$.  This should be thought of as representing the free $R$
module generated by $\mathbb M$, which we denote by $R\langle \mathbb
M\rangle$.  We define a map $\mathbb M\to\mathbb M'$ to be a map of $R$-modules 
$ R\langle \mathbb M\rangle \to R\langle \mathbb M'\rangle$.

We define the category of {\it schematic $R$ modules} as follows.
An object of this category is
a triple $(\mathbb M,\mathbb N,T)$, where $\mathbb M$ and
$\mathbb N$ are schematic free $R$ modules and $T$ is a map
$\mathbb N\to \mathbb M$.
Such a triple should be thought of as representing the quotient of
$R\langle \mathbb M \rangle$ by the image of $T$; we
write $R\langle(\mathbb M,\mathbb N,T)\rangle$ for this quotient.
A map $(\mathbb M,\mathbb N,T)\to (\mathbb M',
\mathbb N',T')$ is defined to be an $R$-module map
$R\langle(\mathbb M,\mathbb N,T)\rangle
\to 
R\langle(\mathbb M', \mathbb N',T')\rangle$.

\begin{lemma}
The functor from schematic $R$ modules to $R$ modules
which takes $(\mathbb M,\mathbb N,T)$ to $R\langle(\mathbb M,\mathbb
N,T)\rangle$
is an equivalence of categories.
\end{lemma}

\begin{proof}
The functor is the identity on morphism sets, so it's only necessary to show 
that every $R$ module $P$ is isomorphic to one of the form $R\langle(\mathbb
M,\mathbb N,T)\rangle$.  Choose an exact sequence $Q_1\to Q_2\to P\to 0$ where $Q_1$
and $Q_2$ are free, let $\mathbb M$ and $\mathbb N$ be bases for $Q_1$ and
$Q_2$, and let $T$ be the map induced by $Q_1\to Q_2$.
\end{proof}

A {\it schematic chain complex}
$\mathbb C$ 
over $R$ is a sequence of schematic $R$ modules and 
maps, and we write $R\langle \mathbb C \rangle$ for the 
corresponding 
sequence of $R$ modules and maps.  A map $\mathbb C\to\mathbb C'$ of schematic
chain complexes is a map of $R$ chain complexes $R\langle \mathbb C \rangle\to
R\langle \mathbb C' \rangle$.


Let $h:R_1\to R_2$ be a homomorphism.  For a schematic free $R_1$ {module} 
$\mathbb
M$ we write $h_\sch\mathbb M$ for $\mathbb M$ thought of as a schematic free 
$R_2$
{module}.  There is a canonical isomorphism
\begin{equation}
\label{aa12}
R_2\langle h_\sch\mathbb M\rangle
\cong
R_2\otimes_{R_1} R_1\langle \mathbb M\rangle
\end{equation}
which takes an element $m$ of $\mathbb M$ to $1\otimes m$.
For a {map} $T:\mathbb M\to \mathbb N$ we write $h_\sch T$ for the {map}
$h_\sch\mathbb M\to h_\sch\mathbb N$ defined by the following diagram.
\[
\xymatrix{
R_2\langle h_\sch\mathbb M\rangle
\ar[r]^-{h_\sch T}
\ar[d]_\cong
&
R_2\langle h_\sch\mathbb N\rangle
\ar[d]_\cong
\\
R_2\otimes_{R_1} R_1\langle \mathbb M\rangle
\ar[r]^-{1\otimes T}
&
R_2\otimes_{R_1} R_1\langle \mathbb N\rangle
}
\]
For a schematic $R_1$ {module} 
$({\mathbb M},{\mathbb
N},T)$,  we define $h_\sch({\mathbb M},{\mathbb
N},T)$ to be the schematic $R_2$ {module} $(h_\sch{\mathbb M},h_\sch{\mathbb
N},h_\sch T)$.  The isomorphism \ref{aa12} induces a canonical isomorphism
\begin{equation}
\label{aa13}
R_2\langle h_\sch({\mathbb M},{\mathbb N},T)\rangle
\cong
R_2\otimes_{R_1} R_1\langle ({\mathbb M},{\mathbb N},T)\rangle
\end{equation}
For a {map} of $R_1$ modules $U:({\mathbb M},{\mathbb 
N},T)\to ({\mathbb M'},{\mathbb N'},T')$, we define $h_\sch U$ to be the 
{map} defined by the following diagram.
\[
\xymatrix{
R_2\langle h_\sch({\mathbb M},{\mathbb N},T)\rangle
\ar[r]^-{h_\sch U}
\ar[d]_\cong
&
R_2\langle h_\sch({\mathbb M}',{\mathbb N}',T')\rangle
\ar[d]_\cong
\\
R_2\otimes_{R_1} R_1\langle ({\mathbb M},{\mathbb N},T)\rangle
\ar[r]^-{1\otimes U}
&
R_2\otimes_{R_1} R_1\langle ({\mathbb M}',{\mathbb N}',T')\rangle
}
\]
This gives a functor $h_\sch$ from schematic $R_1$ {
modules} to schematic $R_2$ {modules}.  If $h':R_2\to R_3$ is a homomorphism 
we have
$(h'\circ h)_\sch=h'_\sch\circ h_\sch$, and thus the category of schematic $R$ {bf 
modules} is a functor of $R$.

\subsection{The ad theory $\ad^R_\Rel$}
\label{ff5}

First we translate Definition \ref{111} into the language of {schematic 
modules}:

\begin{definition}
(i)
A {\it schematic Relaxed quasi-symmetric complex of dimension $n$}
is a quadruple $(\mathbb C,\mathbb E,\gamma,\phi)$, where $\mathbb C$ is a
schematic $R$ chain complex, $\mathbb E$ is a schematic $(R^{\mathrm{op}}\otimes
R)$ chain complex, $R\langle \mathbb
C\rangle$ is homotopy 
finite,
$(R^{\mathrm{op}}\otimes R)\langle  \mathbb E\rangle$ is a
homotopy finite chain complex over $R^{\mathrm{op}}\otimes R$ with a
$\Z/2$ action for which the generator acts quasi-linearly, $\gamma$ is a
$\Z/2$ equivariant
$R^{\mathrm{op}}\otimes R$-linear
quasi-isomorphism $(R\langle \mathbb
C\rangle)^t\otimes R\langle \mathbb
C\rangle \to (R^{\mathrm{op}}\otimes R)\langle  \mathbb E\rangle$, and $\phi$ 
is an $n$-dimensional
element of
$(R\otimes_{R^{\mathrm{op}}\otimes R} (R^{\mathrm{op}}\otimes R)\langle
\mathbb E\rangle)^{\Z/2}$ (where the $\Z/2$ action is
given by Lemma \ref{aa1}).

(ii) We define a category $\A_\Rel^R$ as
follows.  The objects of $\A_\Rel^R$ are the schematic Relaxed quasi-symmetric
complexes.
A morphism $(\mathbb 
C,\mathbb E,\gamma,\phi)\to (\mathbb
C',\mathbb E',\gamma',\phi')$ 
is a pair $(f:R\langle\mathbb C\rangle\to R\langle\mathbb C'\rangle,
g:(R^{\mathrm{op}}\otimes R)\langle \mathbb E\rangle
\to
(R^{\mathrm{op}}\otimes R)\langle \mathbb E'\rangle)$, where $f$ 
is an $R$-linear chain map and $g$ 
is a
$\Z/2$ equivariant $R^{\mathrm{op}}\otimes R$-linear chain map, such that
$g\gamma=\gamma'(f\otimes f)$,
and (if $\dim \phi=\dim\phi'$)
$(1\otimes g)_*(\phi)=\phi'$.
\end{definition}

$\A_\Rel^R$ is a balanced $\Z$-graded
category, where $i$ takes $(\mathbb C,\mathbb E,\gamma, \phi)$
to $(\mathbb C,\mathbb E,\gamma, -\phi)$.

There is a morphism 
\[
\Lambda:\A_\Rel^R\to \A_{\text{Rel}}^R
\]
of $\Z$-graded
categories which takes
$(\mathbb C,\mathbb E,\gamma,\phi)$ 
to 
$(R\langle\mathbb C\rangle,(R^{\mathrm{op}}\otimes R)\langle \mathbb 
E\rangle,\gamma,\phi)$; this is an equivalence of categories.

\begin{definition}
\label{ff13}
A $K$-ad with values in $\A_\Rel^R$ is a pre $K$-ad $F$ for which $\Lambda\circ
F$ is a $K$-ad.
\end{definition}

We write $\ad_\Rel^R(K)$ for the set of $K$-ads with values in $\A_\Rel^R$.

\begin{prop}
{\rm{(i)}} $\ad_\Rel^R$ is an ad theory.

{\rm{(ii)}} $\Lambda$ induces a morphism of ad theories which is an
isomorphism on bordism groups.
\end{prop}

This is an easy consequence of Theorem \ref{ff9} and the following lemma.

\begin{lemma}
\label{ff8}
Let $K$ be a ball complex and $L$ a subcomplex.  Given a commutative diagram
\[
\xymatrix{
\Cell(L)
\ar[r]^-F
\ar[d]
&
\A
\ar[d]_I
\\
\Cell(K) 
\ar[r]^-G
&
\A'
}
\]
in which $I$ is an equivalence of categories,
there is a functor $H: \Cell(K)\to \A$ such that 
$H|_{\Cell(L)}=F$ and $I\circ H$ is naturally isomorphic to $G$.
\qed
\end{lemma}

\subsection{$\ad_\Rel^R$ is a functor of $R$.}
\label{ff51}

Let $h:R\to S$ be a homomorphism of rings with involution.

\begin{definition}
Define 
a functor
\[
h_\Rel:\A^R_\Rel \to \A^S_\Rel
\]
as follows.  For an object
$(\mathbb C,\mathbb E,\gamma,\phi)$ of $\A^R_\Rel$, 
Let
\[
h_\mathrm{Rel}(C,E,\gamma,\phi)=
(h_\sch{\mathbb C}, (h\otimes h)_\sch {\mathbb E},
\delta,\psi),
\]
where $h_\sch$ and $(h\otimes h)_\sch$ are given in Subsection \ref{ff4}, and (
letting
$C=R\langle {\mathbb C}\rangle$ and $E=
(R^{\mathrm op}\otimes R)\langle {\mathbb E}\rangle$, and
using the notation of Definition \ref{ff12} and the isomorphism of Equation
\eqref{aa13}),
$\delta$ is defined by the diagram
\[
\xymatrix{
S\langle h_\sch{\mathbb C}\rangle^t\otimes S\langle h_\sch{\mathbb C}\rangle
\ar[r]^\delta
\ar[d]_\cong
&
(S^{\mathrm{op}}\otimes S)\langle(h\otimes h)_\sch{\mathbb E}\rangle
\ar[d]_\cong
\\
(S\otimes_R C)^t\otimes (S\otimes_R C)
\ar[r]^{\gamma'}
&
(S^{\mathrm{op}}\otimes
S)\otimes_{R^{\mathrm{op}}\otimes R} E
}
\] 
and $\psi$ is the image of $\phi'$ under the isomorphism
\[
(S\otimes_{S^{\mathrm{op}}\otimes S} (S^{\mathrm{op}}\otimes
S)\otimes_{R^{\mathrm{op}}\otimes R} E)^{\Z/2}
\cong
(S\otimes_{S^{\mathrm{op}}\otimes S}(S^{\mathrm{op}}\otimes 
S)\langle(h\otimes h)_\sch {\mathbb E}\rangle)^{\Z/2}.
\]
\end{definition}

\bigskip

\begin{prop}
\label{aa8}
$h_\Rel$ takes ads to ads.
\end{prop}

\begin{proof}
Let $K$ be a ball complex and let
$F\in \ad^R_\Rel(K)$.  By Definition \ref{ff13}, we only need to show that 
$\Lambda\circ h_\Rel\circ F$ is in $\ad^R_\mathrm{Rel}(K)$.  But 
$\Lambda\circ h_\Rel\circ F$ is naturally isomorphic to 
$h_\mathrm{Rel}\circ\Lambda\circ F$, which is in $\ad^R_\mathrm{Rel}(K)$
by Definition \ref{ff13} and Proposition \ref{ff10}.
\end{proof}

The reader can check
that if $k:S\to T$ is another ring homomorphism then
$(kh)_\Rel$
is equal to
$k_\Rel h_\Rel$, so $\ad_\Rel^R$ is a functor of $R$ as required.


\providecommand{\bysame}{\leavevmode\hbox to3em{\hrulefill}\thinspace}
\providecommand{\MR}{\relax\ifhmode\unskip\space\fi MR }
\providecommand{\MRhref}[2]{%
  \href{http://www.ams.org/mathscinet-getitem?mr=#1}{#2}
}
\providecommand{\href}[2]{#2}

\end{document}